\documentclass[11pt,draft,reqno]{amsart}
\usepackage{amsmath,amsfonts,amssymb}
\usepackage[cp1251]{inputenc}
\usepackage[english]{babel}
\usepackage{mathrsfs}

\theoremstyle{theorem}
\newtheorem{theorem}{Theorem}[section]
\newtheorem{proposition}[theorem]{Proposition}
\newtheorem{lemma}[theorem]{Lemma}

\newtheorem{remark}[theorem]{Remark}
\newtheorem{corollary}[theorem]{Corollary}

\numberwithin{equation}{section}

\theoremstyle{plain}
\newtoks\thehProclaim
\newtheorem*{Proclaim}{\the\thehProclaim}

\usepackage[height = 23.0cm, a4paper, hmargin={3.0cm , 2.5cm}]{geometry}

\def\supind#1{${}^\mathrm{#1}$}

\begin{document}

\title[Homogenization of nonlocal convolution type operators]
{Homogenization of nonlocal convolution type operators:
Approximation for the resolvent \\with corrector}

\author[A.~Piatnitski, V.~Sloushch, T.~Suslina, E.~Zhizhina]{A.~Piatnitski\supind{1,2,3}, V.~Sloushch\supind{4}, T.~Suslina\supind{4}, E.~Zhizhina\supind{1,2}}

\address{\supind{1}{Institute of Information Transmission Problems of RAS, Bolshoi Karetnyi 19,  Moscow, 127051, Russia}}

\address{\supind{2}The Arctic university of Norway, campus Narvik, P.O. Box 385, 8505 Narvik, Norway}

\address{\supind{3}People's Friendship University of Russia, Miklukho-Maklaya str. 6, Moscow, 117198, Russia}

\address{\supind{4}St.~Petersburg State University,
Universitetskaya nab. 7/9, St.~Petersburg, 199034, Russia}

\email{apiatnitski@gmail.com}

\email{v.slouzh@spbu.ru}

\email{t.suslina@spbu.ru}

\email{elena.jijina@gmail.com}


\dedicatory{Dedicated to Nina Nikolaevna Ural'tseva on the occasion of her jubilee}


\begin{abstract}
In $L_2(\mathbb{R}^d)$, we consider a selfadjoint bounded operator
${\mathbb A}_\varepsilon$, $\varepsilon >0$, of the form
$$
({\mathbb A}_\varepsilon u) (\mathbf{x}) = \varepsilon^{-d-2} \int_{\mathbb{R}^d}
a((\mathbf{x} - \mathbf{y} )/ \varepsilon ) \mu(\mathbf{x} /\varepsilon, \mathbf{y} /\varepsilon) \left( u(\mathbf{x}) - u(\mathbf{y}) \right)\, d\mathbf{y}.
$$
It is assumed that  $a(\mathbf{x})$ is a nonnegative function of class  $L_1(\mathbb{R}^d)$ such that
\hbox{$a(-\mathbf{x}) = a(\mathbf{x})$} and $\mu(\mathbf{x},\mathbf{y})$ is $\mathbb{Z}^d$-periodic in each variable and such that
$\mu(\mathbf{x},\mathbf{y}) = \mu(\mathbf{y},\mathbf{x})$ and $0< \mu_- \leqslant \mu(\mathbf{x},\mathbf{y})
\leqslant \mu_+< \infty$. Moreover, it is assumed that
the moments $M_k (a)= \int_{\mathbb{R}^d} | \mathbf{x} |^k a(\mathbf{x})\,d\mathbf{x}$, $k=1,2,3,4,$ are finite.
We obtain approximation of the resolvent  $({\mathbb A}_\varepsilon + I)^{-1}$ for small $\varepsilon$
in the operator norm on  $L_2(\mathbb{R}^d)$ with error of order  $O(\varepsilon^2)$.
\end{abstract}

\maketitle

\section*{Introduction}

The paper concerns homogenization theory of periodic  operators. This is a broad field of theoretical and applied science; we  mention here only a few basic monographs \cite{BaPa, BeLP, JKO}.

We consider  the homogenization problem for a periodic nonlocal convolution type operator and estimate the rate of convergence in the operator norm.
The present paper is a continuation of the research started in  \cite{PSlSuZh};
regarding motivation, see the introduction to this article. The method is a modification of the
operator-theoretic approach.

\subsection{Operator estimates in homogenization theory. Operator-theoretic approach}\label{Sec0.1}
In a series of papers  \cite{BSu1, BSu3, BSu4} by Birman and Suslina,
an operator-theoretic  approach (a version of the spectral method) to homogenization problems for periodic differential operators in  ${\mathbb R}^d$ was suggested and developed.
By this approach,  the so-called
 \emph{operator error estimates} for a wide class of homogenization problems were found.
 Let us explain the nature of the results using the example of homogenization of an elliptic operator
$A_\varepsilon = - \operatorname{div} g(\mathbf{x} / \varepsilon) \nabla$, $\varepsilon >0$, in $L_2(\mathbb{R}^d)$. Here the matrix of coefficients $g(\mathbf{x})$ is assumed to be bounded, positive definite, and
$\mathbb{Z}^d$-periodic. In   \cite{BSu1}, it was shown that, as  $\varepsilon \to 0$, the resolvent $(A_\varepsilon +I)^{-1}$ converges in the operator norm on  $L_2(\mathbb{R}^d)$ to the resolvent of the operator  $A^0$ and the following estimate holds:
\begin{equation}
\label{BSu1}
\| (A_\varepsilon +I)^{-1} - (A^0 +I)^{-1}\|_{L_2(\mathbb{R}^d) \to L_2(\mathbb{R}^d)} \leqslant C \varepsilon.
\end{equation}
Here $A^0 = - \operatorname{div} g_{\operatorname{hom}} \nabla$ is the so-called  \emph{effective operator} with constant positive definite \emph{effective matrix} $g_{\operatorname{hom}}$.
Inequalities of this type are called operator error estimates in homogenization theory.
In \cite{BSu3}, a  more accurate approximation for the resolvent
\hbox{$(A_\varepsilon +I)^{-1}$} in the operator norm on  $L_2(\mathbb{R}^d)$ with corrector taken  into account  was obtained,
the error being of order $O(\varepsilon^2)$.
 In \cite{BSu4}, approximation of the resolvent
{$(A_\varepsilon +I)^{-1}$} in the norm of operators acting from $L_2(\mathbb{R}^d)$ to the Sobolev space  $H^1(\mathbb{R}^d)$
(also with corrector) with error of order $O(\varepsilon)$ was found.

The operator-theoretic  approach  is based on the scaling transformation, the Floquet--Bloch theory, and the analytic perturbation theory. Let us explain the method using as an example  the proof of estimate \eqref{BSu1}. By the scaling transformation,  estimate \eqref{BSu1} is reduced to the inequality
 \begin{equation}
\label{BSu2}
\| (A + \varepsilon^2 I)^{-1} - (A^0 + \varepsilon^2 I)^{-1}\|_{L_2(\mathbb{R}^d) \to L_2(\mathbb{R}^d)}
\leqslant C \varepsilon^{-1},
\end{equation}
 where $A = - \operatorname{div} g(\mathbf{x}) \nabla = \mathbf{D}^* g(\mathbf{x}) \mathbf{D}$, $\mathbf{D} = -i \nabla$. Then, by the unitary Gelfand transformation, the operator  $A$ is decomposed in the direct integral of operators  $A(\boldsymbol{\xi})$ acting in  $L_2(\Omega)$ and depending on the parameter
 $\boldsymbol{\xi} \in \widetilde{\Omega}$ (the quasimomentum). Here $\Omega = [0,1)^d$ is the cell of the lattice  $\mathbb{Z}^d$, and $\widetilde{\Omega} = [-\pi,\pi)^d$ is the cell of the dual lattice. The operator
$A(\boldsymbol{\xi})$  is given by  $A = (\mathbf{D} + \boldsymbol{\xi})^* g(\mathbf{x}) (\mathbf{D}+ \boldsymbol{\xi})$ and acts on  periodic functions.
Estimate  \eqref{BSu2} is equivalent to the similar estimate for the operators depending on  the quasimomentum :
\begin{equation*}
\| (A(\boldsymbol{\xi}) + \varepsilon^2 I)^{-1} - (A^0(\boldsymbol{\xi}) + \varepsilon^2 I)^{-1}\|_{L_2(\Omega) \to L_2(\Omega)} \leqslant C \varepsilon^{-1},
\quad \boldsymbol{\xi} \in \widetilde{\Omega}.
\end{equation*}
The main part of the proof consists in studying the operator family  $A(\boldsymbol{\xi})$,
which is an analytic operator family with compact resolvent. Thus, the methods of the analytic perturbation theory can be applied.  It turns out that the resolvent  \hbox{$(A(\boldsymbol{\xi}) + \varepsilon^2 I )^{-1}$} can be approximated in terms of the spectral characteristics of the operator at the edge of the spectrum. In particular, the effective matrix is expressed in terms of the Hesse matrix for the first eigenvalue  $\lambda_1(\boldsymbol{\xi})$ of the operator $A(\boldsymbol{\xi})$ at the point $\boldsymbol{\xi}=0$.
Therefore, the homogenization effect is a  \emph{spectral threshold effect} at the bottom of the
spectrum of an elliptic operator $A$.

A different approach to obtaining operator error estimates in homogenization problems (the so-called  ``shift method'') was suggested in the works by Zhikov and Pastukhova  (see \cite{Zh, ZhPas1}, as well as the survey  \cite{ZhPas3} and references therein).

In recent years, operator error estimates in various homogenization problems for differential operators attracted the attention of an increasing number of  researchers; many meaningful results have been obtained. A fairly detailed survey of the  current state of this area  can be found in  \cite[Sec. 0.2]{PSlSuZh} and  \cite[Introduction]{Su_UMN2023}.

We emphasize, however, that before the appearance of the authors' work \cite{PSlSuZh}
\emph{operator estimates for nonlocal homogenization problems have not been studied}.

\subsection{Statement of the problem. Main result}
In $L_2(\mathbb{R}^d)$,  we study a  nonlocal operator $\mathbb{A}_\varepsilon$ defined by
 \begin{equation}
 \label{Sus1}
 \mathbb{A}_\varepsilon u( \mathbf{x})=\frac1{\varepsilon^{d+2}}\int\limits_{\mathbb R^d} a\Big(\frac{\mathbf{x}-\mathbf{y}}{\varepsilon}\Big)\mu\Big(\frac{\mathbf{x}}{\varepsilon},\frac{\mathbf{y}}{\varepsilon}\Big)\big( u(\mathbf{x})-u(\mathbf{y})\big)\,d\mathbf{y},\ \ \mathbf{x}\in\mathbb{R}^{d},\ \ u\in L_{2}(\mathbb{R}^{d}).
 \end{equation}
 Here $\varepsilon$ is a small positive parameter. It is assumed that  $a(\mathbf{x})$ is an even nonnegative function of class $L_{1}(\mathbb{R}^{d})$ such that $\|a\|_{L_{1}}>0$; and $\mu(\mathbf{x}, \mathbf{y})$ is bounded, positive definite,  $\mathbb{Z}^{d}$-periodic in each variable, and such that
$\mu( \mathbf{x},\mathbf{y})=\mu(\mathbf{y}, \mathbf{x})$. Under these assumptions, the operator
$\mathbb{A}_{\varepsilon}$ is bounded, self-adjoint and nonnegative. It is also assumed that  the first few moments
$M_{k}=\int_{\mathbb{R}^{d}}| \mathbf{x} |^{k}a(\mathbf{x})\,d\mathbf{x}$
are finite.

The operator of the form \eqref{Sus1} arises in models of mathematical biology and population dynamics
and has been actively studied recently.
Homogenization problem for such operators was studied in  \cite{PZh}, where under the condition
$M_2(a) < \infty$  it was shown that  the resolvent  $(\mathbb{A}_{\varepsilon}+I)^{-1}$
strongly converges to the resolvent  \hbox{$(\mathbb{A}^{0}+I)^{-1}$} of the effective operator, as
$\varepsilon \to 0$. The effective operator is a second-order elliptic differential operator
$\mathbb{A}^{0}=-\operatorname{div}g^{0}\nabla$ with constant coefficients.
Thus, an interesting effect is observed in this problem: for  a bounded nonlocal operator $\mathbb{A}_{\varepsilon}$
the effective operator is  unbounded local operator $\mathbb{A}^{0}$.

For operators with a nonsymmetric kernel, similar problems were studied in \cite{PiaZhi19}, where for the corresponding parabolic equations, the homogenization result is valid in moving coordinates.
The problem in a periodically perforated domain was investigated by variational methods in \cite{BraPia21}.

In the work of the authors  \cite{PSlSuZh}, under the condition  $M_3(a) < \infty$ it was proved that the resolvent
$(\mathbb{A}_{\varepsilon}+I)^{-1}$ converges to the resolvent $(\mathbb{A}^{0}+I)^{-1}$ of the effective operator in  the operator norm on  $L_{2}(\mathbb{R}^{d})$ and the following sharp order error estimate holds:
\begin{equation*}\label{Sloushch2}
\|(\mathbb{A}_{\varepsilon}+I)^{-1}-(\mathbb{A}^{0}+I)^{-1}\|_{L_{2}(\mathbb{R}^{d})\to
L_{2}(\mathbb{R}^{d})}\leqslant C(a,\mu)\varepsilon,\ \ \varepsilon>0.
\end{equation*}

Let us describe the effective operator. As usual in homogenization theory,
in order to describe the effective matrix $g^0$, we need to consider auxiliary problems on the cell. As in Section~\ref{Sec0.1}, let
$\Omega:=[0,1)^{d}$ be the cell of the lattice  $\mathbb{Z}^{d}$.
 Suppose that a vector-valued function  $\mathbf{v}(\mathbf{x})=(v_{1}(\mathbf{x}),\dots,v_{d}(\mathbf{x}))^{t}$,
$\mathbf{x} \in\mathbb{R}^{d}$, is a  $\mathbb{Z}^{d}$-periodic solution of the problem
\begin{equation}\label{Sloushch1}
\left\{
\begin{array}{l}
\int_{\mathbb{R}^{d}}a(\mathbf{x}-\mathbf{y})\mu(\mathbf{x},\mathbf{y})(\mathbf{v}(\mathbf{x})-\mathbf{v}(\mathbf{y}))\, d\mathbf{y}
=\int_{\mathbb{R}^{d}}a(\mathbf{x}- \mathbf{y})\mu(\mathbf{x}, \mathbf{y})(\mathbf{x}-\mathbf{y})\, d\mathbf{y},\\
\\
\int_{\Omega} \mathbf{v}(\mathbf{y})\,d\mathbf{y}=0.
\end{array}
\right.
\end{equation}
Note that problem  (\ref{Sloushch1}) has a unique
$\mathbb{Z}^{d}$-periodic solution.  We show that this solution is bounded (see Section \ref{Sec5}). Let $g^{0}$ be a  $(d\times d)$-matrix with the entries  $\frac{1}{2}g_{kl}$, $k,l =1,\dots,d$, given by
\begin{equation*}
g_{kl} \!=\! \intop_{\Omega}\!d\mathbf{x} \intop_{\mathbb{R}^{d}} \!d\mathbf{y} \Big((x_{k}-y_{k})(x_{l}-y_{l}) \!-\! v_{l}(\mathbf{x})(x_{k}-y_{k}) \!-\!v_{k}(\mathbf{x})(x_{l}-y_{l})\Big)a(\mathbf{x}-\mathbf{y})\mu(\mathbf{x},\mathbf{y}),
\end{equation*}
 $k, l =1,\dots,d$. It turns out that the matrix $g^{0}$ is positive definite.
The effective operator  $\mathbb{A}^{0}=-\operatorname{div}g^{0}\nabla$ is defined on the Sobolev space  $H^{2}(\mathbb{R}^{d})$.

\emph{Our goal} is to obtain more accurate approximation of the resolvent $(\mathbb{A}_{\varepsilon}+I)^{-1}$
for small $\varepsilon$. We assume that  \hbox{$M_4(a) < \infty$}. The \emph{main result} of the paper
(Theorem \ref{teor3.2}) is the following estimate:
\begin{equation*}\label{Sloushch3}
\|(\mathbb{A}_{\varepsilon}+I)^{-1}-(\mathbb{A}^{0}+I)^{-1} - \varepsilon K_\varepsilon \|_{L_{2}(\mathbb{R}^{d})\to
L_{2}(\mathbb{R}^{d})}\leqslant C(a,\mu)\varepsilon^2,\ \ \varepsilon>0.
\end{equation*}
The corrector $K_\varepsilon$ is given by
$$
K_\varepsilon = -  \sum_{j=1}^d [v_j^\varepsilon] \partial_j (\mathbb{A}^0+ I)^{-1} +
 \sum_{j=1}^d (\mathbb{A}^0+ I)^{-1} \partial_j  [v_j^\varepsilon],
$$
where $[v_j^\varepsilon]$ is the operator of multiplication by the function $v_j(\mathbf{x} / \varepsilon)$.

\subsection{Method} 
As in \cite{PSlSuZh}, in order to find  approximation for the resolvent of the operator \eqref{Sus1},
we modify the operator-theoretic approach, which was developed by Birman and Suslina and discussed above in
Section \ref{Sec0.1}.

The first two steps, namely, the scaling transformation and the decomposition of the operator $\mathbb{A}$
into the direct integral of the operators $\mathbb{A}(\boldsymbol{\xi})$ with the help of the unitary Gelfand transform, remain the same. Here $\mathbb{A} = \mathbb{A}_{\varepsilon_0},\ \varepsilon_0 =1$.
The problem is reduced to studying the asymptotics of the resolvent  \hbox{$(\mathbb{A}(\boldsymbol{\xi})+\varepsilon^{2}I)^{-1}$}
for small $\varepsilon>0$. However, to the family of operators $\mathbb{A}( \boldsymbol{\xi})$ acting in the space
$L_2(\Omega)$ and depending on the parameter $\boldsymbol{\xi}  \in \widetilde{\Omega}$, the methods of the analytic perturbation theory are no longer applicable. In contrast with the case of differential operators,
this operator family is not analytic. Instead, we have only the finite smoothness of the family
$\mathbb{A}(\boldsymbol{\xi})$, which is granted by the assumption of finiteness of the first few moments of the coefficient  $a(\mathbf{x})$.

According to the operator-theoretic approach, to obtain an approximation for the resolvent
$(\mathbb{A} (\boldsymbol{\xi}) + \varepsilon^{2}I)^{-1}$ for small $\varepsilon$, it suffices to find the asymptotics of the operator-valued function
$\mathbb{A}( \boldsymbol{\xi})F( \boldsymbol{\xi})$, $\boldsymbol{\xi} \to0$; here $F(\boldsymbol{\xi})$ is the spectral projection of the operator $\mathbb{A}(\boldsymbol{\xi})$ corresponding to some neighborhood of zero. Traditionally, the asymptotics of the operator $\mathbb{A}(\boldsymbol{\xi})F(\boldsymbol{\xi})$ for $\boldsymbol{\xi} \to0$ was calculated using the asymptotics of the first eigenvalue  $\lambda_1(\boldsymbol{\xi})$ of
$\mathbb{A}(\boldsymbol{\xi})$. We consistently apply
an alternative way of calculating the asymptotics of the operator
 $\mathbb{A}(\boldsymbol{\xi})F(\boldsymbol{\xi})$ for $\boldsymbol{\xi}\to0$, which relies on
 integrating the resolvent  $(\mathbb{A}(\boldsymbol{\xi}) - \zeta I)^{-1}$ along a suitable contour on the complex plane. In \cite[Sec. 4.2]{PSlSuZh}, this approach was outlined as the  ``third method''.

As in the case of differential operators, it turns out that the main role in this problem is played by the spectral characteristics of the operator $\mathbb{A}$ near the lower edge of the spectrum (the so-called threshold characteristics). Thus, the homogenization effect for a nonlocal operator
${\mathbb A}_\varepsilon$ is also a threshold effect at the edge of the spectrum.

\subsection{Plan of the paper} 
The paper consists of Introduction and five sections. In Section \ref{Sec1},
we introduce the operator ${\mathbb A}$, describe  the decomposition of this operator into the direct integral
of the operators ${\mathbb A}(\boldsymbol{\xi})$,  and discuss estimates for the quadratic form of
the operator ${\mathbb A}(\boldsymbol{\xi})$. In Section \ref{Sec2}, threshold characteristics of the operator family
${\mathbb A}(\boldsymbol{\xi})$ near the bottom of the spectrum are studied. In Section \ref{Sec3},
we find approximation of the resolvent $( {\mathbb A}(\boldsymbol{\xi}) + \varepsilon^2 I)^{-1}$ for small $\varepsilon$, and, using the direct integral decomposition for $\mathbb A$, deduce approximation of the resolvent
$( {\mathbb A} + \varepsilon^2 I)^{-1}$. In Section \ref{Sec4}, the main result of the paper, which is
approximation of the resolvent $({\mathbb A}_\varepsilon + I)^{-1}$ in the operator norm on $L_2(\mathbb{R}^d)$,
is derived from the results of Section \ref{Sec3} by means of the scaling transformation. Appendix (Section \ref{Sec5})
contains the proof of boundedness of the solution of problem \eqref{Sloushch1}.

\subsection{Notation} 
The norm in the  normed  space $X$ is denoted by  $\|\cdot\|_{X}$
(or without index, if this does not lead to confusion); if
 $X$ and $Y$ are normed spaces, then the standard norm of a linear bounded operator $T:X\to Y$ is denoted by  $\|T\|_{X\to Y}$ or $\|T\|$ (without index).  The linear span of a vector system
$F\subset X$  is denoted by  $\mathcal{L}\{F\}$.
The space of bounded linear operators in the normed space  $X$ is denoted by  $\mathcal{B}(X)$.

Let $\mathfrak{H}$ and $\mathfrak{H}_*$ be complex separable Hilbert spaces.
 If $A: \mathfrak{H} \to \mathfrak{H}_*$ is a linear operator, its domain and kernel are denoted by
 $\operatorname{Dom} A$ and $\operatorname{Ker} A$. If ${\mathfrak N}$ is a subspace of $\mathfrak{H}$,
 then ${\mathfrak N}^\perp$ denotes  its orthogonal complement.
 If  $P$ is the orthogonal projection of  ${\mathfrak H}$ onto ${\mathfrak N}$, then $P^\perp$
 is the orthogonal projection onto  ${\mathfrak N}^\perp$.
For a selfadjoint operator  ${A}$ in a Hilbert space
$\mathfrak{H}$, the symbols $\sigma(A)$ and $\sigma_{\text{e}}(A)$ stand for the spectrum and
the essential spectrum of   $A$, respectively; if $\delta$ is a Borel set on  $\mathbb{R}$,
then $E_{A}(\delta)$ denotes the spectral projection of  $A$ corresponding to the set  $\delta$.

 If $\mathcal O$ is a domain  in $\mathbb{R}^d$,
 then $L_{p}({\mathcal O})$,
\hbox{$1 \leqslant p \leqslant \infty$}, stand for the standard  $L_p$-classes.
The standard inner product in  $L_{2}({\mathcal O})$ is denoted by  $(\cdot,\cdot)_{L_{2}({\mathcal O})}$
or by  $(\cdot,\cdot)$ (without index).
If $f\in L_\infty({\mathcal O})$, then the symbol  $[f]$ means the operator of multiplication  by the function
 $f(\mathbf{x})$ in the space $L_{2}({\mathcal O})$. The standard Sobolev spaces of order $s>0$ in the domain
$\mathcal O$ are denoted by  $H^s({\mathcal O})$.

The standard inner product in   $\mathbb{R}^d$ or
$\mathbb{C}^{d}$ is denoted by  $\langle\cdot,\cdot\rangle$.
Next, we use the notation $\mathbf{x} = (x_1,\dots, x_d)^{t} \in \mathbb{R}^d$, $i D_j = \partial_j = \partial / \partial x_j$, $j=1,\dots,d$; $\mathbf{D} = - i \nabla = (D_1,\dots,D_d)^t$.
By $\mathcal{S}(\mathbb{R}^{d})$ we denote the Schwarz class in $\mathbb{R}^{d}$.
The characteristic function of the set $E\subset \mathbb{R}^d$ is denoted by  $\mathbf{1}_{E}$.

\section{Nonlocal convolution type operator: \\direct integral decomposition and estimates 
\label{Sec1}}

\subsection{Operator $\mathbb{A}(a,\mu)$\label{Sec1.1}} 
Let $a\in L_{1}(\mathbb{R}^d)$ and $\mu \in L_{\infty}(\mathbb{R}^d \times \mathbb{R}^d)$.
In $L_{2}(\mathbb{R}^d)$, we define  a {\it nonlocal convolution type operator\/} $\mathbb{A} =\mathbb{A}(a,\mu)$  by
\begin{equation*}
\mathbb{A}(a,\mu) u(\mathbf{x}):=\intop_{\mathbb{R}^d} a(\mathbf{x}-\mathbf{y})\mu(\mathbf{x},\mathbf{y})(u(\mathbf{x})-u(\mathbf{y})) \,d\mathbf{y},\ \ \mathbf{x}\in\mathbb{R}^d.
\end{equation*}
The operator $\mathbb{A}$ can be represented as  $\mathbb{A} =p(\mathbf{x};a,\mu)-\mathbb{B}(a,\mu)$, where
\begin{equation*}
\begin{gathered}
p(\mathbf{x};a,\mu):= \intop_{\mathbb{R}^d} a(\mathbf{x} - \mathbf{y}) \mu(\mathbf{x},\mathbf{y}) \,d\mathbf{y},\ \ \mathbf{x} \in \mathbb{R}^d,\\
\mathbb{B}(a,\mu) u(\mathbf{x}):=\intop_{\mathbb{R}^d} a(\mathbf{x} -\mathbf{y})\mu(\mathbf{x},\mathbf{y})u(\mathbf{y})\,d\mathbf{y},\ \ \mathbf{x} \in \mathbb{R}^d.
\end{gathered}
\end{equation*}
By the Schur lemma  (see, e.\,g.,  \cite[Lemma 4.1]{PSlSuZh}), the operator  $\mathbb{B}(a,\mu)$ is bounded and $\| \mathbb{B}(a,\mu)\|_{L_2 \to L_2}\leqslant \|\mu\|_{L_\infty}\|a\|_{L_1}$. Moreover, the potential $p(\mathbf{x}) = p(\mathbf{x};a,\mu)$ satisfies
$$
\|p\|_{L_\infty}\leqslant \|\mu\|_{L_\infty}\|a\|_{L_1}.
$$
 Hence, the operator $\mathbb{A}(a,\mu):L_{2}(\mathbb{R}^d)\to L_{2}(\mathbb{R}^d)$ is bounded. We denote   $\mathbb{A}_{0}(a):=\mathbb{A}(a,\mu_{0})$,
where $\mu_{0}\equiv 1$;  $p_{0}(\mathbf{x};a):=p(\mathbf{x};a,\mu_{0})$; $\mathbb{B}_{0}(a):= \mathbb{B}(a,\mu_{0})$. Obviously,  $\mathbb{B}_{0}(a)$ is a convolution operator with the kernel  $a$, and the potential $p_{0}(\mathbf{x};a) = \int_{\mathbb{R}^d} a(\mathbf{y})\,d\mathbf{y}$ is constant.

From now on  we assume that the functions $a$ and  $\mu$ possess the following properties:
\begin{gather}
\label{e1.1}
a\in L_{1}(\mathbb{R}^d),\ \ \operatorname{mes}\{\mathbf{x}\in \mathbb{R}^d: a(\mathbf{x}) \ne 0 \} > 0,\ \ a(\mathbf{x})\geqslant 0,\ \ a(- \mathbf{x})=a(\mathbf{x}),\ \ \mathbf{x}\in \mathbb{R}^d;
\\
\label{e1.2}
0<\mu_{-}\leqslant \mu(\mathbf{x},\mathbf{y})\leqslant \mu_{+}<+\infty,\ \ \mu(\mathbf{x},\mathbf{y})= \mu(\mathbf{y},\mathbf{x}),\ \ \mathbf{x}, \mathbf{y}\in \mathbb{R}^d;
\\
\label{e1.3}
\mu(\mathbf{x}+\mathbf{m},\mathbf{y}+\mathbf{n})=\mu(\mathbf{x}, \mathbf{y}),\ \ \mathbf{x}, \mathbf{y}\in
\mathbb{R}^d,\ \ \mathbf{m}, \mathbf{n} \in \mathbb{Z}^d.
\end{gather}
We introduce the notation for the moments of the function  $a(\mathbf{x})$:
\begin{equation*}
\label{e1.4}
M_{k}(a):=\intop_{\mathbb{R}^d}| \mathbf{x}|^{k}a(\mathbf{x})\, d\mathbf{x},\ \ k \in \mathbb{N}.
\end{equation*}
Due to the condition  $0 < \int_{\mathbb{R}^d} a(\mathbf{x})\, d\mathbf{x} < \infty$,
the finiteness of the moment $M_k(a)$ automatically implies the finiteness of the moments
 $M_1(a),\dots, M_{k-1}(a)$.
In what follows, in various statements we will assume the finiteness of  $M_k(a)$
for various values of $k \leqslant 4$. In particular, in the main result (Theorem \ref{teor3.2})
it is assumed that  $M_4(a) < \infty$.

Obviously, under assumptions (\ref{e1.1}) and (\ref{e1.2}) the potential $p(\mathbf{x})$ is real-valued and the operator $\mathbb{B}(a,\mu)$ is selfadjoint. Consequently, the operator $\mathbb{A}(a,\mu)$ is also selfadjoint. Clearly, the potential $p(\mathbf{x};a,\mu)$ satisfies the estimates
\begin{equation}
\label{e1.5}
\mu_{-}\|a\|_{L_1(\mathbb{R}^d)} \leqslant p(\mathbf{x})\leqslant\mu_{+}\|a\|_{L_1(\mathbb{R}^d)},\ \ \mathbf{x} \in \mathbb{R}^d.
\end{equation}

\subsection{Estimates for the quadratic form of the operator $\mathbb{A}(a,\mu)$}
Under assumptions \eqref{e1.1} and \eqref{e1.2} the quadratic form of the operator $\mathbb{A}(a,\mu)$
admits the following representation
 (see, e.\,g., \cite{KPMZh} or \cite[Sec.~1.2]{PSlSuZh})
\begin{equation}\label{e1.6}
(\mathbb{A}(a,\mu) u,u)=\frac{1}{2} \intop_{\mathbb{R}^d} \intop_{\mathbb{R}^d} d\mathbf{x}\,d\mathbf{y}\, a(\mathbf{x}- \mathbf{y})\mu(\mathbf{x},\mathbf{y})
|u(\mathbf{x})-u(\mathbf{y})|^{2},\ \ u\in L_{2}(\mathbb{R}^d).
\end{equation}

By \eqref{e1.6}, the operator $\mathbb{A}(a,\mu)$ is nonnegative and
\begin{equation}
\label{e1.7}
\mu_{-}(\mathbb{A}_{0}(a) u,u)\leqslant (\mathbb{A}(a,\mu) u,u)\leqslant \mu_{+}(\mathbb{A}_{0}(a) u,u),\ \ u\in
L_{2}(\mathbb{R}^d).
\end{equation}
Since $\mathbb{B}_{0}(a)$ is a convolution operator, the Fourier transformation translates $\mathbb{B}_{0}(a)$
into the operator of multiplication by the function $\widehat a(\boldsymbol{\xi})$, $\boldsymbol{\xi}\in \mathbb{R}^d$, where
\begin{equation*}
\widehat a(\boldsymbol{\xi}) := \intop_{\mathbb{R}^d}e^{-i \langle \boldsymbol{\xi}, \mathbf{z} \rangle}a(\mathbf{z})\, d\mathbf{z},\ \   \boldsymbol{\xi} \in \mathbb{R}^d.
\end{equation*}
It follows that the operator  $\mathbb{A}_{0}(a)$ is unitarily equivalent  to the operator of multiplication by the function $\widehat a(\mathbf{0})-\widehat a( \boldsymbol{\xi})$. Hence, $\lambda_{0}=0$ belongs to the spectrum of  $\mathbb{A}_{0}(a)$.  Since $\mathbb{A}_{0}(a)$ is a nonnegative operator,  $\lambda_{0}$ is the spectral edge.
In view of estimates \eqref{e1.7}, the point $\lambda_{0}=0$ is also the lower edge  of the spectrum of  $\mathbb{A}(a,\mu)$.

\subsection{Direct integral decomposition for the operator $\mathbb{A}(a,\mu)$} 
Due to conditions \eqref{e1.1}--\eqref{e1.3}, the operator of multiplication by the potential $p(\mathbf{x})$ and the operator  $\mathbb{B}(a,\mu)$ (and then also $\mathbb{A}(a,\mu)$) commute with the shift operators $S_{\mathbf{n}}$
defined by
\begin{equation*}
S_{\mathbf{n}} u(\mathbf{x}) =u(\mathbf{x}+\mathbf{n}),\ \ \mathbf{x}\in \mathbb{R}^d,\ \ \mathbf{n} \in \mathbb{Z}^d.
\end{equation*}
This means that   $\mathbb{A}(a,\mu)$ and $\mathbb{B}(a,\mu)$ are periodic operators with a periodicity lattice $\mathbb{Z}^d$. Denote by  $\Omega:=[0,1)^{d}$ the cell of $\mathbb{Z}^d$ and by $\widetilde\Omega:=[-\pi,\pi)^{d}$ the cell of the dual lattice  $(2\pi\mathbb{Z})^{d}$.

Recall the definition of the Gelfand transform $\mathcal{G}$ (see, e.\,g., \cite{Sk} or \cite[Chapter~2]{BSu1}). First,  $\mathcal{G}$ is defined on the Schwarz class  ${\mathcal S}(\mathbb{R}^d)$  as follows:
\begin{equation*}
\mathcal{G}u(\boldsymbol{\xi},\mathbf{x}):=(2\pi)^{-d/2}\sum_{\mathbf{n}\in\mathbb{Z}^d} u(\mathbf{x}+\mathbf{n}) e^{-i \langle \boldsymbol{\xi}, \mathbf{x} +\mathbf{n}\rangle},\
\  \boldsymbol{\xi} \in\widetilde\Omega,\ \ \mathbf{x}\in\Omega,\ \ u\in {\mathcal S}(\mathbb{R}^d).
\end{equation*}
Next,  $\mathcal{G}$ is extended by continuity up to the unitary mapping
 $$
 \mathcal{G}: L_{2}(\mathbb{R}^d) \to \int_{\widetilde\Omega}\oplus L_{2}(\Omega)\, d\boldsymbol{\xi} =L_{2}(\widetilde\Omega\times\Omega).
 $$

Like all periodic operators,  $\mathbb{A}(a,\mu)$ and $\mathbb{B}(a,\mu)$ are decomposed into  the direct integrals with the help of the Gelfand transform (i.\,e., they are partially diagonalized):
\begin{equation}
\label{e1.8}
\mathbb{A}(a,\mu) = {\mathcal G}^* \Bigl( \int_{\widetilde\Omega} \oplus  \mathbb{A}(\boldsymbol{\xi};a,\mu) \,d\boldsymbol{\xi}\Bigr) {\mathcal G},\quad
\mathbb{B}(a,\mu) = {\mathcal G}^* \Bigl( \int_{\widetilde\Omega} \oplus  \mathbb{B}(\boldsymbol{\xi};a,\mu) \,
d\boldsymbol{\xi}\Bigr) {\mathcal G}.
\end{equation}
Here  $\mathbb{A}(\boldsymbol{\xi}) = \mathbb{A}(\boldsymbol{\xi}; a,\mu)$ and $\mathbb{B}(\boldsymbol{\xi}) = \mathbb{B}(\boldsymbol{\xi}; a,\mu)$ are selfadjoint bounded operators in
$L_{2}(\Omega)$ defined by
\begin{align}
\label{A_xi}
\mathbb{A}(\boldsymbol{\xi}; a,\mu) u (\mathbf{x}) &= p(\mathbf{x};a,\mu) u(\mathbf{x}) - \mathbb{B}(\boldsymbol{\xi}; a,\mu)u(\mathbf{x}),\ \ u\in L_{2}(\Omega),
\\
\label{B_xi}
\mathbb{B}( \boldsymbol{\xi}; a,\mu)u(\mathbf{x}) &= \intop_{\Omega}\widetilde a( \boldsymbol{\xi},\mathbf{x} - \mathbf{y})\mu(\mathbf{x},\mathbf{y}) u(\mathbf{y} )\,d\mathbf{y},\ \ u\in L_{2}(\Omega),
\end{align}
where
\begin{equation}
\label{a_tilde}
\widetilde{a}(\boldsymbol{\xi},\mathbf{z}) :=\sum_{\mathbf{n} \in\mathbb{Z}^d} a(\mathbf{z}+\mathbf{n}) e^{-i \langle \boldsymbol{\xi}, \mathbf{z} +\mathbf{n} \rangle },\ \
\boldsymbol{\xi} \in\widetilde\Omega,\ \  \mathbf{z} \in\mathbb{R}^d.
\end{equation}
Note that
\begin{equation}\label{e1.9}
p(\mathbf{x};a,\mu)=\intop_{\Omega}\widetilde a(\mathbf{0},\mathbf{x} -\mathbf{y})\mu(\mathbf{x},\mathbf{y})\,d\mathbf{y}.
\end{equation}
Let us explain how \eqref{e1.8} is understood, using the example of the first equality. Let $u \in L_2(\mathbb{R}^d)$
and $v= \mathbb{A} u$.
Then $\mathcal{G} v (\boldsymbol{\xi}, \cdot) = \mathbb{A}(\boldsymbol{\xi}) \mathcal{G} u (\boldsymbol{\xi}, \cdot)$, $\boldsymbol{\xi} \in \widetilde{\Omega}$.

The operator
$\mathbb{B}(\boldsymbol{\xi}; a,\mu)$ is compact (see \cite[Sec 4.1]{PSlSuZh}); by the Schur lemma,
its norm satisfies
\begin{equation*}
\| \mathbb{B} (\boldsymbol{\xi}; a,\mu)\|  \leqslant \mu_{+}\|a\|_{L_1(\mathbb{R}^d)},\ \
\boldsymbol{\xi} \in\widetilde\Omega.
\end{equation*}
We conclude that the essential spectrum of $\mathbb{A}(\boldsymbol{\xi}; a,\mu)$, $\boldsymbol{\xi} \in\widetilde\Omega$, coincides with
$$
\sigma_{\textrm e}(\mathbb{A}(\boldsymbol{\xi};a,\mu))=\mathrm{ess}\text{-}\mathrm{Ran}\,p(\cdot;a,\mu).
$$
Due to the compactness of $\mathbb{B}(\boldsymbol{\xi}; a,\mu)$ and the lower bound  \eqref{e1.5}, for any
$\boldsymbol{\xi} \in \widetilde\Omega$ the spectrum of $\mathbb{A}(\boldsymbol{\xi}; a,\mu)$ in the interval
 $(-\infty, \mu_{-}\|a\|_{L_1})$ is discrete.

 In \cite[Lemma~1.1]{PSlSuZh}, a convenient representation for the quadratic form
of the operator $\mathbb{A}(\boldsymbol{\xi}; a,\mu)$ was obtained.

\begin{lemma}[\cite{PSlSuZh}]
\label{lem1.1}
Under the assumptions  \eqref{e1.1}--\eqref{e1.3} we have
\begin{equation}
\label{e1.11}
(\mathbb{A}(\boldsymbol{\xi}; a,\mu)u,u) = \frac{1}{2}\intop_{\Omega}d\mathbf{x}\intop_{\mathbb{R}^d}\,d\mathbf{y}\,
a(\mathbf{x}- \mathbf{y})\mu( \mathbf{x}, \mathbf{y}) \bigl|e^{i\langle \boldsymbol{\xi}, \mathbf{x} \rangle}u(\mathbf{x})-e^{i \langle \boldsymbol{\xi}, \mathbf{y}\rangle}u(\mathbf{y}) \bigr|^{2}
\end{equation}
for $u\in L_{2}(\Omega)$ and
$\boldsymbol{\xi}\in\widetilde\Omega$.
It is assumed here that the function  $u\in L_{2}(\Omega)$ is  extended periodically to  $\mathbb{R}^d$.
\end{lemma}

\subsection{Estimates for the quadratic form of the operator  $\mathbb{A}(\boldsymbol{\xi};a,\mu)$}\label{sec1.4}
As above, it is convenient to consider the case  $\mu=\mu_{0}\equiv 1$ separately.
We denote
$\mathbb{A}_{0}(\boldsymbol{\xi}; a):= \mathbb{A}(\boldsymbol{\xi}; a,\mu_{0})$, $\mathbb{B}_{0}(\boldsymbol{\xi}; a):=\mathbb{B}(\boldsymbol{\xi}; a,\mu_{0})$. Relations \eqref{e1.2} and \eqref{e1.11} imply the estimates
\begin{equation}
\label{e1.17}
\mu_{-}(\mathbb{A}_{0}(\boldsymbol{\xi}; a)u,u)\leqslant (\mathbb{A}(\boldsymbol{\xi}; a,\mu)u,u)\leqslant \mu_{+}(\mathbb{A}_{0}(\boldsymbol{\xi}; a)u,u),\
\ u\in L_{2}(\Omega),\ \ \boldsymbol{\xi} \in\widetilde\Omega.
\end{equation}
The operators $\mathbb{A}_{0}(\boldsymbol{\xi}; a)$, $\boldsymbol{\xi} \in\widetilde\Omega$, are diagonalized by means of the unitary discrete Fourier transform ${\mathcal F}: L_{2}(\Omega)\to \ell_{2}( \mathbb{Z}^d)$ defined as follows:
\begin{gather*}
{\mathcal F} u(\mathbf{n})=\intop_{\Omega}u(\mathbf{x})e^{-2\pi i \langle \mathbf{n}, \mathbf{x} \rangle}d\mathbf{x},\ \ \mathbf{n} \in \mathbb{Z}^d,\ \ u\in L_{2}(\Omega);
\\
{\mathcal F}^{*}v(\mathbf{x})=\sum_{\mathbf{n} \in \mathbb{Z}^d}v_{\mathbf n}e^{2\pi i \langle \mathbf{n}, \mathbf{x}\rangle},\ \ \mathbf{x} \in \Omega,\ \
v=\{v_{\mathbf{n}}\}_{\mathbf{n} \in \mathbb{Z}^d}\in\ell_{2}( \mathbb{Z}^d).
\end{gather*}
We have
\begin{equation}
\label{e1.18}
\mathbb{A}_{0}(\boldsymbol{\xi}; a)= {\mathcal F}^{*} \bigl[\widehat a({\mathbf 0})-\widehat a(2\pi \mathbf{n}+\boldsymbol{\xi}) \bigr] {\mathcal F},\ \ \widehat a(\mathbf{k}):=\intop_{\mathbb{R}^d}a(\mathbf{x})e^{-i \langle \mathbf{k}, \mathbf{x} \rangle}d\mathbf{x},\ \
\mathbf{k} \in\mathbb{R}^d.
\end{equation}
Here   $[\widehat{a}(\mathbf{0})-\widehat a(2\pi \mathbf{n}+ \boldsymbol{\xi})]$ stands for the operator of multiplication by the function \hbox{$\widehat{a}(\mathbf{0})-\widehat a(2\pi \mathbf{n}+ \boldsymbol{\xi})$} in the space
$\ell_2(\mathbb{Z}^d)$.
Thus, the symbol of the operator  $\mathbb{A}_{0}(\boldsymbol{\xi}; a)$ is a sequence
$\{\widehat A_{\mathbf{n}}(\boldsymbol{\xi})\}_{\mathbf{n} \in\mathbb{Z}^d}$, where
\begin{equation*}
\begin{aligned}
\widehat A_{\mathbf{n}}(\boldsymbol{\xi})&=\widehat A(\boldsymbol{\xi} +2\pi \mathbf{n})=\widehat a(\mathbf{0})-\widehat a(2\pi \mathbf{n} + \boldsymbol{\xi})
\\
&=\intop_{\mathbb{R}^d} \bigl(1-\cos( \langle \mathbf{z}, \boldsymbol{\xi}+2\pi \mathbf{n} \rangle) \bigr)a(\mathbf{z})\,d\mathbf{z}, \ \ \mathbf{n} \in \mathbb{Z}^d,\ \ \boldsymbol{\xi}\in\widetilde\Omega.
\end{aligned}
\end{equation*}
Here we have used the fact that the integral  $\int_{\mathbb{R}^d}\sin ( \langle \mathbf{z}, \boldsymbol{\xi} +2\pi \mathbf{n} \rangle )a(\mathbf{z})\, d\mathbf{z}$ is equal to zero because   $a(\mathbf{z})= a(- \mathbf{z})$.

From the described diagonalization it easily follows that
$\operatorname{Ker} \mathbb{A}_{0}(\mathbf{0};a)=\mathcal{L}\{\mathbf{1}_{\Omega}\}$.
Therefore, by \eqref{e1.17},
$\operatorname{Ker} \mathbb{A}(\mathbf{0}; a,\mu)=\mathcal{L}\{\mathbf{1}_{\Omega}\}$.
We arrive at the following statement (see \cite[Lemma 1.2]{PSlSuZh}).

\begin{lemma}[\cite{PSlSuZh}]
\label{lemma1.2}
Under conditions \eqref{e1.1}--\eqref{e1.3} the number $\lambda_{0}=0$ is a simple eigenvalue
of the operator $\mathbb{A}(\mathbf{0};a,\mu)$. We have
$\operatorname{Ker}\mathbb{A}(\mathbf{0};a,\mu)=\mathcal{L}\{\mathbf{1}_{\Omega}\}$.
\end{lemma}

In order to estimate the quadratic form of the operator $\mathbb{A}( \boldsymbol{\xi}; a,\mu)$,
we carry out the detailed analysis of the symbol of the operator $\mathbb{A}_{0}( \boldsymbol{\xi}; a)$,
$\boldsymbol{\xi}\in\widetilde\Omega$.
Under condition \eqref{e1.1} the function
\begin{equation*}
\label{e1.17a}
\widehat A(\mathbf{y}):=\intop_{\mathbb{R}^d} \bigl(1-\cos(\langle \mathbf{z}, \mathbf{y} \rangle) \bigr)a(\mathbf{z})\,d\mathbf{z}
= 2\intop_{\mathbb{R}^d}\sin^{2}\left(\frac{\langle \mathbf{z}, \mathbf{y} \rangle}{2}\right)
a(\mathbf{z})\,d\mathbf{z},\ \ \mathbf{y} \in \mathbb{R}^d,
\end{equation*}
depends on  $\mathbf{y} \in \mathbb{R}^d$ continuously and, by the Riemann--Lebesgue lemma, converges to
\hbox{$\|a\|_{L_1}>0$} as $| \mathbf{y} | \to\infty$. Furthermore, it is easy to check that $\widehat A(\mathbf{y}) >0$
for $\mathbf{y} \not=0$.  Consequently, we have
\begin{equation}
\label{e1.19}
\min_{| \mathbf{y} | \geqslant r} \widehat A(\mathbf{y})=:\mathcal{C}_{r}(a)>0,\ \ r>0.
\end{equation}
Since $|  \boldsymbol{\xi}  + 2\pi \mathbf{n} |\geqslant \pi$ if  $ \boldsymbol{\xi} \in \widetilde\Omega$ and
$\mathbf{n} \in \mathbb{Z}^d\setminus\{ \mathbf{0} \}$, then
\begin{equation}
\label{e1.22}
\widehat A( \boldsymbol{\xi} +2\pi \mathbf{n})\geqslant \mathcal{C}_{\pi}(a),\ \  \boldsymbol{\xi} \in \widetilde\Omega,\ \
\mathbf{n} \in \mathbb{Z}^d\setminus\{\mathbf{0}\}.
\end{equation}

Next, under the condition $M_2(a) < \infty$ the function
\begin{equation*}
\label{e1.18a}
\intop_{\mathbb{R}^d}a(\mathbf{z}) \langle \mathbf{z}, \mathbf{y} \rangle^{2}\, d\mathbf{z} =:M_{a}(\mathbf{y})
\end{equation*}
is continuous in  $\mathbf{y} \in \mathbb{R}^d$ and not equal to zero if $\mathbf{y} \ne \mathbf{0}$. It follows that
\begin{equation}
\label{e1.20}
\min_{|\boldsymbol{\theta}|=1}M_{a}(\boldsymbol{\theta}) =:\mathcal{M}(a)>0.
\end{equation}

The following statement was obtained in  \cite[Lemma 1.3]{PSlSuZh}.

\begin{lemma}[\cite{PSlSuZh}]
\label{lemma1.3}
Suppose that conditions  \eqref{e1.1}--\eqref{e1.3} are fulfilled and  $M_3(a)<\infty$.
Then
\begin{equation*}
\label{e1.21}
\widehat A(\boldsymbol{\xi} +2\pi \mathbf{n})\geqslant C(a)| \boldsymbol{\xi} |^{2},\ \ \boldsymbol{\xi} \in\widetilde\Omega,\ \ \mathbf{n} \in\mathbb{Z}^d,
\end{equation*}
where
\begin{equation}
\label{C(a)}
C(a):=\min \Bigl\{\frac{1}{4}\mathcal{M}(a),\mathcal{C}_{r(a)}(a)\pi^{-2}d^{-1},\mathcal{C}_{\pi}(a)\pi^{-2}d^{-1}
\Bigr\}>0,
\end{equation}
the quantity $\mathcal{C}_{r}(a)$, $r>0$, is defined by \eqref{e1.19}, the constant  $\mathcal{M}(a)$ is given by
\eqref{e1.20}, and
$$r(a):=\frac{3 \mathcal{M}(a)}{2 M_{3}(a)}.$$
\end{lemma}

Under conditions \eqref{e1.1}--\eqref{e1.3}, relations \eqref{e1.17}, \eqref{e1.18}, and \eqref{e1.22}
imply that
\begin{equation}\label{e1.31}
(\mathbb{A}( \boldsymbol{\xi}; a,\mu)u,u)\geqslant \mu_{-}\mathcal{C}_{\pi}(a) \|u\|^{2}_{L_2(\Omega)},\ \
u\in L_{2}(\Omega)\ominus\mathcal{L}\{\mathbf{1}_{\Omega}\},\ \
 \boldsymbol{\xi} \in \widetilde\Omega.
\end{equation}

Relations  \eqref{e1.17}, \eqref{e1.18}, and Lemma \ref{lemma1.3} imply the following statement;  see \cite[Proposition 1.4]{PSlSuZh}.

\begin{proposition}[\cite{PSlSuZh}]
\label{prop1.4}
Suppose that conditions  \eqref{e1.1}--\eqref{e1.3} are satisfied and $M_3(a)<\infty$.
Then
\begin{equation}\label{e1.30}
(\mathbb{A}( \boldsymbol{\xi};a,\mu)u,u)\geqslant \mu_{-}C(a)| \boldsymbol{\xi} |^{2}\|u\|_{L_2(\Omega)}^{2},\ \ u\in
L_{2}(\Omega),\ \ \boldsymbol{\xi} \in\widetilde\Omega.
\end{equation}
\end{proposition}

We also need the following estimate which is proved by the Schur lemma under conditions  \eqref{e1.1}--\eqref{e1.3}  and condition $M_1(a) < \infty$:
\begin{equation}
\label{e1.32}
\| \mathbb{A}(\boldsymbol{\xi};a,\mu) - \mathbb{A}(\boldsymbol{\eta};a,\mu)\|\leqslant \mu_{+} M_{1}(a)
|\boldsymbol{\xi} - \boldsymbol{\eta} |,\ \
\boldsymbol{\xi}, \boldsymbol{\eta} \in\widetilde\Omega.
\end{equation}

\section{Threshold characteristics of nonlocal convolution type operator  near the lower edge of the spectrum }\label{Sec2}

\subsection{The spectral edge of the operator $\mathbb{A}(\boldsymbol{\xi};a,\mu)$}
According to Lemma \ref{lemma1.2}, under conditions \eqref{e1.1}--\eqref{e1.3} the lower edge of the spectrum  of the operator
$\mathbb{A}(\mathbf{0};a,\mu)$ is an isolated simple eigenvalue  $\lambda_{0}=0$.
Let  $d_{0}:=d_{0}(a,\mu)$ be the distance from the point  $\lambda_{0}$ to the rest of the spectrum of
$\mathbb{A}(\mathbf{0};a,\mu)$. Under conditions \eqref{e1.1}--\eqref{e1.3}, estimate (\ref{e1.31})
implies that
\begin{equation}
\label{d0}
d_{0}(a,\mu)\geqslant\mu_{-}\mathcal{C}_{\pi}(a).
\end{equation}
Let
\begin{equation}
\label{delta0}
\delta_{0}(a,\mu):=  \frac{\mu_{-}\mathcal{C}_{\pi}(a)}{3 M_{1}(a) \mu_{+}}.
\end{equation}
It is easily seen that  $\delta_{0}(a,\mu) < \pi$, and therefore the ball  $|\boldsymbol{\xi}| \leqslant \delta_{0}(a,\mu)$ lies inside
 $\widetilde\Omega$.
Under conditions  \eqref{e1.1}--\eqref{e1.3} and $M_1(a)<\infty$, applying the perturbation theory arguments, we deduce from estimate \eqref{e1.32} that
\begin{equation*}
\operatorname{rank}E_{\mathbb{A}(\boldsymbol{\xi};a,\mu)}[0,d_{0}/3]=1,\ \
\sigma(\mathbb{A}( \boldsymbol{\xi};a,\mu))\cap(d_{0}/3;2d_{0}/3)=\varnothing, \quad
 |\boldsymbol{\xi}|\leqslant \delta_0(a, \mu).
\end{equation*}
We arrive at the following statement.

\begin{proposition}
\label{prop2.1}
Suppose that conditions  \eqref{e1.1}--\eqref{e1.3} are satisfied and {$M_1(a) < \infty$}.
Let  $\delta_0(a,\mu)$ be given by  \eqref{delta0}. Then for
$| \boldsymbol{\xi} |\leqslant \delta_{0}(a,\mu)$ the spectrum of the operator  $\mathbb{A}(\boldsymbol{\xi}) = \mathbb{A}(\boldsymbol{\xi};a,\mu)$ on the interval  $[0,d_{0}/3]$ consists of just one simple eigenvalue\/{\rm ;}
while the interval  $(d_{0}/3,2d_{0}/3)$ is free of the spectrum of  $\mathbb{A}(\boldsymbol{\xi};a,\mu)$.
\end{proposition}

Under conditions \eqref{e1.1}--\eqref{e1.3} and $M_k(a) < \infty$ the operator-valued function
$\mathbb{A}(\cdot;a,\mu)$ is $k$ times continuously differentiable in the operator norm on  $\mathcal{B}(L_{2}(\Omega))$. We have
\begin{equation}
\label{e2.1}
\begin{aligned}
&\partial^{\alpha}\mathbb{A}( \boldsymbol{\xi}; a,\mu)u(\mathbf{x}) = \intop_{\Omega}\widetilde
a_{\alpha}(\boldsymbol{\xi}, \mathbf{x}- \mathbf{y})\mu(\mathbf{x},\mathbf{y})u(\mathbf{y})\, d\mathbf{y},\ \ \mathbf{x} \in\Omega,\ \ u\in L_{2}(\Omega);
\\
&\widetilde a_{\alpha}(\boldsymbol{\xi},\mathbf{z})
=(-1)(-i)^{|\alpha|}\sum_{\mathbf{n} \in\mathbb{Z}^d}(\mathbf{z}+\mathbf{n})^{\alpha}a(\mathbf{z} + \mathbf{n})e^{-i \langle \boldsymbol{\xi}, \mathbf{z}+ \mathbf{n} \rangle},\
\ \alpha\in \mathbb{Z}^d_{+},\ \ |\alpha|\leqslant k.
\end{aligned}
\end{equation}

\begin{lemma}
Suppose that conditions \eqref{e1.1}--\eqref{e1.3} are satisfied.

\noindent
$1^\circ$. If $M_1(a) < \infty$, then
\begin{equation}
\label{2.3}
\| \mathbb{A}(\boldsymbol{\xi}) - \mathbb{A}(\mathbf{0})\| \leqslant \mu_+ M_1(a) |\boldsymbol{\xi}|,\quad
| \boldsymbol{\xi}| \leqslant \delta_0(a,\mu).
\end{equation}

\noindent
$2^\circ$. If $M_2(a) < \infty$, then
\begin{align}
\label{2.4}
&\mathbb{A}(\boldsymbol{\xi})= \mathbb{A}(\mathbf{0})+ [\Delta_1 \mathbb{A}](\boldsymbol{\xi} ) +
\mathbb{K}_1(\boldsymbol{\xi}),\quad  [\Delta_1 \mathbb{A}](\boldsymbol{\xi} ) :=\sum_{j=1}^{d}\partial_{j}
 \mathbb{A}(\mathbf{0})\xi_{j},
\\
\label{2.5}
&\| \mathbb{K}_1(\boldsymbol{\xi}) \| \leqslant \frac{1}{2} \mu_+ M_2(a) |\boldsymbol{\xi}|^2,\quad
|\boldsymbol{\xi}| \leqslant \delta_0(a,\mu).
\end{align}

\noindent
$3^\circ$. If $M_3(a) < \infty$, then
\begin{equation}
\label{2.6}
\begin{aligned}
&\mathbb{A}(\boldsymbol{\xi})= \mathbb{A}(\mathbf{0})+ [\Delta_1 \mathbb{A}](\boldsymbol{\xi} ) + [\Delta_2 \mathbb{A}](\boldsymbol{\xi} )
+ \mathbb{K}_2(\boldsymbol{\xi}),
\\
&[\Delta_2\mathbb{A}](\boldsymbol{\xi} ) := \frac{1}{2}\sum_{k,l=1}^{d}\partial_{k}\partial_{l} \mathbb{A}(\mathbf{0})\xi_{k}\xi_{l},
\end{aligned}
\end{equation}
and
\begin{equation}
\label{2.7}
\| \mathbb{K}_2(\boldsymbol{\xi}) \| \leqslant \frac{1}{6} \mu_+ M_3(a) | \boldsymbol{\xi}|^3,\quad
| \boldsymbol{\xi} | \leqslant \delta_0(a,\mu).
\end{equation}

\noindent
$4^\circ$. If $M_4(a) < \infty$, then
\begin{equation}
\label{2.8}
\begin{aligned}
& \mathbb{A}(\boldsymbol{\xi}) =\mathbb{A}(\mathbf{0})+ [\Delta_1 \mathbb{A}](\boldsymbol{\xi} ) + [\Delta_2 \mathbb{A}](\boldsymbol{\xi} )+
 [\Delta_3 \mathbb{A}](\boldsymbol{\xi} ) +\mathbb{K}_3(\boldsymbol{\xi}),
\\
&[\Delta_3 \mathbb{A}](\boldsymbol{\xi} ) :=  \frac{1}{6}\sum_{j,k,l=1}^{d}\partial_{j}\partial_{k} \partial_{l} \mathbb{A}(\mathbf{0})\xi_{j}\xi_{k} \xi_l,
\end{aligned}
\end{equation}
and
\begin{equation}
\label{2.9}
\| \mathbb{K}_3(\boldsymbol{\xi}) \| \leqslant \frac{1}{24} \mu_+ M_4(a) | \boldsymbol{\xi} |^4,\quad | \boldsymbol{\xi} | \leqslant \delta_0(a,\mu).
\end{equation}
\end{lemma}

\begin{proof}
The proof of all four statements  is based on the Hadamard lemma and the Schur lemma.
For example, let us prove statement $4^\circ$.

By the Hadamard lemma, the operator-valued function
$\mathbb{A}(\boldsymbol{\xi})$ admits expansion  \eqref{2.8}, where
\begin{equation*}
\mathbb{K}_3(\boldsymbol{\xi}) =
\sum_{j=1}^{d}\sum_{k=1}^{d}\sum_{l=1}^{d} \sum_{m=1}^d
\xi_{j} \xi_{k} \xi_{l} \xi_m
\intop_{0}^{1}ds_{1}s_{1}^{3}\intop_{0}^{1}ds_{2}\,s_{2}^2
\intop_{0}^{1}ds_{3} \,s_3 \intop_{0}^{1}ds_{4}\,\partial_{j}\partial_{k}\partial_{l} \partial_m {\mathbb A}(s_{1}s_{2}s_{3} s_4 \boldsymbol{\xi}).
\end{equation*}
By \eqref{e2.1}, it follows that $\mathbb{K}_3(\boldsymbol{\xi}) : L_2(\Omega) \to L_2(\Omega)$ is an integral
operator with the kernel
\begin{multline*}
{\mathcal K}(\boldsymbol{\xi}; \mathbf{x}, \mathbf{y})
= - \mu(\mathbf{x},\mathbf{y})\! \sum_{\mathbf{n} \in \mathbb{Z}^d} \langle \mathbf{x} - \mathbf{y} +\mathbf{n}, \boldsymbol{\xi} \rangle^4 a(\mathbf{x} - \mathbf{y}+ \mathbf{n})
\\
\times
\intop_{0}^{1}\!ds_{1}s_{1}^{3}\intop_{0}^{1}\!ds_{2}\,s_{2}^2
\intop_{0}^{1}\!ds_{3} \,s_3 \intop_{0}^{1}\!ds_{4}\,   e^{- i s_{1}s_{2}s_{3} s_4 \langle \boldsymbol{\xi}, \mathbf{x} - \mathbf{y} +\mathbf{n} \rangle}.
\end{multline*}
Hence,
$$
\intop_\Omega |{\mathcal K}( \boldsymbol{\xi};\mathbf{x},\mathbf{y})| \,d\mathbf{y}
\leqslant \frac{1}{24} \mu_+ \intop_{\mathbb{R}^d} a(\mathbf{x} - \mathbf{y})\langle \mathbf{x} - \mathbf{y}, \boldsymbol{\xi} \rangle^4 \, d\mathbf{y}
\leqslant \frac{1}{24} \mu_+ M_4(a) | \boldsymbol{\xi} |^4.
$$
The quantity $\int_\Omega |{\mathcal K}(\boldsymbol{\xi}; \mathbf{x}, \mathbf{y})| \,d\mathbf{x}$
satisfies the same  estimate.
By the Schur lemma, we obtain estimate \eqref{2.9}.
\end{proof}

 Using the Schur lemma, it is easy to justify the following estimates:
\begin{align}\label{e2.3}
\left\| [\Delta_1\mathbb{A}](\boldsymbol{\xi} ) \right\| &\leqslant \mu_{+} M_{1}(a) | \boldsymbol{\xi} |;
\\
\label{e2.4}
\left\| [\Delta_2 \mathbb{A}](\boldsymbol{\xi} ) \right\|
&\leqslant \frac{1}{2}\mu_{+} M_{2}(a) | \boldsymbol{\xi} |^2.
\end{align}

\subsection{Threshold approximations}
Let $F(\boldsymbol{\xi})$ be the spectral projetion of the operator $\mathbb{A}(\boldsymbol{\xi};a,\mu)$
that corresponds to the interval $[0,d_{0}/3]$. The symbol $\mathfrak{N}$ stands for  the kernel
$\operatorname{Ker} \mathbb{A}(\mathbf{0};a,\mu)=\mathcal{L}\{\mathbf{1}_{\Omega}\}$;
by $P$ we denote the orthogonal projection onto $\mathfrak{N}$; then $P = (\cdot, \mathbf{1}_\Omega)\mathbf{1}_\Omega$.
Let $\Gamma$ be a contour on the complex plane that is equidistant to the interval $[0,d_{0}/3]$
and passes through the middle point of the interval $(d_{0}/3,2d_{0}/3)$.
By the Riesz formula, the following representations are valid:
\begin{align}
\label{e2.6}
F(\boldsymbol{\xi})  &= - \frac{1}{2\pi i}\ointop_{\Gamma}(\mathbb{A}(\boldsymbol{\xi})-\zeta I)^{-1}\, d\zeta,\ \
|\boldsymbol{\xi} |\leqslant\delta_{0}(a,\mu),
\\
\label{e2.6a}
\mathbb{A}(\boldsymbol{\xi}) F(\boldsymbol{\xi}) &= - \frac{1}{2\pi i}\ointop_{\Gamma}(\mathbb{A}(\boldsymbol{\xi})-\zeta I)^{-1}\zeta\,d\zeta,\ \ | \boldsymbol{\xi} |\leqslant \delta_{0}(a,\mu);
\end{align}
here we integrate along the contour counterclockwise.

Our goal is to obtain an approximation for the operator $F(\boldsymbol{\xi})$ with an error
$O(|\boldsymbol{\xi}|^2)$ and an approximation for the operator $\mathbb{A}(\boldsymbol{\xi}) F(\boldsymbol{\xi})$
with an error  $O(| \boldsymbol{\xi} |^4)$.
Earlier in  \cite{PSlSuZh}, less accurate approximations for
 $F(\boldsymbol{\xi})$ and $\mathbb{A}(\boldsymbol{\xi}) F(\boldsymbol{\xi})$ were obtained
(with errors $O(| \boldsymbol{\xi}|)$ and $O(|\boldsymbol{\xi}|^3)$ respectively).
It is convenient for us to reproduce these results here  (see Propositions \ref{prop2.3} and \ref{prop2.6} below);
when calculating approximations, we consistently apply the method of integrating the resolvent along the contour (called the  ``third method'' in \cite[Sec.~4]{PSlSuZh}).

\begin{proposition}
\label{prop2.3}
Suppose that conditions  \eqref{e1.1}--\eqref{e1.3} are satisfied and
$M_1(a) < \infty$. Then
\begin{equation}
\label{F-P}
\| F(\boldsymbol{\xi}) - P \| \leqslant C_1(a,\mu) | \boldsymbol{\xi} |, \quad  |\boldsymbol{\xi}| \leqslant \delta_0(a,\mu).
\end{equation}
The constant $C_1(a,\mu)$ is defined below in \eqref{C1} and depends only on $\mu_-$, $\mu_+$,  $\mathcal{C}_\pi(a)$, $M_1(a)$.
\end{proposition}

\begin{proof}
Denote
\begin{align*}
R( \boldsymbol{\xi},\zeta)&:=(\mathbb{A}(\boldsymbol{\xi})-\zeta I)^{-1},\ \
| \boldsymbol{\xi} |\leqslant \delta_{0}(a,\mu),\ \ \zeta\in\Gamma;
\\
R_{0}(\zeta)&:=R(\mathbf{0},\zeta),\ \ \zeta\in\Gamma;
\\
\Delta \mathbb{A}(\boldsymbol{\xi})&:=\mathbb{A}(\boldsymbol{\xi})- \mathbb{A}(\mathbf{0}),\ \
| \boldsymbol{\xi} | \leqslant \delta_{0}(a,\mu).
\end{align*}

By the Riesz formula \eqref{e2.6}, the  difference $F(\boldsymbol{\xi})-P$ can be represented as
\begin{equation}\label{e2.30}
F(\boldsymbol{\xi})-P= - \frac{1}{2\pi i}\ointop_{\Gamma} \left( R(\boldsymbol{\xi},\zeta) - R_0(\zeta)\right) \, d\zeta,
\ \ | \boldsymbol{\xi}|\leqslant \delta_{0}(a,\mu).
\end{equation}
We apply the resolvent identity
\begin{equation}\label{e2.31}
R(\boldsymbol{\xi},\zeta) = R_0(\zeta)
- R(\boldsymbol{\xi},\zeta) \Delta\mathbb{A}(\boldsymbol{\xi}) R_0(\zeta), \quad
| \boldsymbol{\xi} |\leqslant \delta_{0}(a,\mu),\ \ \zeta\in\Gamma.
\end{equation}
The length of the contour  $\Gamma$ is equal to
$\frac{\pi+2}{3}d_{0}$, and both resolvents on the contour
$\Gamma$ satisfy the estimates
\begin{equation}\label{e2.32}
\| R(\boldsymbol{\xi},\zeta)\|\leqslant 6d_{0}^{-1},\ \
\| R_0 (\zeta)\|\leqslant 6d_{0}^{-1}, \ \ |\boldsymbol{\xi}| \leqslant \delta_{0}(a,\mu),\quad \zeta \in \Gamma.
\end{equation}
Combining \eqref{2.3} and
\eqref{e2.30}--\eqref{e2.32}, we arrive at estimate  \eqref{F-P} with the constant
\begin{equation}
\label{C1}
C_{1}(a,\mu):= \frac{6(\pi+2)\mu_{+}M_{1}(a)}{\pi d_0}.
\end{equation}
\end{proof}

\begin{proposition}
\label{prop2.4}
Suppose that conditions  \eqref{e1.1}--\eqref{e1.3} are satisfied and  {$M_2(a) < \infty$}. Then
\begin{equation}
\label{F= P+ F1}
 F(\boldsymbol{\xi} ) =  P + [F]_1( \boldsymbol{\xi}) + \Phi(\boldsymbol{\xi}),
 \quad  [F]_1(\boldsymbol{\xi}) : = \sum_{j=1}^d F_j \xi_j,
\end{equation}
and
\begin{equation}
\label{F-P-O(xi)}
\left\| \Phi(\boldsymbol{\xi} ) \right\| \leqslant C_2(a,\mu) |\boldsymbol{\xi} |^2, \quad
| \boldsymbol{\xi} | \leqslant \delta_0(a,\mu).
\end{equation}
The operators $F_j$ are given by
\begin{equation}
\label{F_j=prop}
 F_j =  - P \partial_{j}\mathbb{A}(\mathbf{0})  P^\perp \mathbb{A}(\mathbf{0})^{-1}P^\perp   -
 P^\perp \mathbb{A}(\mathbf{0})^{-1}P^\perp \partial_{j}\mathbb{A}(\mathbf{0}) P, \quad j=1,\dots,d.
\end{equation}
Here   $\mathbb{A}(\mathbf{0})^{-1}$ is understood as the inverse operator to  $ \mathbb{A}(\mathbf{0})
\vert_{\mathfrak{N}^\perp} : {\mathfrak{N}^\perp} \to {\mathfrak{N}^\perp}$.
The constant $C_2(a,\mu)$ is defined below in  \eqref{C2} and depends only on $\mu_-$, $\mu_+$,  $\mathcal{C}_\pi(a)$, $M_1(a)$, $M_2(a)$.
The operator $[F]_1(\boldsymbol{\xi})$  satisfies the estimate
\begin{equation}
\label{[F]1j=le}
\left\|[F]_1(\boldsymbol{\xi}) \right\| \leqslant C_1(a,\mu) | \boldsymbol{\xi} |,\ \ \boldsymbol{\xi} \in \widetilde{\Omega}.
\end{equation}
\end{proposition}

\begin{proof}
Iterating the resolvent identity  \eqref{e2.31}, we obtain
\begin{equation}
\label{res_iden2}
R(\boldsymbol{\xi},\zeta) = R_0(\zeta)
- R_0(\zeta) \Delta\mathbb{A}(\boldsymbol{\xi}) R_0(\zeta) + R(\boldsymbol{\xi},\zeta) \Delta\mathbb{A}(\boldsymbol{\xi}) R_0(\zeta) \Delta\mathbb{A}(\boldsymbol{\xi}) R_0(\zeta)
\end{equation}
for $| \boldsymbol{\xi}| \leqslant \delta_{0}(a,\mu)$ and $\zeta\in\Gamma$.
By \eqref{2.3} and \eqref{e2.32}, the operator $$Z_1(\boldsymbol{\xi},\zeta) := R(\boldsymbol{\xi},\zeta) \Delta\mathbb{A}(\boldsymbol{\xi}) R_0(\zeta) \Delta\mathbb{A}(\boldsymbol{\xi}) R_0(\zeta)$$ satisfies
\begin{equation}
\label{res_est1}
\| Z_1( \boldsymbol{\xi},\zeta) \| \leqslant (6 d_0^{-1})^3 \mu^2_+ M_1(a)^2 | \boldsymbol{\xi} |^2,
\quad
| \boldsymbol{\xi} |\leqslant \delta_{0}(a,\mu),\ \ \zeta\in\Gamma.
\end{equation}
Next, from \eqref{2.4} and \eqref{res_iden2} it follows that
\begin{equation}
\label{res_iden3a}
R( \boldsymbol{\xi},\zeta) = R_0(\zeta)
- R_0(\zeta) [\Delta_1 \mathbb{A}](\boldsymbol{\xi} )  R_0(\zeta)
+ Z_1(\boldsymbol{\xi},\zeta) + Z_2(\boldsymbol{\xi},\zeta)
\end{equation}
for $| \boldsymbol{\xi}| \leqslant \delta_{0}(a,\mu)$ and $\zeta\in\Gamma$,
where $Z_2(\boldsymbol{\xi},\zeta) = - R_0(\zeta) \mathbb{K}_1(\boldsymbol{\xi}) R_0(\zeta)$.
By \eqref{2.5} and \eqref{e2.32}, we have
\begin{equation}
\label{res_est2}
\| Z_2(\boldsymbol{\xi},\zeta) \| \leqslant 18 d_0^{-2} \mu_+ M_2(a) | \boldsymbol{\xi} |^2,
\quad
| \boldsymbol{\xi}|\leqslant \delta_{0}(a,\mu),\ \ \zeta\in\Gamma.
\end{equation}
Applying the Riesz formula \eqref{e2.6} and representation  \eqref{res_iden3a}, we obtain
\begin{equation}
\label{F=}
F(\boldsymbol{\xi}) = P + \sum_{j=1}^d F_j {\xi}_j + \Phi(\boldsymbol{\xi}), \quad
| \boldsymbol{\xi} |\leqslant \delta_{0}(a,\mu),
\end{equation}
where
\begin{align}
\label{Fj=}
F_j &=  \frac{1}{2\pi i}\ointop_{\Gamma} R_0(\zeta) \partial_j \mathbb{A}(\mathbf{0}) R_0(\zeta)  \, d\zeta,
\\
\label{Phi=}
\Phi(\boldsymbol{\xi}) &= - \frac{1}{2\pi i}\ointop_{\Gamma} (Z_1(\boldsymbol{\xi},\zeta) + Z_2(\boldsymbol{\xi},\zeta))  \, d\zeta.
\end{align}
Relations  \eqref{res_est1} and \eqref{res_est2} yield  the following estimate for the operator \eqref{Phi=}:
\begin{equation}
\label{res_est3}
 \| \Phi(\boldsymbol{\xi}) \| \leqslant C_2(a,\mu) | \boldsymbol{\xi} |^2, \quad
| \boldsymbol{\xi} |\leqslant\delta_{0}(a,\mu),
\end{equation}
\begin{equation}
\label{C2}
C_{2}(a,\mu):=\frac{(\pi+2)}{\pi} \left( \frac{36 \mu_{+}^2 M_{1}(a)^2}{d_0^2} +
\frac{3 \mu_{+} M_{2}(a)}{d_0}\right).
\end{equation}
Now representation  \eqref{F= P+ F1} and estimate \eqref{F-P-O(xi)} follow from  \eqref{F=} and \eqref{res_est3}.

To calculate the integral in  \eqref{Fj=}, we use the following representation of the resolvent of
$\mathbb{A}(\mathbf{0})$:
\begin{equation}\label{4.10}
R_{0}(\zeta)=R_{0}(\zeta)P+R_{0}(\zeta)P^{\bot}=-\frac{1}{\zeta}P+R_{0}(\zeta)P^{\bot},\ \ \zeta\in\Gamma.
\end{equation}
Substituting  \eqref{4.10} into the contour integral  \eqref{Fj=} and taking into account the fact that
the operator-valued function  $R_{0}^{\bot}(\zeta):=R_{0}(\zeta)P^{\bot}$ is holomorphic inside the contour
$\Gamma$, we obtain
\begin{multline*}
F_{j}=  \frac{1}{2\pi i}\oint_{\Gamma} \Bigl( -\frac{1}{\zeta}P+R_{0}^{\bot}(\zeta)\Bigr) \partial_{j}\mathbb{A}(\mathbf{0})
\Bigl(-\frac{1}{\zeta}P+R_{0}^{\bot}(\zeta)\Bigr) \,d\zeta
\\
=
 - \frac{1}{2\pi i}\oint_{\Gamma} \frac{1}{\zeta} \left( P \partial_{j}\mathbb{A}(\mathbf{0}) R_{0}^{\bot}(\zeta)
 +R_{0}^{\bot}(\zeta)  \partial_{j}\mathbb{A}(\mathbf{0}) P \right) \,d\zeta
\\
= - P \partial_{j}\mathbb{A}(\mathbf{0}) R_0^\bot(0) - R_0^\bot(0) \partial_{j}\mathbb{A}(\mathbf{0}) P,\ \ j=1,\dots,d.
\end{multline*}
The operator $R_0^\bot(0)$ can be written as  $R_0^\bot(0) = P^\perp \mathbb{A}(\mathbf{0})^{-1}P^\perp$.
Here   $\mathbb{A}(\mathbf{0})^{-1}$ is understood as the inverse operator to  $ \mathbb{A}(\mathbf{0})
\vert_{\mathfrak{N}^\perp} : {\mathfrak{N}^\perp} \to {\mathfrak{N}^\perp}$; this operator is correctly defined and bounded.  We arrive at representation \eqref{F_j=prop}.

It remains to prove estimate \eqref{[F]1j=le}.
From \eqref{Fj=}, the estimate $\|R_0(\zeta)\| \leqslant 6 d_0^{-1}$ for  $\zeta \in \Gamma$, and \eqref{e2.3}
it follows that
$$
\left\|[F]_1(\boldsymbol{\xi}) \right\| \leqslant \frac{6 (\pi +2)}{\pi d_0} \| [\Delta_1 \mathbb{A} ](\boldsymbol{\xi}) \|
\leqslant   \frac{6 (\pi +2) \mu_+ M_1(a)}{\pi d_0} | \boldsymbol{\xi} | = C_1(a,\mu) |\boldsymbol{\xi}|,\quad \boldsymbol{\xi} \in \widetilde{\Omega}.
$$
\end{proof}

Representation \eqref{F_j=prop}  can be  ``deciphered''  in terms of solutions of auxiliary problems.
From the identity
$P=(\cdot,\mathbf{1}_{\Omega})\mathbf{1}_{\Omega}$ it follows that
\begin{equation}\label{e2.20}
\partial_{j}\mathbb{A}(\mathbf{0})P=i(\cdot,\mathbf{1}_{\Omega})w_{j},\ \
\text{where}\ \ iw_{j}=\partial_{j}\mathbb{A}(\mathbf{0})\mathbf{1}_{\Omega},\
\ j=1,\dots,d.
\end{equation}
By \eqref{e2.1},
\begin{multline}
\label{e2.21}
w_{j}(\mathbf{x})=\overline{w_{j}(\mathbf{x})}=\intop_{\Omega}\sum_{\mathbf{n}\in\mathbb{Z}^d}(x_{j}-y_{j}+n_{j})a(\mathbf{x}-\mathbf{y}+\mathbf{n})\mu(\mathbf{x},\mathbf{y})\,d\mathbf{y}=
\\
= \intop_{\mathbb{R}^d}(x_{j}-y_{j})a(\mathbf{x} - \mathbf{y})\mu(\mathbf{x},\mathbf{y})\,d\mathbf{y},\ \ \mathbf{x}\in\Omega,\ \ j=1,\dots,d.
\end{multline}
Here we have taken into account periodicity of  $\mu$.
It is easily seen that  $P w_j=0$, i.~e., $w_j = P^\perp w_j$. Hence,
\begin{equation}\label{e2.20a}
 P \partial_{j}\mathbb{A}(\mathbf{0})P= 0,\ \ j=1,\dots,d.
\end{equation}
From  \eqref{e2.20} and \eqref{e2.20a} it follows that
\begin{equation}
\label{e2.22}
P^{\perp}\mathbb{A}(\mathbf{0})^{-1}P^{\perp}\partial_{j} \mathbb{A}(\mathbf{0})P=i(\cdot,\mathbf{1}_{\Omega})v_{j},
\ \ v_{j}=P^{\perp}\mathbb{A}(\mathbf{0})^{-1}P^{\perp}w_{j},\ \ j=1,\dots,d.
\end{equation}
The functions $v_{j}=\overline{v_{j}}\in L_{2}(\Omega)$,
$j=1,\dots,d$, are solutions of the following problems on the cell $\Omega$:
\begin{equation}
\label{e2.22a}
\int\limits_{\Omega}\widetilde
a(\mathbf{0},\mathbf{x}-\mathbf{y})\mu(\mathbf{x},\mathbf{y})(v_{j}(\mathbf{x})-v_{j}(\mathbf{y}))\,d\mathbf{y}=w_{j}(\mathbf{x}),\ \ \mathbf{x} \in\Omega;\quad
\int\limits_{\Omega}v_{j}(\mathbf{x})\,d\mathbf{x} =0.
\end{equation}
Assuming that the functions $v_{j}\in L_{2}(\Omega)$, $j=1,\dots,d$,
are periodically extended to  $\mathbb{R}^d$, we can rewrite the auxiliary problems as follows:
\begin{multline}\label{e2.23}
\intop\limits_{\mathbb{R}^d}
a(\mathbf{x}- \mathbf{y})\mu(\mathbf{x},\mathbf{y})(v_{j}(\mathbf{x})-v_{j}(\mathbf{y}))\,d\mathbf{y} = \intop\limits_{\mathbb{R}^d}
a(\mathbf{x}-\mathbf{y})\mu(\mathbf{x},\mathbf{y})(x_{j}-y_{j}) \,d\mathbf{y},\ \ \mathbf{x}\in\Omega;
\\
 \intop\limits_{\Omega} v_{j}(\mathbf{x})\,d\mathbf{x}=0.
\end{multline}
The problems \eqref{e2.23} are uniquely solvable.
In Section \ref{Sec5} it is shown that  $v_j \in L_\infty(\Omega)$.

From \eqref{e2.22} it follows that
\begin{equation}
\label{e2.23a}
P \partial_{j} \mathbb{A}(\mathbf{0})
P^{\perp}\mathbb{A}(\mathbf{0})^{-1}P^{\perp} =
\left(P^{\perp}\mathbb{A}(\mathbf{0})^{-1}P^{\perp}\partial_{j} \mathbb{A}(\mathbf{0})P \right)^* = - i(\cdot,v_j) \mathbf{1}_{\Omega},
\ \ j=1,\dots,d.
\end{equation}
Finally, as a consequence of relations \eqref{F_j=prop}, \eqref{e2.22},
and \eqref{e2.23a} we have the following statement:
\begin{proposition}
Under the assumptions of Proposition \emph{\ref{prop2.4}} the operators  $F_j$
admit the following representations\emph{:}
\begin{equation}
\label{Fj===}
F_j =  i(\cdot,v_j) \mathbf{1}_{\Omega} - i(\cdot,\mathbf{1}_{\Omega})v_{j},
\ \ j=1,\dots,d,
\end{equation}
where $v_j$ is the periodic solution of problem \eqref{e2.23}.
\end{proposition}

Now we proceed to the threshold approximations for the operator  $\mathbb{A}(\boldsymbol{\xi})F(\boldsymbol{\xi})$.

\begin{proposition}
\label{prop2.6}
Suppose that conditions  \eqref{e1.1}--\eqref{e1.3} are satisfied and
$M_3(a) < \infty$.
Then
\begin{equation}
\label{AF = G2}
 \mathbb{A}(\boldsymbol{\xi})F(\boldsymbol{\xi}) =  [G]_2(\boldsymbol{\xi}) + \Psi(\boldsymbol{\xi}), \quad
 [G]_2(\boldsymbol{\xi}) := \frac{1}{2}\sum_{k,l=1}^d G_{kl} \xi_k \xi_l,
\end{equation}
and
\begin{equation}
\label{AF-G}
\left\| \Psi (\boldsymbol{\xi})  \right\| \leqslant C_3(a,\mu) | \boldsymbol{\xi} |^3, \quad  | \boldsymbol{\xi} | \leqslant \delta_0(a,\mu).
\end{equation}
The operators  $G_{kl}$ are given by
\begin{equation}
\label{G_kl=prop}
\begin{split}
G_{kl} =& P \partial_{k} \partial_{l} \mathbb{A}(\mathbf{0}) P
- P \partial_{k}  \mathbb{A}(\mathbf{0}) P^\perp \mathbb{A}(\mathbf{0})^{-1}P^\perp
 \partial_{l}  \mathbb{A}(\mathbf{0}) P
\\
&-  P \partial_{l}  \mathbb{A}(\mathbf{0}) P^\perp \mathbb{A}(\mathbf{0})^{-1}P^\perp   \partial_{k}
 \mathbb{A}(\mathbf{0}) P,\ \ k,l=1,\dots,d.
\end{split}
\end{equation}
The constant $C_3(a,\mu)$ is defined below in  \eqref{C3} and depends only on  $\mu_-$, $\mu_+$,  $\mathcal{C}_\pi(a)$, $M_1(a)$, $M_2(a)$, $M_3(a)$.
\end{proposition}

\begin{proof}
Iterating the resolvent identity  \eqref{e2.31} once again, we obtain
\begin{equation}
\label{res_iden3}
\begin{aligned}
&R(\boldsymbol{\xi},\zeta) = R_0(\zeta)
- R_0(\zeta) \Delta\mathbb{A}(\boldsymbol{\xi}) R_0(\zeta) + R_0(\zeta) \Delta\mathbb{A}(\boldsymbol{\xi}) R_0(\zeta) \Delta\mathbb{A}(\boldsymbol{\xi}) R_0(\zeta)
+ Z_3(\boldsymbol{\xi},\zeta),
\\
&Z_3(\boldsymbol{\xi},\zeta) :=
- R(\boldsymbol{\xi},\zeta) \Delta\mathbb{A}(\boldsymbol{\xi}) R_0(\zeta) \Delta\mathbb{A}(\boldsymbol{\xi}) R_0(\zeta) \Delta\mathbb{A}(\boldsymbol{\xi}) R_0(\zeta), \ \;
| \boldsymbol{\xi} |\leqslant \delta_{0}(a,\mu),\ \zeta\in\Gamma.
\end{aligned}
\end{equation}
From \eqref{2.3} and \eqref{e2.32} it follows that
\begin{equation}
\label{res_est4a}
\| Z_3(\boldsymbol{\xi},\zeta) \| \leqslant (6 d_0^{-1})^4 \mu^3_+ M_1(a)^3 |\boldsymbol{\xi}|^3,
\quad
|\boldsymbol{\xi} |\leqslant\delta_{0}(a,\mu),\ \ \zeta\in\Gamma.
\end{equation}
Next, substituting \eqref{2.6} into the second term on the right-hand side of  \eqref{res_iden3}
and   \eqref{2.4} into the third term, we obtain
\begin{equation}
\label{res_iden4}
\begin{aligned}
R(\boldsymbol{\xi},\zeta) =& \,R_0(\zeta)
- R_0(\zeta) \left( [\Delta_1 \mathbb{A}](\boldsymbol{\xi} ) + [\Delta_2 \mathbb{A}](\boldsymbol{\xi} )  \right)  R_0(\zeta)
\\
+& R_0(\zeta)   [\Delta_1 \mathbb{A}](\boldsymbol{\xi} )   R_0(\zeta)
 [\Delta_1 \mathbb{A}](\boldsymbol{\xi} )  R_0(\zeta)
+ Z_3(\boldsymbol{\xi},\zeta) + Z_4(\boldsymbol{\xi},\zeta),
\\
Z_4(\boldsymbol{\xi},\zeta) :=&
- R_0(\zeta) \mathbb{K}_2(\boldsymbol{\xi})  R_0(\zeta) + R_0(\zeta) \mathbb{K}_1(\boldsymbol{\xi})  R_0(\zeta) \Delta\mathbb{A}(\boldsymbol{\xi}) R_0(\zeta)
\\
+& R_0(\zeta) [\Delta_1 \mathbb{A}](\boldsymbol{\xi} )   R_0(\zeta) \mathbb{K}_1(\boldsymbol{\xi})  R_0(\zeta)
, \quad
|\boldsymbol{\xi} | \leqslant\delta_{0}(a,\mu),\ \ \zeta\in\Gamma.
\end{aligned}
\end{equation}
From \eqref{2.3}, \eqref{2.5}, \eqref{2.7}, \eqref{e2.3}, and \eqref{e2.32} it follows that
 \begin{equation}
\label{res_est5}
\| Z_4(\boldsymbol{\xi},\zeta) \| \leqslant \left(6 d_0^{-2} \mu_+ M_3(a)  + (6 d_0^{-1})^3 \mu_+^2 M_1(a)M_2(a) \right)|\boldsymbol{\xi}|^3 ,
\ \;
| \boldsymbol{\xi} |\leqslant\delta_{0}(a,\mu),\  \zeta\in\Gamma.
\end{equation}

Applying the Riesz formula  \eqref{e2.6a} and representation  \eqref{res_iden4}, we obtain
\begin{equation}
\label{AF=}
\mathbb{A}( \boldsymbol{\xi} )F(\boldsymbol{\xi}) = G_0  + \sum_{j=1}^d G_j \xi_j + \frac{1}{2} \sum_{k,l=1}^d G_{kl} \xi_k \xi_l + \Psi(\boldsymbol{\xi}), \quad
| \boldsymbol{\xi} | \leqslant \delta_{0}(a,\mu),
\end{equation}
where
\begin{align}
\label{G0=}
G_0 &= - \frac{1}{2\pi i}\oint_{\Gamma} R_0(\zeta) \zeta  \, d\zeta,
\\
\label{Gj=}
G_j &=  \frac{1}{2\pi i}\oint_{\Gamma} R_0(\zeta)\partial_{j}\mathbb{A}(\mathbf{0}) R_0(\zeta) \zeta  \, d\zeta,
\\
\label{Gkl=}
\begin{split}
G_{kl} &=  \frac{1}{2\pi i}\oint_{\Gamma} R_0(\zeta)\partial_{k} \partial_{l} \mathbb{A}(\mathbf{0})
R_0(\zeta) \zeta  \, d\zeta
\\
&- \frac{1}{2\pi i}\oint_{\Gamma}   R_0(\zeta)\partial_{k} \mathbb{A}(\mathbf{0}) R_0(\zeta) \partial_{l}
\mathbb{A}(\mathbf{0})R_0(\zeta)  \zeta  \, d\zeta
\\
&- \frac{1}{2\pi i}\oint_{\Gamma} R_0(\zeta)\partial_{l} \mathbb{A}(\mathbf{0}) R_0(\zeta) \partial_{k} \mathbb{A}(\mathbf{0})R_0(\zeta)  \zeta  \, d\zeta,
\end{split}
\\
\label{Psi=}
\Psi(\boldsymbol{\xi}) &= - \frac{1}{2\pi i}\oint_{\Gamma} (Z_3(\boldsymbol{\xi},\zeta) + Z_4(\boldsymbol{\xi},\zeta)) \zeta  \, d\zeta.
\end{align}
Relations \eqref{res_est4a} and \eqref{res_est5} yield estimate  \eqref{AF-G} for the operator  \eqref{Psi=}
with the constant
\begin{equation}
\label{C3}
C_{3}(a,\mu):=\frac{(\pi+2)}{2\pi} \left( \frac{6^3  \mu_{+}^3 M_{1}(a)^3}{d_0^2} + \mu_{+} M_{3}(a)
+ \frac{36 \mu_{+}^2 M_{1}(a) M_2(a)}{d_0} \right).
\end{equation}

To calculate the integrals in  \eqref{G0=}--\eqref{Gkl=},
we apply  \eqref{4.10} and take into account that the operator-valued function $R_{0}^{\bot}(\zeta)$ is holomorphic inside the contour  $\Gamma$. We get
\begin{equation}
\label{G0=0}
G_{0}= - \frac{1}{2\pi i}\oint_{\Gamma} \Bigl(-\frac{1}{\zeta}P+R_{0}^{\bot}(\zeta)\Bigr)
\zeta \,d\zeta= 0;
\end{equation}
\begin{multline}
\label{Gj=0}
G_{j}=  \frac{1}{2\pi i}\oint_{\Gamma} \Bigl(-\frac{1}{\zeta}P+R_{0}^{\bot}(\zeta)\Bigr)
\partial_{j}\mathbb{A}(\mathbf{0})
\Bigl(-\frac{1}{\zeta}P+R_{0}^{\bot}(\zeta)\Bigr) \zeta \,d\zeta
\\
=
\frac{1}{2\pi i}\oint_{\Gamma} \frac{1}{\zeta}  P \partial_{j}\mathbb{A}(\mathbf{0}) P  \,d\zeta
=  P \partial_{j}\mathbb{A}(\mathbf{0}) P =0,\ \ j=1,\dots,d.
\end{multline}
We have taken \eqref{e2.20a} into account. Next, we have
\begin{multline*}
\label{Gkl==}
G_{kl}= \frac{1}{2\pi i}\oint_{\Gamma} \Bigl(-\frac{1}{\zeta}P+R_{0}^{\bot}(\zeta)\Bigr)
\partial_{k} \partial_l \mathbb{A}(\mathbf{0})
\Bigl(-\frac{1}{\zeta}P+R_{0}^{\bot}(\zeta)\Bigr) \zeta \,d\zeta
\\
- \frac{1}{2\pi i}\oint_{\Gamma}  \Bigl(-\frac{1}{\zeta}P+R_{0}^{\bot}(\zeta)\Bigr)\partial_{k} \mathbb{A}(\mathbf{0})
\Bigl(-\frac{1}{\zeta}P+R_{0}^{\bot}(\zeta)\Bigr)  \partial_{l} \mathbb{A}(\mathbf{0}) \Bigl(-\frac{1}{\zeta}P+R_{0}^{\bot}(\zeta)\Bigr)
 \zeta  \, d\zeta
 \\
- \frac{1}{2\pi i}\oint_{\Gamma}  \Bigl(-\frac{1}{\zeta}P+R_{0}^{\bot}(\zeta)\Bigr)\partial_{l} \mathbb{A}(\mathbf{0})
\Bigl(-\frac{1}{\zeta}P+R_{0}^{\bot}(\zeta)\Bigr)  \partial_{k} \mathbb{A}(\mathbf{0}) \Bigl(-\frac{1}{\zeta}P+R_{0}^{\bot}(\zeta)\Bigr)
 \zeta  \, d\zeta
 \\
=  P \partial_{k} \partial_{l} \mathbb{A}(\mathbf{0}) P
- P \partial_{k}  \mathbb{A}(\mathbf{0}) R_{0}^{\bot}(0)  \partial_{l}  \mathbb{A}(\mathbf{0}) P
-  P \partial_{l}  \mathbb{A}(\mathbf{0}) R_{0}^{\bot}(0)  \partial_{k}  \mathbb{A}(\mathbf{0}) P,\ \ k,l=1,\dots,d.
\end{multline*}
We have taken into account  that  the operator-valued function $R_{0}^{\bot}(\zeta)$ is holomorphic inside the contour $\Gamma$, and also used  \eqref{e2.20a}. This yields representation  \eqref{G_kl=prop}.

Representation  \eqref{AF = G2} follows from  \eqref{AF=},  \eqref{G0=0}, and \eqref{Gj=0}.
\end{proof}

Let us now decipher representation  \eqref{G_kl=prop} in terms of solutions of auxiliary problems.
From \eqref{e2.20} and \eqref{e2.22} it follows that
\begin{equation}\label{e2.24}
P\partial_{k}\mathbb{A}(\mathbf{0})P^{\perp}\mathbb{A} (\mathbf{0})^{-1}P^{\perp}\partial_{l}\mathbb{A}(\mathbf{0})P=(v_{l},w_{k})P,\
\ k,l=1,\dots,d,
\end{equation}
where the functions $w_{j}\in L_{2}(\Omega)$, $j=1,\dots,d$, are defined by  \eqref{e2.21} and the functions  $v_{j}\in L_{2}(\Omega)$,
$j=1,\dots,d$, satisfy the auxiliary problems  (\ref{e2.23}).
Since $P = (\cdot, \mathbf{1}_\Omega) \mathbf{1}_\Omega$, we have
\begin{equation}
\label{e2.25}
P\partial_{k}\partial_{l}\mathbb{A} (\mathbf{0})P=(w_{kl},\mathbf{1}_{\Omega})P,\
\ k,l=1,\dots,d.
\end{equation}
Here
\begin{equation}
\label{e2.54a}
w_{kl} = \overline{w_{kl}}=\partial_{k}\partial_{l}\mathbb{A}(\mathbf{0})\mathbf{1}_{\Omega}\in L_{2}(\Omega),
\end{equation}
 i.\,e.,
\begin{multline}\label{e2.26}
w_{kl}(\mathbf{x})=\intop_{\Omega}\sum_{\mathbf{n}\in \mathbb{Z}^d}(x_{k}-y_{k}+n_{k})(x_{l}-y_{l}+n_{l})a(\mathbf{x} - \mathbf{y}+\mathbf{n})
\mu(\mathbf{x}, \mathbf{y} )\,d\mathbf{y} 
\\
= \intop_{\mathbb{R}^d}(x_{k}-y_{k})(x_{l}-y_{l})a( \mathbf{x} - \mathbf{y})\mu( \mathbf{x}, \mathbf{y})\, d\mathbf{y},\ \
 \mathbf{x} \in\Omega,\ \ k,l=1,\dots,d.
\end{multline}
Thus, relations  \eqref{G_kl=prop} and \eqref{e2.24}--\eqref{e2.26} imply the folllowing statement.

\begin{proposition}
Under the assumptions of Proposition \emph{\ref{prop2.6}} the operators  $G_{kl}$
admit the representation
\begin{equation*}
G_{kl} = g_{kl} P,
\ \ k,l=1,\dots,d,
\end{equation*}
where
\begin{multline}
\label{Gkl===}
g_{kl} = (w_{kl},\mathbf{1}_{\Omega})-(v_{k},w_{l})-(v_{l},w_{k})
\\
= \intop_{\Omega}d\mathbf{x}\intop_{\mathbb{R}^d} d \mathbf{y}((x_{k}-y_{k})(x_{l}-y_{l})-v_{k}(\mathbf{x})(x_{l}-y_{l})-v_{l}(\mathbf{x})(x_{k}-y_{k}))a(\mathbf{x} -\mathbf{y})\mu(\mathbf{x},\mathbf{y}),
\end{multline}
and $v_j$ is the periodic solution of problem \eqref{e2.23}.
Thus,
\begin{equation}
\label{Gkl====}
 [G]_2(\boldsymbol{\xi}) = \frac{1}{2}\sum_{k,l=1}^d G_{kl} \xi_k \xi_l =
 \frac{1}{2}\sum_{k,l=1}^d g_{kl} \xi_k \xi_l P =
 \langle g^0 \boldsymbol{\xi}, \boldsymbol{\xi} \rangle P,\quad  \boldsymbol{\xi}   \in \mathbb{R}^d,
 \end{equation}
 where $g^0$ is the  $(d \times d)$-matrix with the entries $\frac{1}{2} g_{kl}$, $k,l=1, \dots,d$.
\end{proposition}

The matrix  $g^0$ is called the  \emph{effective matrix}; below in Section \ref{sec3.1} we will show that $g^0$ is positive definite.

\begin{proposition}
\label{prop2.8}
Assume that conditions  \eqref{e1.1}--\eqref{e1.3} hold and $M_4(a) < \infty$.
Then
\begin{equation}
\label{AF= G + G3}
 \mathbb{A}( \boldsymbol{\xi} )F( \boldsymbol{\xi} ) = [G]_2(\boldsymbol{\xi})  + [G]_3(\boldsymbol{\xi})
 + \Upsilon(\boldsymbol{\xi})
\end{equation}
and
\begin{equation}
\label{AF-G - G3}
\left\| \Upsilon (\boldsymbol{\xi})   \right\| \leqslant C_4(a,\mu) |\boldsymbol{\xi} |^4, \quad  |\boldsymbol{\xi}| \leqslant \delta_0(a,\mu).
\end{equation}
Here $[G]_2(\boldsymbol{\xi})$ is defined by \eqref{AF = G2}, \eqref{G_kl=prop}, and $[G]_3(\boldsymbol{\xi})$ is given by
\begin{equation}
\label{G_jkl=prop}
\begin{split}
 &[G]_3(\boldsymbol{\xi}) := \frac{1}{6}\sum_{j,k,l=1}^d G_{jkl} \xi_j \xi_k \xi_l
 \\
 =&\,
  P  [\Delta_3 \mathbb{A}](\boldsymbol{\xi}) P - P [\Delta_1 \mathbb{A}](\boldsymbol{\xi})  R_0^\perp(0)
   [\Delta_2\mathbb{A}](\boldsymbol{\xi}) P
  \\
  -& P[\Delta_2 \mathbb{A}](\boldsymbol{\xi}) R_0^\perp(0)  [\Delta_1 \mathbb{A}](\boldsymbol{\xi})P  +  P [\Delta_1\mathbb{A}](\boldsymbol{\xi}) R_0^\perp(0)   [\Delta_1 \mathbb{A}]( \boldsymbol{\xi} ) R_0^\perp(0)  [\Delta_1\mathbb{A}](\boldsymbol{\xi}) P
\\
   -& R_0^\perp(0) [\Delta_1\mathbb{A}]( \boldsymbol{\xi} )  P  [\Delta_2\mathbb{A}]( \boldsymbol{\xi} ) P
   -  P  [\Delta_2 \mathbb{A} ](\boldsymbol{\xi} ) P  [\Delta_1 \mathbb{A}](\boldsymbol{\xi})  R_0^\perp(0)
  \\
 +&    R_0^\perp(0)   [\Delta_1 \mathbb{A} ](\boldsymbol{\xi} ) P  [\Delta_1 \mathbb{A}]( \boldsymbol{\xi} ) R_0^\perp(0)  [\Delta_1\mathbb{A}](\boldsymbol{\xi} ) P
  \\
  +&  P [\Delta_1 \mathbb{A}]( \boldsymbol{\xi}) R_0^\perp(0)   [\Delta_1 \mathbb{A} ](\boldsymbol{\xi}) P
   [\Delta_1 \mathbb{A} ](\boldsymbol{\xi}) R_0^\perp(0).
  \end{split}
\end{equation}
The constant $C_4(a,\mu)$ is defined below in  \eqref{C4} and depends only on $\mu_-$, $\mu_+$,  $\mathcal{C}_\pi(a)$, $M_1(a)$, $M_2(a)$, $M_3(a)$, $M_4(a)$.
\end{proposition}

\begin{proof}
Iterating the resolvent identity  \eqref{e2.31} once again, we obtain
\begin{equation}
\label{res_iden10}
\begin{aligned}
R(\boldsymbol{\xi},\zeta) =& \,R_0(\zeta)
- R_0(\zeta) \Delta \mathbb{A}(\boldsymbol{\xi}) R_0(\zeta) + R_0(\zeta) \Delta \mathbb{A}(\boldsymbol{\xi}) R_0(\zeta) \Delta\mathbb{A}(\boldsymbol{\xi}) R_0(\zeta)
\\
&- R_0(\zeta) \Delta\mathbb{A}(\boldsymbol{\xi}) R_0(\zeta) \Delta\mathbb{A}(\boldsymbol{\xi} ) R_0(\zeta) \Delta\mathbb{A}(\boldsymbol{\xi}) R_0(\zeta)
+ Z_5(\boldsymbol{\xi},\zeta),
\\
Z_5(\boldsymbol{\xi},\zeta) :=&\,
 R(\boldsymbol{\xi},\zeta) \Delta\mathbb{A}(\boldsymbol{\xi}) R_0(\zeta) \Delta\mathbb{A}(\boldsymbol{\xi}) R_0(\zeta) \Delta \mathbb{A}(\boldsymbol{\xi}) R_0(\zeta) \Delta\mathbb{A} ( \boldsymbol{\xi} ) R_0(\zeta),
\end{aligned}
\end{equation}
for $| \boldsymbol{\xi} |\leqslant \delta_{0}(a,\mu)$ and $\zeta\in\Gamma$.
From  \eqref{2.3} and \eqref{e2.32} it follows that
\begin{equation}
\label{res_est10}
\| Z_5( \boldsymbol{\xi},\zeta) \| \leqslant (6 d_0^{-1})^5 \mu^4_+ M_1(a)^4 | \boldsymbol{\xi}|^4,
\quad
| \boldsymbol{\xi} | \leqslant \delta_{0}(a,\mu),\ \ \zeta\in\Gamma.
\end{equation}

Substituting  \eqref{2.8} into the second term on the right-hand side of  \eqref{res_iden10},
\eqref{2.6} into the third term, and   \eqref{2.4} into the fourth term, we obtain
\begin{equation}
\label{res_iden10a}
\begin{aligned}
R(  \boldsymbol{\xi},\zeta) =& \,R_0(\zeta)
- R_0(\zeta)  \left( [\Delta_1 \mathbb{A}](  \boldsymbol{\xi} ) +  [\Delta_2 \mathbb{A}](  \boldsymbol{\xi}) +
 [\Delta_3 \mathbb{A}]( \boldsymbol{\xi}) \right) R_0(\zeta)
\\
+& R_0(\zeta) [\Delta_1\mathbb{A}]( \boldsymbol{\xi})  R_0(\zeta) [\Delta_1 \mathbb{A}]( \boldsymbol{\xi})  R_0(\zeta)
\\
+& R_0(\zeta) [\Delta_1 \mathbb{A}](\boldsymbol{\xi})   R_0(\zeta) [\Delta_2 \mathbb{A}](\boldsymbol{\xi})   R_0(\zeta)
\\
+& R_0(\zeta) [\Delta_2 \mathbb{A}](\boldsymbol{\xi} )   R_0(\zeta) [\Delta_1 \mathbb{A}](\boldsymbol{\xi})  R_0(\zeta)
\\
-& R_0(\zeta) [\Delta_1 \mathbb{A}]( \boldsymbol{\xi} )  R_0(\zeta) [\Delta_1 \mathbb{A}](\boldsymbol{\xi})  R_0(\zeta) [\Delta_1 \mathbb{A}](\boldsymbol{\xi})  R_0(\zeta)
\\
+& Z_5(\boldsymbol{\xi},\zeta) + Z_6(\boldsymbol{\xi},\zeta), \quad
| \boldsymbol{\xi} | \leqslant \delta_{0}(a,\mu),\ \ \zeta\in\Gamma.
\end{aligned}
\end{equation}
where
\begin{equation*}
\label{res_iden11}
\begin{aligned}
Z_6(\boldsymbol{\xi},\zeta) :=&
- R_0(\zeta) \mathbb{K}_3(\boldsymbol{\xi})  R_0(\zeta) + R_0(\zeta) \mathbb{K}_2(\boldsymbol{\xi})  R_0(\zeta) \Delta\mathbb{A}(\boldsymbol{\xi}) R_0(\zeta)
\\
&+ R_0(\zeta)  [\Delta_2 \mathbb{A}](\boldsymbol{\xi})  R_0(\zeta)  \mathbb{K}_1(\boldsymbol{\xi})  R_0(\zeta)
\\
&+ R_0(\zeta) [\Delta_1 \mathbb{A}](\boldsymbol{\xi})  R_0(\zeta) \mathbb{K}_2(\boldsymbol{\xi})  R_0(\zeta)
\\
&-  R_0(\zeta) \Delta \mathbb{A}(\boldsymbol{\xi})  R_0(\zeta) \Delta \mathbb{A}(\boldsymbol{\xi})   R_0(\zeta) \mathbb{K}_1(\boldsymbol{\xi})  R_0(\zeta)
\\
& -  R_0(\zeta) \Delta \mathbb{A}(\boldsymbol{\xi})  R_0(\zeta) \mathbb{K}_1 (\boldsymbol{\xi})  R_0(\zeta)   [\Delta_1 \mathbb{A}](\boldsymbol{\xi}) R_0(\zeta)
 \\
 & -  R_0(\zeta) \mathbb{K}_1 (\boldsymbol{\xi})  R_0(\zeta)   [\Delta_1 \mathbb{A}](\boldsymbol{\xi}) R_0(\zeta)
  [\Delta_1 \mathbb{A}](\boldsymbol{\xi}) R_0(\zeta).
\end{aligned}
\end{equation*}

Relations \eqref{2.3}, \eqref{2.5}, \eqref{2.7}, \eqref{2.9}--\eqref{e2.4}, and \eqref{e2.32} imply that
 \begin{multline}
\label{res_est12}
\| Z_6(\boldsymbol{\xi},\zeta) \| \leqslant \Bigl(\frac{3}{2} d_0^{-2} \mu_+ M_4(a)  +  d_0^{-3} \mu_+^2 (72 M_1(a)M_3(a) \!+\!
54  M_2(a)^2)
\\
+\frac{3}{2} \cdot 6^4 d_0^{-4}\mu_+^3 M_1(a)^2 M_2(a)\Bigr)|\boldsymbol{\xi}|^4,
\ \
| \boldsymbol{\xi} | \leqslant \delta_{0}(a,\mu),\ \ \zeta\in\Gamma.
\end{multline}

Applying the Riesz formula  \eqref{e2.6a} and representation  \eqref{res_iden10a}, we obtain
\begin{equation*}
\label{AF=corrector}
 \mathbb{A} (\boldsymbol{\xi})F(\boldsymbol{\xi}) = G_0  + \sum_{j=1}^d G_j \xi_j + \frac{1}{2} \sum_{k,l=1}^d G_{kl} \xi_k \xi_l +
\frac{1}{6} \sum_{j,k,l=1}^d G_{jkl} \xi_j \xi_k \xi_l +
 \Upsilon(\boldsymbol{\xi})
 \end{equation*}
for $| \boldsymbol{\xi}| \leqslant \delta_{0}(a,\mu)$, where the operators
$G_0$, $G_j$, $G_{kl}$ are defined by  \eqref{G0=}--\eqref{Gkl=},
and the fourth and fifth terms on the right are given by
\begin{equation}
\label{Gjkl=1}
\begin{aligned}
&[G]_3(\boldsymbol{\xi}) = \frac{1}{6} \sum_{j,k,l=1}^d G_{jkl} \xi_j \xi_k \xi_l =
 \frac{1}{2\pi i}\oint_{\Gamma} R_0(\zeta)  [\Delta_3\mathbb{A}](\boldsymbol{\xi})  R_0(\zeta)  \zeta  \, d\zeta
\\
&- \frac{1}{2\pi i}\oint_{\Gamma} R_0(\zeta)  [\Delta_1 \mathbb{A}](\boldsymbol{\xi})  R_0(\zeta)   [\Delta_2\mathbb{A}](\boldsymbol{\xi})  R_0(\zeta)
 \zeta  \, d\zeta
 \\
 &- \frac{1}{2\pi i}\oint_{\Gamma} R_0(\zeta)  [\Delta_2 \mathbb{A}]( \boldsymbol{\xi} )  R_0(\zeta)   [\Delta_1 \mathbb{A}](\boldsymbol{\xi} )  R_0(\zeta)
 \zeta  \, d\zeta
 \\
 &+   \frac{1}{2\pi i}\oint_{\Gamma} R_0(\zeta)  [\Delta_1 \mathbb{A} ]( \boldsymbol{\xi})  R_0(\zeta)   [\Delta_1 \mathbb{A}](\boldsymbol{\xi})  R_0(\zeta)
  [\Delta_1 \mathbb{A} ](\boldsymbol{\xi})  R_0(\zeta) \zeta  \, d\zeta
\end{aligned}
\end{equation}
and
\begin{equation}
\label{Upsilon=1}
\Upsilon(\boldsymbol{\xi}) = -  \frac{1}{2\pi i}\oint_{\Gamma}  (Z_5(\boldsymbol{\xi},\zeta) + Z_6( \boldsymbol{\xi},\zeta)) \zeta  \, d\zeta.
\end{equation}
Relations \eqref{res_est10} and \eqref{res_est12} imply estimate \eqref{AF-G - G3} for the operator  \eqref{Upsilon=1}
with the constant
\begin{multline}
\label{C4}
C_{4}(a,\mu):= \frac{(\pi+2)}{2\pi} \Bigl( \frac{6^4 \mu_{+}^4 M_{1}(a)^4}{d_{0}^{3}} +
\frac{\mu_{+} M_{4}(a)}{4}
\\
+ \frac{\mu_{+}^2 ( 12 M_{1}(a) M_3(a) + 9  M_{2}(a)^2)}{d_0}
+ \frac{6^4 \mu_{+}^3 M_{1}(a)^2 M_2(a)}{4 d_0^2} \Bigr).
\end{multline}

The operators $G_0$, $G_j$, $G_{kl}$ have already been determined: $G_0=0$, $G_j=0$, $j=1,\dots,d$, and $G_{kl}$
are given by  \eqref{G_kl=prop}. To calculate the integrals in  \eqref{Gjkl=1}, substitute  \eqref{4.10}  and  take into account that  the operator-valued function $R_{0}^{\bot}(\zeta)$ is holomorphic inside the contour  $\Gamma$.
For the first term on the right, we get
$$
 \frac{1}{2\pi i}\oint_{\Gamma}  R_0(\zeta)  [\Delta_3 \mathbb{A}](\boldsymbol{\xi})   R_0(\zeta) \zeta  \, d\zeta =
 P  [\Delta_3\mathbb{A}](\boldsymbol{\xi})  P.
$$
Using  \eqref{e2.20a}, we calculate the second term:
$$
\begin{aligned}
-& \frac{1}{2\pi i}\oint_{\Gamma}  R_0(\zeta)  [\Delta_1 \mathbb{A}](\boldsymbol{\xi})  R_0(\zeta)
 [\Delta_2 \mathbb{A}](\boldsymbol{\xi}) R_0(\zeta) \zeta  \, d\zeta
 \\
 &=  - R_0^\perp(0) [\Delta_1\mathbb{A}](\boldsymbol{\xi})  P  [\Delta_2 \mathbb{A} ]( \boldsymbol{\xi}) P
-  P  [\Delta_1 \mathbb{A}](\boldsymbol{\xi})  R_0^\perp(0)  [\Delta_2 \mathbb{A}](\boldsymbol{\xi} )  P.
\end{aligned}
$$
Similarly, the third term is represented as
$$
\begin{aligned}
-& \frac{1}{2\pi i}\oint_{\Gamma} R_0(\zeta)  [\Delta_2 \mathbb{A}](\boldsymbol{\xi})  R_0(\zeta)
[\Delta_1 \mathbb{A}](\boldsymbol{\xi})  R_0(\zeta)
 \zeta  \, d\zeta
 \\
 & = - P [\Delta_2\mathbb{A}](\boldsymbol{\xi}) R_0^\perp(0)  [\Delta_1 \mathbb{A}](\boldsymbol{\xi}) P
-  P  [\Delta_2\mathbb{A}](\boldsymbol{\xi}) P  [\Delta_1 \mathbb{A}](\boldsymbol{\xi})  R_0^\perp(0).
\end{aligned}
$$
Finally, the fourth term takes the form
$$
\begin{aligned}
  \frac{1}{2\pi i}\oint_{\Gamma} R_0(\zeta)  [\Delta_1 \mathbb{A}](\boldsymbol{\xi})  R_0(\zeta)   [\Delta_1 \mathbb{A}](\boldsymbol{\xi})  R_0(\zeta)
  [\Delta_1 \mathbb{A}](\boldsymbol{\xi} )  R_0(\zeta) \zeta  \, d\zeta
  \\
  = R_0^\perp(0)   [\Delta_1 \mathbb{A}](\boldsymbol{\xi}) P  [\Delta_1\mathbb{A}](\boldsymbol{\xi}) R_0^\perp(0)  [\Delta_1\mathbb{A}](\boldsymbol{\xi}) P
  \\
  +  P [\Delta_1 \mathbb{A}] (\boldsymbol{\xi}) R_0^\perp(0)   [\Delta_1 \mathbb{A} ](\boldsymbol{\xi}) P  [\Delta_1 \mathbb{A}](\boldsymbol{\xi}) R_0^\perp(0)
\\
  +  P [\Delta_1\mathbb{A}](\boldsymbol{\xi}) R_0^\perp(0)   [\Delta_1 \mathbb{A}](\boldsymbol{\xi}) R_0^\perp(0)
  [\Delta_1 \mathbb{A}](\boldsymbol{\xi}) P.
  \end{aligned}
$$
As a result, we arrive at  \eqref{G_jkl=prop}.
\end{proof}

\begin{proposition}
\label{prop2.9}
Under the assumptions of Proposition  \emph{\ref{prop2.8}} we have
 \begin{equation}
\label{PG3P}
P  [G]_3(\boldsymbol{\xi} ) P =0.
\end{equation}
\end{proposition}

\begin{proof}
By \eqref{G_jkl=prop},
\begin{equation}
\label{PG_3P=}
\begin{split}
 &P [G]_3(\boldsymbol{\xi}) P =
  P  [\Delta_3 \mathbb{A}](\boldsymbol{\xi}) P - P [\Delta_1 \mathbb{A}](\boldsymbol{\xi})  R_0^\perp(0)  [\Delta_2 \mathbb{A}](\boldsymbol{\xi}) P
  \\
  &- P[\Delta_2\mathbb{A}](\boldsymbol{\xi}) R_0^\perp(0)  [\Delta_1 \mathbb{A}](\boldsymbol{\xi})P  +  P [\Delta_1 \mathbb{A}](\boldsymbol{\xi}) R_0^\perp(0)   [\Delta_1\mathbb{A}](\boldsymbol{\xi}) R_0^\perp(0)  [\Delta_1 \mathbb{A} ](\boldsymbol{\xi}) P.
  \end{split}
\end{equation}
According to  \eqref{2.8}, we have
$$
P [\Delta_3 \mathbb{A}](\boldsymbol{\xi})P = \frac{1}{6} \sum_{j,k,l=1}^d \xi_j \xi_k \xi_l (f_{jkl}, \mathbf{1}_\Omega ) P, \quad
f_{jkl} := \partial_j \partial_k \partial_l \mathbb{A}(\mathbf{0}) \mathbf{1}_\Omega.
$$
It follows from  \eqref{e2.1} that the function  $f_{jkl}(\mathbf{x})$ takes purely imaginary values,  and therefore the value
$(f_{jkl}, \mathbf{1}_\Omega )$ is purely imaginary.
Consequently,
\begin{equation*}
\left( P [\Delta_3 \mathbb{A}](\boldsymbol{\xi})P \right)^* = - \frac{1}{6} \sum_{j,k,l=1}^d \xi_j \xi_k \xi_l (f_{jkl},
\mathbf{1}_\Omega ) P.
\end{equation*}
On the other hand, the operator $P [\Delta_3 \mathbb{A} ](\boldsymbol{\xi}) P$ is selfadjoint, whence
\begin{equation}
\label{PA3P=0}
P [\Delta_3 \mathbb{A}](\boldsymbol{\xi})P = 0.
\end{equation}

Next, according to  \eqref{e2.22}, we have
\begin{equation}
\label{Delta1A_P}
R_0^\perp(0) [\Delta_1 \mathbb{A}](\boldsymbol{\xi}) P =  i \sum_{j=1}^d  \xi_j  (\cdot, \mathbf{1}_\Omega) v_{j}.
\end{equation}
By \eqref{e2.54a},
$$
[\Delta_2 \mathbb{A}](\boldsymbol{\xi}) P = \frac{1}{2} \sum_{k,l=1}^d  \xi_k \xi_l (\cdot, \mathbf{1}_\Omega ) w_{kl},
$$
whence
\begin{equation}
\label{P_DeltaA_P}
P [\Delta_2 \mathbb{A}](\boldsymbol{\xi}) =  \frac{1}{2} \sum_{k,l=1}^d  \xi_k \xi_l (\cdot, w_{kl})  \mathbf{1}_\Omega.
\end{equation}
From  \eqref{Delta1A_P} and \eqref{P_DeltaA_P} it follows that
\begin{equation*}
\label{PPP}
P [\Delta_2 \mathbb{A}](\boldsymbol{\xi})R_0^\perp(0) [\Delta_1 \mathbb{A}](\boldsymbol{\xi}) P =   \frac{i}{2} \sum_{j,k,l=1}^d  \xi_j \xi_k \xi_l (v_j,  w_{kl}) P.
\end{equation*}
Combining this with the relations $v_j = \overline{v_j}$ and $w_{kl} = \overline{w_{kl}}$, we obtain
$$
P [\Delta_1 \mathbb{A}](\boldsymbol{\xi}) R_0^\perp(0) [\Delta_2 \mathbb{A}](\boldsymbol{\xi}) P =
(P [\Delta_2 \mathbb{A}](\boldsymbol{\xi}) R_0^\perp(0) [\Delta_1 \mathbb{A}](\boldsymbol{\xi}) P)^*
%\\
=
 - \frac{i}{2} \sum_{j,k,l=1}^d  \xi_j \xi_k \xi_l (v_j,  w_{kl}) P.
$$
Therefore, the sum of the second and the third terms on the right-hand side of  \eqref{PG_3P=} is equal to zero:
\begin{equation}
\label{PPPPP}
P [\Delta_1 \mathbb{A}](\boldsymbol{\xi})R_0^\perp(0) [\Delta_2 \mathbb{A}](\boldsymbol{\xi}) P +
P [\Delta_2 \mathbb{A}](\boldsymbol{\xi})R_0^\perp(0) [\Delta_1 \mathbb{A}](\boldsymbol{\xi}) P = 0.
\end{equation}

Finally, taking  \eqref{Delta1A_P} into account, we obtain
$$
 P [\Delta_1 \mathbb{A}]( \boldsymbol{\xi}) R_0^\perp(0)   [\Delta_1 \mathbb{A}](\boldsymbol{\xi}) R_0^\perp(0)  [\Delta_1\mathbb{A}](\boldsymbol{\xi}) P =
 \sum_{j,k,l=1}^d  \xi_j \xi_k \xi_l (h_{kj},  v_{l}) P,
$$
where $h_{kj}:= \partial_k \mathbb{A}(\mathbf{0}) v_j$. Using \eqref{e2.1} and the relation $v_j = \overline{v_j}$, we see that  the function $h_{kj}(\mathbf{x})$ takes purely imaginary values. Therefore, the value
$(h_{kj}, v_l )$ is purely imaginary. On the other hand, the operator
$ P [\Delta_1 \mathbb{A}](\boldsymbol{\xi}) R_0^\perp(0)   [\Delta_1 \mathbb{A}](\boldsymbol{\xi}) R_0^\perp(0)  [\Delta_1\mathbb{A}](\boldsymbol{\xi}) P$ is selfadjoint.
Hence,
\begin{equation}
\label{PA3P=000}
 P [\Delta_1 \mathbb{A}](\boldsymbol{\xi}) R_0^\perp(0)   [\Delta_1 \mathbb{A}](\boldsymbol{\xi}) R_0^\perp(0)  [\Delta_1 \mathbb{A}](\boldsymbol{\xi}) P =  0.
\end{equation}

As a result, relations  \eqref{PG_3P=}, \eqref{PA3P=0}, \eqref{PPPPP}, and \eqref{PA3P=000} imply
the required identity \eqref{PG3P}.
\end{proof}

\section{Approximation for the resolvent  $(\mathbb{A} + \varepsilon^2 I)^{-1}$}\label{Sec3}

\subsection{Approximation for the resolvent of  $\mathbb{A}(\boldsymbol{\xi})$ in the principal order}\label{sec3.1}
Approximation for the resolvent $(\mathbb{A}(\boldsymbol{\xi}) + \varepsilon^2 I)^{-1}$ for small
$\varepsilon >0$ with an error
$O(\varepsilon^{-1})$ was constructed in  \cite[Sec. 2]{PSlSuZh}. It is convenient for us to  reproduce this result here
(see Theorem \ref{teor2.2}), and then proceed to deriving  a more  accurate approximation.

From \eqref{e1.30} it follows that
\begin{equation}\label{e2.35}
( \mathbb{A}( \boldsymbol{\xi};a,\mu)F(\boldsymbol{\xi} )u,u)\geqslant\mu_{-}C(a)| \boldsymbol{\xi}|^{2}(F(\boldsymbol{\xi})u,u),\ \
u\in L_{2}(\Omega),\ \ | \boldsymbol{\xi}|\leqslant \delta_{0}(a,\mu).
\end{equation}
Substituting  the expansions $F(\boldsymbol{\xi} ) =P+O(|\boldsymbol{\xi}|)$ (see \eqref{F-P}) and
$$\mathbb{A}( \boldsymbol{\xi})F(\boldsymbol{\xi})=
 \langle g^0 \boldsymbol{\xi}, \boldsymbol{\xi} \rangle P+O(|\boldsymbol{\xi}|^3)$$ (see \eqref{AF = G2}, \eqref{AF-G}, and \eqref{Gkl====}) in \eqref{e2.35} and letting $u = \mathbf{1}_\Omega$, we obtain
\begin{equation*}
\langle g^0 \boldsymbol{\xi},\boldsymbol{\xi} \rangle \geqslant \mu_{-}C(a)|\boldsymbol{\xi} |^{2}+O(| \boldsymbol{\xi}|^{3}),\
\ | \boldsymbol{\xi} |\to 0.
\end{equation*}
Divide the last relation by  $|\boldsymbol{\xi}|^{2}$ and let $|\boldsymbol{\xi}|\to0$.
As a result, we  see that the matrix $g^0$ is positive definite:
\begin{equation*}
 \langle g^0 \boldsymbol{\theta},\boldsymbol{\theta} \rangle \geqslant \mu_{-}C(a),\
\ \boldsymbol{\theta}\in\mathbb{S}^{d-1},
\end{equation*}
or, finally,
\begin{equation}\label{e2.36}
\langle g^0 \boldsymbol{\xi},\boldsymbol{\xi} \rangle \geqslant \mu_{-}C(a)| \boldsymbol{\xi} |^{2},\
\  \boldsymbol{\xi} \in\mathbb{R}^d.
\end{equation}

\begin{proposition}
Let conditions  \eqref{e1.1}--\eqref{e1.3} be satisfied, and assume that  $M_3(a) < \infty$.
Denote
\begin{equation}
\label{Xi=}
\Xi(\boldsymbol{\xi}, \varepsilon) := (\mathbb{A}( \boldsymbol{\xi} )+\varepsilon^{2}I)^{-1} F(\boldsymbol{\xi}) - \left( \langle g^0 \boldsymbol{\xi}, \boldsymbol{\xi} \rangle +\varepsilon^2 \right)^{-1} P,
 \ \  \boldsymbol{\xi} \in\widetilde{\Omega}, \ \ \varepsilon >0.
\end{equation}
Then we have
\begin{equation}
\label{Xi=le}
\| \Xi(\boldsymbol{\xi}, \varepsilon ) \| \leqslant
\frac{2 C_1(a,\mu) | \boldsymbol{\xi} |}{\mu_{-}C(a) | \boldsymbol{\xi} |^{2} + \varepsilon^2} + \frac{C_3(a,\mu) |\boldsymbol{\xi}|^3}{(\mu_{-}C(a)|\boldsymbol{\xi}|^{2} + \varepsilon^2)^2},
\quad | \boldsymbol{\xi} | \leqslant \delta_{0}(a,\mu), \   \varepsilon >0.
\end{equation}
\end{proposition}

\begin{proof}
 From \eqref{e1.30} and \eqref{e2.36} it follows that
\begin{align}\label{e2.39}
\|(\mathbb{A}(\boldsymbol{\xi})+\varepsilon^{2}I)^{-1} \|
& \leqslant (\mu_{-}C(a)| \boldsymbol{\xi} |^{2}+\varepsilon^{2})^{-1},\
\ \varepsilon>0,\ \ \boldsymbol{\xi} \in \widetilde{\Omega};
\\
\label{e2.40}
(\langle g^0 \boldsymbol{\xi}, \boldsymbol{\xi} \rangle +\varepsilon^{2})^{-1}
& \leqslant (\mu_{-}C(a)| \boldsymbol{\xi}|^{2}+\varepsilon^{2})^{-1},\
\ \varepsilon>0,\ \  \boldsymbol{\xi} \in \widetilde{\Omega}.
\end{align}
Obviously,
\begin{multline}\label{e2.41}
\Xi(\boldsymbol{\xi},\varepsilon) =
F(\boldsymbol{\xi})( \mathbb{A}(\boldsymbol{\xi})+\varepsilon^{2}I)^{-1}(F(\boldsymbol{\xi})-P)+(F(\boldsymbol{\xi})-P) (\langle g^0 \boldsymbol{\xi}, \boldsymbol{\xi} \rangle +\varepsilon^{2})^{-1} P
\\
-
F(\boldsymbol{\xi})(\mathbb{A}(\boldsymbol{\xi})+\varepsilon^{2}I)^{-1}
\left(\mathbb{A}(\boldsymbol{\xi})F(\boldsymbol{\xi}) - \langle g^0 \boldsymbol{\xi}, \boldsymbol{\xi} \rangle P \right)
(\langle g^0 \boldsymbol{\xi}, \boldsymbol{\xi} \rangle +\varepsilon^{2})^{-1} P.
\end{multline}
Now,  \eqref{Xi=le} follows from  relations \eqref{F-P}, \eqref{AF-G}, and
(\ref{e2.39})--(\ref{e2.41}).
\end{proof}

\begin{theorem}\label{teor2.2}
Suppose that conditions  \eqref{e1.1}--\eqref{e1.3} are satisfied and  $M_3(a) < \infty$.
Then
\begin{equation}\label{e2.37}
\left\|(\mathbb{A}(\boldsymbol{\xi})+\varepsilon^{2}I)^{-1}-
(\langle g^0 \boldsymbol{\xi}, \boldsymbol{\xi} \rangle +\varepsilon^{2})^{-1} P \right\|\leqslant
C_5(a,\mu)\varepsilon^{-1},\ \ \varepsilon>0,\ \
| \boldsymbol{\xi}| \leqslant \delta_{0}(a,\mu).
\end{equation}
Here the constant  $C_5(a,\mu)$ is given by
\begin{equation*}
C_5(a,\mu):= \Bigl(\frac{3}{d_{0}} \Bigr)^{1/2} + \frac{C_{1}(a,\mu)}{(\mu_{-}C(a))^{1/2}}+\frac{C_{3}(a,\mu)}{(\mu_{-}C(a))^{3/2}},
\end{equation*}
and due to \eqref{C(a)}, \eqref{d0}, \eqref{C1} and \eqref{C3} it depends only on $d$, $\mu_-$, $\mu_+$, $M_1(a)$, $M_2(a)$, $M_3(a)$, ${\mathcal M}(a)$, ${\mathcal C}_\pi(a)$,
${\mathcal C}_{r(a)}(a)$.
\end{theorem}

\begin{proof}
From the definition of $F(\boldsymbol{\xi})$ it follows that
\begin{equation}\label{e2.38}
\|(\mathbb{A}(\boldsymbol{\xi}) +\varepsilon^{2}I)^{-1}(I-F(\boldsymbol{\xi}))\|\leqslant
\frac{3}{d_{0}},\ \ | \boldsymbol{\xi} | \leqslant \delta_{0}(a,\mu),\ \ \varepsilon >0.
\end{equation}
Hence,
\begin{equation}\label{e2.38a}
\|(\mathbb{A}( \boldsymbol{\xi})+\varepsilon^{2}I)^{-1}(I- F(\boldsymbol{\xi}))\|\leqslant
\Bigl( \frac{3}{d_{0}} \Bigr)^{1/2} \varepsilon^{-1},\ \ | \boldsymbol{\xi}| \leqslant \delta_{0}(a,\mu), \ \ \varepsilon >0.
\end{equation}

Relation  \eqref{Xi=le} directly implies that
\begin{equation}
\label{Xi=le_2}
\| \Xi(\boldsymbol{\xi},  \varepsilon) \| \leqslant
\left( \frac{C_{1}(a,\mu)}{(\mu_{-}C(a))^{1/2}}+\frac{C_{3}(a,\mu)}{(\mu_{-}C(a))^{3/2}} \right)
 \varepsilon^{-1},\quad | \boldsymbol{\xi} | \leqslant \delta_{0}(a,\mu), \quad \varepsilon>0.
\end{equation}

Obviously, the operator on the left-hand side of  \eqref{e2.37} is equal to
$$(\mathbb{A}(\boldsymbol{\xi})+\varepsilon^{2}I)^{-1}(I-F(\boldsymbol{\xi})) +\Xi(\boldsymbol{\xi},\varepsilon),$$
and therefore, estimate  \eqref{e2.37} follows from  \eqref{e2.38a} and \eqref{Xi=le_2}.
\end{proof}

Now we introduce the {\it effective operator\/}. It is a selfadjoint elliptic second-order differential operator
in $L_{2}(\mathbb{R}^d)$ with  constant coefficients given by
\begin{equation}
\label{eff_op}
\mathbb{A}^{0}:=\frac{1}{2}\sum_{k, l=1}^{d}g_{kl}D_{k}D_{l} = - \operatorname{div} g^0 \nabla,
\quad \operatorname{Dom} \mathbb{A}^0 = H^2(\mathbb{R}^d).
\end{equation}
The \emph{effective matrix} $g^0$ is a  $(d\times d)$-matrix with the entries  $\frac{1}{2}g_{kl}$,  where  the coefficients  $g_{kl}$, $k,l=1,\dots,d$, are defined by \eqref{Gkl===}.
Estimate \eqref{e2.36} shows that the matrix  $g^0$ is positive definite.

Using the unitary Gelfand transform, we decompose  $\mathbb{A}^0$ into the direct integral:
\begin{equation}
\label{direct_int}
\mathbb{A}^0 = {\mathcal G}^*\Bigl( \int_{\widetilde{\Omega}} \oplus \mathbb{A}^0( \boldsymbol{\xi})\,d\boldsymbol{\xi}  \Bigr) {\mathcal G}.
\end{equation}
Here $\mathbb{A}^0(\boldsymbol{\xi})$ is the selfadjoint operator in  $L_2(\Omega)$ given by
$$
\mathbb{A}^0(\boldsymbol{\xi}) = (\mathbf{D} + \boldsymbol{\xi})^* g^0 (\mathbf{D}+ \boldsymbol{\xi}),
\quad \operatorname{Dom} \mathbb{A}^0 (\boldsymbol{\xi}) = \widetilde{H}^2(\Omega).
$$
The space  $\widetilde{H}^2(\Omega)$ is defined  as a subspace of  $H^2(\Omega)$
consisting of functions whose  $\mathbb{Z}^d$-periodic extension to $\mathbb{R}^d$ belongs to
 $H^2_{\operatorname{loc}}(\mathbb{R}^d)$.
Relation  \eqref{direct_int} means the following. Let $u \in \operatorname{Dom} \mathbb{A}^0 = H^2(\mathbb{R}^d)$ and $v = \mathbb{A}^0 u$. Then
 $\mathcal{G}{u}(\boldsymbol{\xi},\cdot) \in \operatorname{Dom} \mathbb{A}^0(\boldsymbol{\xi}) = \widetilde{H}^2(\Omega)$ and
$\mathcal{G}{v}(\boldsymbol{\xi},\cdot) = \mathbb{A}^0( \boldsymbol{\xi}) \mathcal{G}{u}(\boldsymbol{\xi},\cdot)$, $\boldsymbol{\xi} \in \widetilde{\Omega}$.

Theorem \ref{teor2.2} easily implies the following result.

\begin{theorem}\label{teor3.3}
Suppose that conditions  \eqref{e1.1}--\eqref{e1.3} are satisfied and  $M_3(a) < \infty$.
Then
\begin{equation}\label{e2.37_2}
\left\|(\mathbb{A}(\boldsymbol{\xi} )+\varepsilon^{2}I)^{-1}- (\mathbb{A}^0(\boldsymbol{\xi})+\varepsilon^{2}I)^{-1}
\right\|_{L_2(\Omega) \to L_2(\Omega)}\leqslant
{\mathrm C}_1(a,\mu)\varepsilon^{-1},\ \ \varepsilon>0,\ \
\boldsymbol{\xi}  \in \widetilde{\Omega}.
\end{equation}
The constant   ${\mathrm C}_1(a,\mu)$ depends only on  $d,$ $\mu_-,$ $\mu_+,$ $M_1(a),$ $M_2(a),$ $M_3(a),$ ${\mathcal M}(a),$ ${\mathcal C}_\pi(a),$ ${\mathcal C}_{r(a)}(a)$.
\end{theorem}

\begin{proof}
Obviously, from relations \eqref{e2.39} and \eqref{e2.40} it follows  that
$$
\begin{aligned}
\left\|( \mathbb{A}(\boldsymbol{\xi})+\varepsilon^{2}I)^{-1} \right\| &\leqslant (\mu_- C(a))^{-1/2} (\delta_0(a,\mu))^{-1} \varepsilon^{-1},
\quad \varepsilon >0, \ \boldsymbol{\xi} \in \widetilde{\Omega}, \  |\boldsymbol{\xi}| \geqslant \delta_0(a,\mu),
\\
 \left( \langle g^0 \boldsymbol{\xi}, \boldsymbol{\xi} \rangle + \varepsilon^2 \right)^{-1} & \leqslant (\mu_- C(a))^{-1/2} (\delta_0(a,\mu))^{-1} \varepsilon^{-1},
\quad \varepsilon >0,\  \boldsymbol{\xi} \in \widetilde{\Omega}, \  | \boldsymbol{\xi} | \geqslant \delta_0(a,\mu).
\end{aligned}
$$
Combining this with Theorem  \ref{teor2.2} yields
\begin{equation}\label{e2.37_all}
\left\|( \mathbb{A}( \boldsymbol{\xi})+\varepsilon^{2}I)^{-1}-
(\langle g^0 \boldsymbol{\xi}, \boldsymbol{\xi} \rangle +\varepsilon^{2})^{-1} P \right\|\leqslant
\widetilde{C}_5(a,\mu)\varepsilon^{-1},\ \ \varepsilon>0,\ \  \boldsymbol{\xi} \in \widetilde{\Omega},
\end{equation}
where  $\widetilde{C}_5(a,\mu) = \max \{ C_5(a,\mu), 2 (\mu_- C(a))^{-1/2} (\delta_0(a,\mu))^{-1} \}$.

Obviously,
\begin{equation*}
\label{A0_P}
\mathbb{A}^0(\boldsymbol{\xi}) P = \langle g^0 \boldsymbol{\xi}, \boldsymbol{\xi} \rangle P,
\end{equation*}
whence
\begin{equation}
\label{res0_P}
( \mathbb{A}^0(\boldsymbol{\xi}) + \varepsilon^2 I)^{-1}P = \left( \langle g^0 \boldsymbol{\xi}, \boldsymbol{\xi} \rangle + \varepsilon^2 \right)^{-1} P.
\end{equation}
Let us rewrite  \eqref{e2.37_all} as
\begin{equation}
\label{e2.37_all_other}
\left\| ( \mathbb{A}(\boldsymbol{\xi})+\varepsilon^{2}I)^{-1}- ( \mathbb{A}^0(\boldsymbol{\xi})+\varepsilon^{2}I)^{-1}P
 \right\|\leqslant
\widetilde{C}_5(a,\mu)\varepsilon^{-1},\ \ \varepsilon>0,\ \ \boldsymbol{\xi} \in \widetilde{\Omega}.
\end{equation}

Using the discrete Fourier transform, we have
\begin{equation}
\label{discrete}
\left\| (\mathbb{A}^0(\boldsymbol{\xi})+\varepsilon^{2}I)^{-1}(I-P) \right\| = \sup_{0 \ne \mathbf{n} \in \mathbb{Z}^d}
(\langle g^0 (2\pi \mathbf{n}+ \boldsymbol{\xi}), 2\pi \mathbf{n} + \boldsymbol{\xi} \rangle + \varepsilon^2 )^{-1}
\leqslant (\mu_- C(a) \pi^2 + \varepsilon^2)^{-1}.
\end{equation}
We have taken into account \eqref{e2.36} and the obvious inequality
$|2\pi \mathbf{n} + \boldsymbol{\xi} | \geqslant \pi$ for
$\boldsymbol{\xi} \in \widetilde{\Omega}$ and $0 \ne \mathbf{n} \in \mathbb{Z}^d$. Consequently,
$$
\left\| (\mathbb{A}^0(\boldsymbol{\xi})+\varepsilon^{2}I)^{-1}(I-P) \right\| \leqslant (\mu_- C(a))^{-1/2} \pi^{-1} \varepsilon^{-1},
\ \ \varepsilon>0,\ \ \boldsymbol{\xi} \in \widetilde{\Omega}.
$$
Combining this with  \eqref{e2.37_all_other}, we obtain the required estimate \eqref{e2.37_2}
with the constant
$\mathrm{C}_1(a,\mu) = \widetilde{C}_5(a,\mu) + (\mu_- C(a))^{-1/2} \pi^{-1}$.
\end{proof}

   \subsection{More accurate approximation for the resolvent of the operator $\mathbb{A}(\boldsymbol{\xi})$}

\begin{theorem}\label{teor2.3}
Suppose that conditions  \eqref{e1.1}--\eqref{e1.3} are satisfied and  $M_4(a) < \infty$. Then
\begin{multline}
\label{e2.37_corr}
(\mathbb{A} (\boldsymbol{\xi})+\varepsilon^{2}I)^{-1} =
(\langle g^0 \boldsymbol{\xi}, \boldsymbol{\xi} \rangle +\varepsilon^{2})^{-1} P
+ [F]_1( \boldsymbol{\xi} ) (\langle g^0 \boldsymbol{\xi}, \boldsymbol{\xi} \rangle +\varepsilon^{2})^{-1} P
\\
+ (\langle g^0 \boldsymbol{\xi}, \boldsymbol{\xi} \rangle
+ \varepsilon^{2})^{-1} P [F]_1(\boldsymbol{\xi}) + Y(\boldsymbol{\xi},\varepsilon)
\end{multline}
and
\begin{equation}\label{e2.37_corr_2}
\left\| Y(\boldsymbol{\xi}, \varepsilon) \right\|\leqslant
C_6(a,\mu),\ \ \varepsilon>0,\ \
| \boldsymbol{\xi} |\leqslant \delta_{0}(a,\mu).
\end{equation}
Here $[F]_1(\boldsymbol{\xi}) = \sum_{j=1}^d F_j \xi_j$, and the operators $F_j$ are given by  \eqref{Fj===}. The constant  $C_6(a,\mu)$ is given by
$$
C_6(a,\mu):= \frac{3}{d_{0}} +
 \frac{2 C_1(a,\mu)^2 + 2 C_2(a,\mu)}{\mu_{-}C(a)}
% \\
 +  \frac{3 C_1(a,\mu) C_3(a,\mu)+ C_4(a,\mu)}{(\mu_{-}C(a))^{2}}
+  \frac{ C_3(a,\mu)^2}{(\mu_{-}C(a))^{3}},
$$
according to \eqref{C(a)}, \eqref{d0}, \eqref{C1}, \eqref{C2}, \eqref{C3} and \eqref{C4} it depends only on  $d,$ $\mu_-,$ $\mu_+,$ $M_1(a),$ $M_2(a),$ $M_3(a),$ $M_4(a),$ ${\mathcal M}(a),$ ${\mathcal C}_\pi(a),$ ${\mathcal C}_{r(a)}(a)$.
\end{theorem}

\begin{proof}
 Replacing the operator $F(\boldsymbol{\xi})(\mathbb{A} (\boldsymbol{\xi})+\varepsilon^{2}I)^{-1}$
  with $(\langle g^0 \boldsymbol{\xi}, \boldsymbol{\xi} \rangle +\varepsilon^{2})^{-1} P + \Xi(\boldsymbol{\xi}, \varepsilon)$ in the first and third terms
 on the right-hand side of \eqref{e2.41}  we obtain
 \begin{equation}
 \label{th2_1}
\Xi( \boldsymbol{\xi}, \varepsilon) = \widetilde{\Xi}(\boldsymbol{\xi},\varepsilon) + Y_1(\boldsymbol{\xi}, \varepsilon),
\end{equation}
where
\begin{multline}
\label{th2_2}
\widetilde{\Xi}(\boldsymbol{\xi},\varepsilon) = (\langle g^0 \boldsymbol{\xi}, \boldsymbol{\xi} \rangle +\varepsilon^{2})^{-1} P (F(\boldsymbol{\xi})-P)
+(F(\boldsymbol{\xi})-P) (\langle g^0 \boldsymbol{\xi}, \boldsymbol{\xi} \rangle +\varepsilon^{2})^{-1} P
\\
-
(\langle g^0 \boldsymbol{\xi}, \boldsymbol{\xi} \rangle +\varepsilon^{2})^{-1} P
\left(\mathbb{A}(\boldsymbol{\xi})F(\boldsymbol{\xi}) - \langle g^0 \boldsymbol{\xi}, \boldsymbol{\xi} \rangle P \right)
(\langle g^0 \boldsymbol{\xi}, \boldsymbol{\xi} \rangle +\varepsilon^{2})^{-1} P
\end{multline}
and
\begin{equation}
\label{th2_3}
Y_1(\boldsymbol{\xi},\varepsilon) =
\Xi(\boldsymbol{\xi}, \varepsilon) (F(\boldsymbol{\xi})-P)
- \Xi(\boldsymbol{\xi}, \varepsilon) \Psi(\boldsymbol{\xi})
(\langle g^0 \boldsymbol{\xi},\boldsymbol{\xi} \rangle +\varepsilon^{2})^{-1} P.
\end{equation}
Here the operator $\Psi(\boldsymbol{\xi})$ is defined by  \eqref{AF = G2}.
Relations \eqref{F-P}, \eqref{AF-G}, \eqref{Xi=le}, and \eqref{e2.40} imply the following estimate for the operator \eqref{th2_3}:
\begin{multline}
\label{th2_4}
\|Y_1(\boldsymbol{\xi}, \varepsilon) \|
\leqslant
\left( \frac{2 C_1(a,\mu) | \boldsymbol{\xi}|}{\mu_{-}C(a)| \boldsymbol{\xi}|^{2} + \varepsilon^2} + \frac{C_3(a,\mu) |\boldsymbol{\xi}|^3}{(\mu_{-}C(a)|\boldsymbol{\xi}|^{2} + \varepsilon^2)^2} \right)
\left( C_1(a,\mu) | \boldsymbol{\xi}| + \frac{C_3(a,\mu) |\boldsymbol{\xi}|^3}{\mu_{-}C(a)| \boldsymbol{\xi}|^{2} + \varepsilon^2}\right)
\\
\leqslant \frac{2 C_1(a,\mu)^2}{\mu_{-}C(a)} +  \frac{3 C_1(a,\mu) C_3(a,\mu)}{(\mu_{-}C(a))^{2}}
+  \frac{ C_3(a,\mu)^2}{(\mu_{-}C(a))^{3}}, \quad |\boldsymbol{\xi} |\leqslant \delta_{0}(a,\mu),\  \varepsilon >0.
\end{multline}

Now, substituting expansion \eqref{F= P+ F1} into the first two  terms on the right-hand side of  \eqref{th2_2}
and expansion \eqref{AF= G + G3} into the third term, we obtain
\begin{equation}
 \label{th2_5}
\widetilde{\Xi}(\boldsymbol{\xi}, \varepsilon) = {\Xi}^0(\boldsymbol{\xi}, \varepsilon) + Y_2(\boldsymbol{\xi},\varepsilon),
\end{equation}
where
\begin{multline}
\label{th2_6}
{\Xi}^0(\boldsymbol{\xi}, \varepsilon) = (\langle g^0 \boldsymbol{\xi}, \boldsymbol{\xi} \rangle +\varepsilon^{2})^{-1} P [F]_1(\boldsymbol{\xi})
+[F]_1(\boldsymbol{\xi}) (\langle g^0 \boldsymbol{\xi}, \boldsymbol{\xi} \rangle +\varepsilon^{2})^{-1} P
\\
-
(\langle g^0 \boldsymbol{\xi}, \boldsymbol{\xi} \rangle +\varepsilon^{2})^{-1} P  [G]_3(\boldsymbol{\xi})
(\langle g^0 \boldsymbol{\xi}, \boldsymbol{\xi} \rangle +\varepsilon^{2})^{-1} P,
\end{multline}
\begin{multline}
\label{th2_7}
Y_2(\boldsymbol{\xi}, \varepsilon) = (\langle g^0 \boldsymbol{\xi}, \boldsymbol{\xi} \rangle +\varepsilon^{2})^{-1} P \Phi(\boldsymbol{\xi})
+  \Phi(\boldsymbol{\xi}) (\langle g^0 \boldsymbol{\xi}, \boldsymbol{\xi} \rangle +\varepsilon^{2})^{-1} P
\\ -
(\langle g^0 \boldsymbol{\xi}, \boldsymbol{\xi} \rangle +\varepsilon^{2})^{-1} P  \Upsilon(\boldsymbol{\xi})(\langle g^0 \boldsymbol{\xi}, \boldsymbol{\xi} \rangle +\varepsilon^{2})^{-1} P.
\end{multline}
Relations \eqref{F-P-O(xi)}, \eqref{AF-G - G3}, and
\eqref{e2.40} imply the following estimate for the operator \eqref{th2_7}:
\begin{multline}
\label{th2_8}
\|Y_2( \boldsymbol{\xi}, \varepsilon) \|
\leqslant
\frac{2 C_2(a,\mu) | \boldsymbol{\xi}|^2}{\mu_{-}C(a) |\boldsymbol{\xi} |^{2} + \varepsilon^2} + \frac{C_4(a,\mu)
|\boldsymbol{\xi} |^4}{(\mu_{-}C(a) | \boldsymbol{\xi} |^{2} + \varepsilon^2)^2}
\\
\leqslant \frac{2 C_2(a,\mu)}{\mu_{-}C(a)} +  \frac{C_4(a,\mu)}{(\mu_{-}C(a))^{2}}, \quad
 | \boldsymbol{\xi} |\leqslant \delta_{0}(a,\mu),\quad \varepsilon >0.
\end{multline}
By Proposition \ref{prop2.9}, the third term in the right-hand side of  \eqref{th2_6} is equal to zero, whence
\begin{equation}
\label{th2_9}
{\Xi}^0(\boldsymbol{\xi}, \varepsilon) = (\langle g^0 \boldsymbol{\xi}, \boldsymbol{\xi} \rangle +\varepsilon^{2})^{-1} P [F]_1(\boldsymbol{\xi})
+[F]_1(\boldsymbol{\xi}) (\langle g^0 \boldsymbol{\xi}, \boldsymbol{\xi} \rangle +\varepsilon^{2})^{-1} P.
\end{equation}

Combining  \eqref{Xi=}, \eqref{th2_1}, \eqref{th2_5}, and \eqref{th2_9}, we arrive at representation \eqref{e2.37_corr}
with
$$
Y(\boldsymbol{\xi},\varepsilon) = (\mathbb{A}(\boldsymbol{\xi})+\varepsilon^{2}I)^{-1} (I-F(\boldsymbol{\xi})) +
Y_1(\boldsymbol{\xi}, \varepsilon) + Y_2(\boldsymbol{\xi},\varepsilon).
$$
Relations \eqref{e2.38}, \eqref{th2_4}, and \eqref{th2_8} imply estimate  \eqref{e2.37_corr_2}.
\end{proof}

From Theorem \ref{teor2.3} with the help of Corollary  \ref{cor5.5} we deduce the following result.

\begin{theorem}\label{teor3.5}
Suppose that conditions  \eqref{e1.1}--\eqref{e1.3} are satisfied and $M_4(a) < \infty$.
Then
\begin{equation}\label{th3.5_1}
\left\| (\mathbb{A} (\boldsymbol{\xi})+\varepsilon^{2}I)^{-1} - (\mathbb{A}^0(\boldsymbol{\xi})+\varepsilon^{2}I)^{-1} - K(\boldsymbol{\xi}, \varepsilon) \right\|_{L_2(\Omega) \to L_2(\Omega)}
\leqslant {\mathrm C}_2(a,\mu)
\end{equation}
for $\varepsilon>0$ and $\boldsymbol{\xi} \in \widetilde{\Omega}$. Here
\begin{equation*}\label{th3.5_2}
K(\boldsymbol{\xi}, \varepsilon) :=  - i \sum_{j=1}^d [v_j] (D_j + \xi_j)( \mathbb{A}^0(\boldsymbol{\xi})+\varepsilon^{2}I)^{-1} +
i \sum_{j=1}^d ( \mathbb{A}^0(\boldsymbol{\xi})+\varepsilon^{2}I)^{-1} (D_j + \xi_j) [v_j],
\end{equation*}
the functions $v_j(\mathbf{x})$ are solutions of problems \eqref{e2.23}\emph{;}
the symbol $[v_j]$ stands for the operator of multiplication by the function  $v_j(\mathbf{x})$.
The constant $\mathrm{C}_2(a,\mu)$ depends only on   $d,$ $\mu_-,$ $\mu_+,$ $M_1(a),$ $M_2(a),$ $M_3(a),$ $M_4(a),$ ${\mathcal M}(a),$ ${\mathcal C}_\pi(a),$ ${\mathcal C}_{r(a)}(a),$
and the constant $\mathfrak{C}(\widetilde{a},\mu)$ defined by  \eqref{5.24}.
\end{theorem}

\begin{proof}
Obviously, from \eqref{[F]1j=le},  \eqref{e2.39}, and \eqref{e2.40} it follows that
$$
\begin{aligned}
&\left\|(\mathbb{A}(\boldsymbol{\xi})+\varepsilon^{2}I)^{-1} \right\| \leqslant (\mu_- C(a))^{-1} (\delta_0(a,\mu))^{-2},
\quad \varepsilon >0,\  \boldsymbol{\xi}  \in \widetilde{\Omega},\  | \boldsymbol{\xi} | \geqslant \delta_0(a,\mu),
\\
 &\left( \langle g^0 \boldsymbol{\xi}, \boldsymbol{\xi} \rangle + \varepsilon^2 \right)^{-1}  \leqslant (\mu_- C(a))^{-1} (\delta_0(a,\mu))^{-2},
  \quad  \varepsilon >0,
\  \boldsymbol{\xi} \in \widetilde{\Omega}, \  | \boldsymbol{\xi}| \geqslant \delta_0(a,\mu),
\end{aligned}
$$
and
$$
\begin{aligned}
\| [F]_1(\boldsymbol{\xi})\| \left( \langle g^0 \boldsymbol{\xi}, \boldsymbol{\xi} \rangle + \varepsilon^2 \right)^{-1}  \leqslant  C_1(a,\mu) (\mu_- C(a) \delta_0(a,\mu))^{-1},
\\
 \varepsilon >0,\  \boldsymbol{\xi}  \in \widetilde{\Omega}, \  | \boldsymbol{\xi}| \geqslant \delta_0(a,\mu).
\end{aligned}
$$
Together with Theorem \ref{teor2.3} this yields
\begin{multline}
\label{3.32}
\Bigl\| (\mathbb{A} ( \boldsymbol{\xi})+\varepsilon^{2}I)^{-1} - \left( \langle g^0 \boldsymbol{\xi}, \boldsymbol{\xi} \rangle + \varepsilon^2 \right)^{-1} P
- [F]_1( \boldsymbol{\xi})\left( \langle g^0 \boldsymbol{\xi}, \boldsymbol{\xi} \rangle + \varepsilon^2 \right)^{-1} P
\\
-
\left( \langle g^0 \boldsymbol{\xi}, \boldsymbol{\xi} \rangle + \varepsilon^2 \right)^{-1} P [F]_1(\boldsymbol{\xi}) \Bigr\|
\leqslant \widetilde{C}_6(a,\mu), \quad \varepsilon >0,\quad  \boldsymbol{\xi} \in \widetilde{\Omega},
\end{multline}
where
$$\widetilde{C}_6(a,\mu) \!=\! \max \bigl\{ {C}_6(a,\mu), 2 (\mu_- C(a))^{-1} (\delta_0(a,\mu))^{-2} \!+\! 2 C_1(a,\mu) (\mu_- C(a) \delta_0(a,\mu))^{-1}\bigr\}.$$

Now, we take into account  \eqref{res0_P} and the relation
$$
[F]_1(\boldsymbol{\xi}) P = \sum_{j=1}^d F_j \xi_j P = - i \sum_{j=1}^d [v_j] \xi_j P =  - i \sum_{j=1}^d [v_j] (D_j + \xi_j) P;
$$
see \eqref{F= P+ F1} and \eqref{Fj===}. Then
$$
[F]_1(\boldsymbol{\xi}) \left( \langle g^0 \boldsymbol{\xi}, \boldsymbol{\xi} \rangle + \varepsilon^2 \right)^{-1} P
 = - i \sum_{j=1}^d [v_j] (D_j+\xi_j) (\mathbb{A}^0(\boldsymbol{\xi}) + \varepsilon^2 I)^{-1} P.
$$
Let us rewrite  \eqref{3.32} as
\begin{multline}
\label{3.33}
\Bigl\| ( \mathbb{A}(\boldsymbol{\xi})+\varepsilon^{2}I)^{-1} - ( \mathbb{A}^0(\boldsymbol{\xi}) + \varepsilon^2 I)^{-1} P
+ i \sum_{j=1}^d [v_j] (D_j+\xi_j) (\mathbb{A}^0(\boldsymbol{\xi}) + \varepsilon^2 I)^{-1} P
\\
- i \sum_{j=1}^d ( \mathbb{A}^0(\boldsymbol{\xi}) + \varepsilon^2 I)^{-1}  (D_j+\xi_j) P [v_j] \Bigr\|
\leqslant \widetilde{C}_6(a,\mu), \quad \varepsilon >0,\quad \boldsymbol{\xi} \in \widetilde{\Omega}.
\end{multline}
By \eqref{discrete},
\begin{equation}
\label{discrete2}
\left\| (\mathbb{A}^0(\boldsymbol{\xi})+\varepsilon^{2}I)^{-1}(I-P) \right\|
\leqslant (\mu_- C(a))^{-1} \pi^{-2}.
\end{equation}
Using the discrete Fourier transform and Corollary \ref{cor5.5}, we conclude that
\begin{multline}
\label{discrete3}
\Bigl\| \sum_{j=1}^d [v_j] (D_j+\xi_j) (\mathbb{A}^0(\boldsymbol{\xi}) + \varepsilon^2 I)^{-1} (I - P) \Bigr\|
\\
\leqslant \sum_{j=1}^d \| v_j \|_{L_\infty}  \sup_{0 \ne \mathbf{n} \in \mathbb{Z}^d}
| 2\pi n_j + \xi_j| \left(\langle g^0 (2\pi \mathbf{n} + \boldsymbol{\xi}), 2\pi \mathbf{n} +\boldsymbol{\xi} \rangle +\varepsilon^2 \right)^{-1}
\\
\leqslant \Bigl( \sum_{j=1}^d  \| v_j \|_{L_\infty}^2\Bigr)^{1/2} (\mu_- C(a) \pi)^{-1}.
\end{multline}
Finally, relations  \eqref{3.33}--\eqref{discrete3} imply the required estimate \eqref{th3.5_1}
with the constant
$$
{\mathrm C}_2(a,\mu) = \widetilde{C}_6(a,\mu) +  (\mu_- C(a))^{-1} \pi^{-2} + 2 \Bigl( \sum_{j=1}^d  \| v_j \|_{L_\infty}^2\Bigr)^{1/2} (\mu_- C(a) \pi)^{-1}.
$$
Note that the norms  $\| v_j \|_{L_\infty}$ satisfy estimates \eqref{5.26}.
\end{proof}

\subsection{Approximation of the resolvent  $(\mathbb{A} +\varepsilon^{2}I)^{-1}$}

\begin{theorem}\label{teor3.6}
Suppose that conditions  \eqref{e1.1}--\eqref{e1.3} are satisfied and  $M_3(a) < \infty$. Then
\begin{equation}\label{th3.6_1}
\left\|(\mathbb{A}+\varepsilon^{2}I)^{-1}- (\mathbb{A}^0+\varepsilon^{2}I)^{-1}
\right\|_{L_2(\mathbb{R}^d) \to L_2(\mathbb{R}^d)}\leqslant
{\mathrm C}_1(a,\mu)\varepsilon^{-1},\ \ \varepsilon>0.
\end{equation}
\end{theorem}
\begin{proof}
From decompositions \eqref{e1.8} and \eqref{direct_int} one deduces,
with the help of the Gelfand transform, that  the operator inside the norm signs in \eqref{th3.6_1}
can be decomposed into the direct integral of the operators
$(\mathbb{A}(\boldsymbol{\xi}) +\varepsilon^{2}I)^{-1}- ( \mathbb{A}^0(\boldsymbol{\xi})+\varepsilon^{2}I)^{-1}.$
Hence,
\begin{multline*}
\left\|( \mathbb{A} +\varepsilon^{2}I)^{-1}- ( \mathbb{A}^0+\varepsilon^{2}I)^{-1}
\right\|_{L_2(\mathbb{R}^d) \to L_2(\mathbb{R}^d)}
\\
=
\sup_{\boldsymbol{\xi} \in \widetilde{\Omega}} \left\|( \mathbb{A}(\boldsymbol{\xi})+\varepsilon^{2}I)^{-1}- ( \mathbb{A}^0(\boldsymbol{\xi})+\varepsilon^{2}I)^{-1}
\right\|_{L_2(\Omega) \to L_2(\Omega)}.
\end{multline*}
Therefore, estimate  \eqref{th3.6_1} follows from  \eqref{e2.37_2}.
\end{proof}

\begin{theorem}\label{teor3.7}
Suppose that conditions  \eqref{e1.1}--\eqref{e1.3} are satisfied and  $M_4(a) < \infty$.
Then
\begin{equation}\label{th3.7_1}
\left\| ( \mathbb{A}+\varepsilon^{2}I)^{-1} - (\mathbb{A}^0+\varepsilon^{2}I)^{-1} - K(\varepsilon) \right\|_{L_2(\mathbb{R}^d) \to L_2(\mathbb{R}^d)}
\leqslant {\mathrm C}_2(a,\mu), \ \ \varepsilon>0.
\end{equation}
Here
\begin{equation}\label{th3.7_2}
K(\varepsilon) :=  -  \sum_{j=1}^d [v_j] \partial_j (\mathbb{A}^0+\varepsilon^{2}I)^{-1} +
 \sum_{j=1}^d (\mathbb{A}^0+\varepsilon^{2}I)^{-1} \partial_j  [v_j],
\end{equation}
the symbol $[v_j]$ stands for the operator  of multiplication by the function $v_j(\mathbf{x})$.
\end{theorem}

\begin{proof}
With the help of the Gelfand transform, the operator inside the norm signs in  \eqref{th3.7_1}
can be decomposed into the direct integral of the operators
$$(\mathbb{A}(\boldsymbol{\xi}) +\varepsilon^{2}I)^{-1}- ( \mathbb{A}^0(\boldsymbol{\xi})+\varepsilon^{2}I)^{-1} - K(\boldsymbol{\xi}, \varepsilon).
$$
Hence,
\begin{multline*}
\left\|( \mathbb{A}+\varepsilon^{2}I)^{-1}- ( \mathbb{A}^0+\varepsilon^{2}I)^{-1} - K(\varepsilon)
\right\|_{L_2(\mathbb{R}^d) \to L_2(\mathbb{R}^d)}
\\
=
\sup_{\boldsymbol{\xi} \in \widetilde{\Omega}} \left\|( \mathbb{A}(\boldsymbol{\xi})+\varepsilon^{2}I)^{-1}- ( \mathbb{A}^0(\boldsymbol{\xi})+\varepsilon^{2}I)^{-1}- K(\boldsymbol{\xi},\varepsilon)
\right\|_{L_2(\Omega) \to L_2(\Omega)}.
\end{multline*}
Combining this with  \eqref{th3.5_1}, we obtain estimate  \eqref{th3.7_1}.
\end{proof}

\section{Homogenization of nonlocal convolution type operator}\label{Sec4}

\subsection{Main results}
Assuming  that conditions \eqref{e1.1}--\eqref{e1.3} are satisfied,
 consider the family of nonlocal operators in  $L_{2}(\mathbb{R}^d)$ given by
$$
\mathbb{A}_{\varepsilon}u(\mathbf{x}):=\varepsilon^{-d-2}\intop_{\mathbb{R}^d}a((\mathbf{x}-\mathbf{y})/\varepsilon)\mu(\mathbf{x} /\varepsilon,\mathbf{y}/\varepsilon)
(u(\mathbf{x})-u(\mathbf{y}))\,d\mathbf{y},
\quad \mathbf{x} \in \mathbb{R}^d,\ \; u\in L_{2}(\mathbb{R}^d),\ \; \varepsilon>0.
$$
Let $\mathbb{A}^0$ be the effective operator in  $L_{2}(\mathbb{R}^d)$ defined by \eqref{eff_op}.
Recall that the effective matrix  $g^0$ is the matrix with the entries  $\frac{1}{2}g_{kl}$, where the coefficients  $g_{kl}$, $k,l=1,\dots,d$, are defined in  \eqref{Gkl===}.

Making the scaling transformation and applying Theorem \ref{teor3.6} we obtain the following result, which was proved in  \cite[Theorem 4.1]{PSlSuZh}; for completeness, we  provide a proof here.

\begin{theorem}
\label{teor3.1}
Let conditions \eqref{e1.1}--\eqref{e1.3} be fulfilled, and assume that {$M_3(a) < \infty$}. Then
\begin{equation}\label{e3.1}
\|( \mathbb{A}_{\varepsilon}+I)^{-1} - (\mathbb{A}^{0}+I)^{-1}\|_{L_2(\mathbb{R}^d) \to L_2(\mathbb{R}^d)} \leqslant
{\mathrm C}_1(a,\mu)\varepsilon,\ \ \varepsilon>0.
\end{equation}
The constant   ${\mathrm C}_1(a,\mu)$ depends only on $d,$ $\mu_-,$ $\mu_+,$
$M_1(a),$ $M_2(a),$ $M_3(a),$ ${\mathcal M}(a),$ ${\mathcal C}_\pi(a),$ ${\mathcal C}_{r(a)}(a)$.
\end{theorem}

\begin{proof}
Let us introduce the scaling transformation  (a family of unitary operators):
\begin{equation*}
T_{\varepsilon}u(\mathbf{x}):=\varepsilon^{d/2}u(\varepsilon \mathbf{x}),\ \
\mathbf{x} \in \mathbb{R}^d,\ \ u\in L_{2}(\mathbb{R}^d), \ \ \varepsilon >0.
\end{equation*}
It is easily seen that
$$
\mathbb{A}_{\varepsilon} = \varepsilon^{-2} T_{\varepsilon}^* \mathbb{A} T_{\varepsilon},\quad \varepsilon >0.
$$
Hence,
\begin{equation}\label{e3.2}
( \mathbb{A}_{\varepsilon}+I)^{-1}= T_{\varepsilon}^{*}\varepsilon^{2}( \mathbb{A} +\varepsilon^{2}I)^{-1}T_{\varepsilon},
\ \ \varepsilon>0.
\end{equation}
The effective operator also satisfies the relation
$$
\mathbb{A}^0 = \varepsilon^{-2} T_{\varepsilon}^* \mathbb{A}^0 T_{\varepsilon},\quad \varepsilon >0,
$$
whence
\begin{equation}\label{e3.2_eff}
( \mathbb{A}^0+I)^{-1}= T_{\varepsilon}^{*}\varepsilon^{2}( \mathbb{A}^0+\varepsilon^{2}I)^{-1}T_{\varepsilon},
\ \ \varepsilon>0.
\end{equation}
 Since  $T_\varepsilon$ is unitary, from
\eqref{e3.2} and \eqref{e3.2_eff} we deduce that
$$
\|( \mathbb{A}_{\varepsilon}+I)^{-1}-( \mathbb{A}^{0}+I)^{-1}\|_{L_2(\mathbb{R}^d) \to L_2(\mathbb{R}^d)}
=
\varepsilon^2 \|( \mathbb{A} + \varepsilon^2 I)^{-1}-(\mathbb{A}^{0}+ \varepsilon^2 I)^{-1}\|_{L_2(\mathbb{R}^d) \to L_2(\mathbb{R}^d)}.
$$
Combining this with Theorem  \ref{teor3.6}, we obtain the required estimate \eqref{e3.1}.
\end{proof}

Our \emph{main new result} is the following theorem.

\begin{theorem}
\label{teor3.2}
Let conditions  \eqref{e1.1}--\eqref{e1.3} be satisfied, and assume that $M_4(a) < \infty$. Then
\begin{equation}\label{teor3.2_1}
\|( \mathbb{A}_{\varepsilon}+I)^{-1}-(\mathbb{A}^{0}+I)^{-1} - \varepsilon K_\varepsilon \|_{L_2(\mathbb{R}^d) \to L_2(\mathbb{R}^d)} \leqslant
{\mathrm C}_2(a,\mu)\varepsilon^2,\ \ \varepsilon>0;
\end{equation}
here the corrector $K_\varepsilon$ is given by
$$
K_\varepsilon = -  \sum_{j=1}^d [v_j^\varepsilon] \partial_j (\mathbb{A}^0+ I)^{-1} +
 \sum_{j=1}^d ( \mathbb{A}^0+ I)^{-1} \partial_j  [v_j^\varepsilon],
$$
the functions $v_j(\mathbf{x})$ are solutions of problems \eqref{e2.23}\emph{;}
 $[v_j^\varepsilon]$ is the operator of multiplication by the function $v_j(\mathbf{x} / \varepsilon)$.
The constant  ${\mathrm C}_2(a,\mu)$ depends only on  $d,$ $\mu_-,$ $\mu_+,$ $M_1(a),$ $M_2(a),$ $M_3(a),$ $M_4(a),$ ${\mathcal M}(a),$ ${\mathcal C}_\pi(a),$ ${\mathcal C}_{r(a)}(a),$
and $\mathfrak{C}(\widetilde{a},\mu)$.
\end{theorem}

\begin{proof}
Combining  \eqref{e3.2}, \eqref{e3.2_eff} and the relation
\begin{equation*}
K_\varepsilon  = T_{\varepsilon}^{*}\varepsilon K(\varepsilon) T_{\varepsilon},
\ \ \varepsilon>0,
\end{equation*}
where $K( \varepsilon)$ is the operator \eqref{th3.7_2},
and considering the fact that $T_\varepsilon$ is unitary, we obtain
\begin{multline*}
\|( \mathbb{A}_{\varepsilon}+I)^{-1}-(\mathbb{A}^{0}+I)^{-1} - \varepsilon K_\varepsilon  \|_{L_2(\mathbb{R}^d) \to L_2(\mathbb{R}^d)}
\\
=\varepsilon^2 \|( \mathbb{A} + \varepsilon^2 I)^{-1}-( \mathbb{A}^{0}+ \varepsilon^2 I)^{-1} - K( \varepsilon)\|_{L_2(\mathbb{R}^d) \to L_2(\mathbb{R}^d)}.
\end{multline*}
Together with  Theorem \ref{teor3.7} this implies the required estimate \eqref{teor3.2_1}.
\end{proof}

\subsection{Concluding remarks}

1. Theorems \ref{teor3.1} and \ref{teor3.2} remain valid if the lattice of periods  $\mathbb{Z}^d$ is replaced with an arbitrary periodic lattice in  $\mathbb{R}^d$. Then the constants in the above estimates depend  not only on the coefficients
 $a$ and $\mu$, but also on the lattice parameters.

2. In \cite[Sec. 4]{PSlSuZh}, under the additional assumption $a \in L_2(\mathbb{R}^d)$,
 estimates of the quantities  ${\mathcal M}(a)$, ${\mathcal C}_\pi(a)$, ${\mathcal C}_{r(a)}(a)$
were obtained, from which it follows that  ${\mathrm C}_1(a,\mu)$ does not exceed a constant depending on
 the parameters $d$, $\mu_-$, $\mu_+$, $\|a\|_{L_1}$, $\|a\|_{L_2}$,
$M_1(a)$, $M_2(a)$, $M_3(a)$. Below in Remark  \ref{rem5.6} it is noted that under the additional assumption
$\widetilde{a} = \widetilde{a}(\mathbf{0}, \cdot) \in L_2(\Omega)$ the norms $\|v_j\|_{L_\infty}$, $j=1,\dots,d,$ are controlled in terms of  $\mu_-$, $\mu_+$, $\|a\|_{L_1}$,  $\| \widetilde{a}\|_{L_2(\Omega)},$ ${\mathcal C}_\pi(a)$, and $M_1(a)$. Thus, under the additional assumption  $\widetilde{a} \in L_2(\Omega)$
(in this case condition $a \in L_2(\mathbb{R}^d)$ is automatically satisfied)
  ${\mathrm C}_2(a,\mu)$ does not exceed a constant depending on  $d$, $\mu_-$, $\mu_+$, $\|a\|_{L_1}$,
  $\|\widetilde{a}\|_{L_2(\Omega)},$
$M_1(a)$, $M_2(a)$, $M_3(a)$,  $M_4(a)$.

3. If in the assumptions of Theorem \ref{teor3.1}  condition $M_3(a)<\infty$ is replaced with
\begin{equation}
\label{dop}
\intop_{\mathbb{R}^d} | \mathbf{x}|^k a( \mathbf{x}) \,d \mathbf{x} < \infty,
\end{equation}
where $2 < k < 3$, then we have
$$
\|(  \mathbb{A}_{\varepsilon}+I)^{-1}-( \mathbb{A}^{0}+I)^{-1}  \|_{L_2(\mathbb{R}^d) \to L_2(\mathbb{R}^d)} \leqslant C \varepsilon^{k-2}, \quad  \varepsilon >0.
$$
If in the assumptions of Theorem \ref{teor3.2} condition $M_4(a)<\infty$ is replaced by
\eqref{dop}  with  \hbox{$3 < k < 4$}, then we have
$$
\|( \mathbb{A}_{\varepsilon}+I)^{-1}-( \mathbb{A}^{0}+I)^{-1} - \varepsilon K_\varepsilon \|_{L_2(\mathbb{R}^d) \to L_2(\mathbb{R}^d)} \leqslant C \varepsilon^{k-2}, \quad \varepsilon >0.
$$

\section{Appendix: boundedness of solutions of the auxiliary problems}\label{Sec5}

\subsection{Approximation of the operator $\mathbb{A}(\mathbf{0})$}
Consider a function
\begin{equation}
\label{5.1}
\widetilde{a}(\mathbf{z}) := \widetilde{a}(\mathbf{0},\mathbf{z}) = \sum_{\mathbf{n} \in \mathbb{Z}^d} a(\mathbf{z}+\mathbf{n}), \quad \mathbf{z} \in \mathbb{R}^d,
\end{equation}
see \eqref{a_tilde}. From \eqref{e1.1} and \eqref{5.1} it follows that  $\widetilde{a}(\mathbf{z})$ is a nonnegative
$\mathbb{Z}^d$-periodic function such that
$\widetilde{a}(\mathbf{z}) = \widetilde{a}(-\mathbf{z})$, $\mathbf{z} \in \mathbb{R}^d$; $\widetilde{a} \in L_1(\Omega)$, and
\begin{equation}
\label{5.2}
 \| \widetilde{a} \|_{L_1(\Omega)} = \| a \|_{L_1(\mathbb{R}^d)} >0.
 \end{equation}

 For each $N \in \mathbb{N}$ we put
\begin{equation}
\label{5.3}
  \widetilde{a}_N(\mathbf{z}) :=
  \begin{cases}
  \widetilde{a}(\mathbf{z}), & \text{if}\ \widetilde{a}(\mathbf{z}) \leqslant N,
  \\
  N, & \text{if}\ \widetilde{a}(\mathbf{z}) > N.\end{cases}
  \end{equation}
 Obviously, $\widetilde{a}_N(\mathbf{z})$ is a  $\mathbb{Z}^d$-periodic function such that
 $$
 0 \leqslant \widetilde{a}_N(\mathbf{z}) \leqslant \widetilde{a}(\mathbf{z}),\quad \widetilde{a}_N(\mathbf{z}) = \widetilde{a}_N(- \mathbf{z}),\quad \mathbf{z} \in \mathbb{R}^d,\ \ N \in \mathbb{N},
 $$
 and
 $$
 \lim_{N \to \infty} \widetilde{a}_N(\mathbf{z}) = \widetilde{a}(\mathbf{z}),\quad \mathbf{z} \in \mathbb{R}^d.
 $$
Then, by the Lebesgue theorem,
\begin{equation}
\label{5.4}
  \| \widetilde{a}_N - \widetilde{a} \|_{L_1(\Omega)} \to 0 \quad \text{as} \ N \to \infty.
  \end{equation}

According to  \eqref{5.2},   $\widetilde{a}(\mathbf{z}) >0$ on the set of a positive measure on $\Omega$.
By \eqref{5.3},  $\widetilde{a}_N(\mathbf{z}) >0$ on the same set  (for any $N \in \mathbb{N}$).
Moreover,
\begin{equation}
\label{5.4a}
\widetilde{a}_N(\mathbf{z}) \leqslant \widetilde{a}_{N+1}(\mathbf{z}), \quad \mathbf{z} \in \mathbb{R}^d, \quad N \in \mathbb{N}.
  \end{equation}
 Consequently,
\begin{equation}
\label{5.5}
 0< \| \widetilde{a}_1 \|_{L_1(\Omega)} \leqslant \| \widetilde{a}_N \|_{L_1(\Omega)}
 \leqslant \| \widetilde{a} \|_{L_1(\Omega)}, \quad N \in\mathbb{N}.
  \end{equation}

Now, let us consider a bounded selfadjoint operator $\mathbf{A}:= \mathbb{A}(\mathbf{0})$ in
$L_2(\Omega)$ defined according to \eqref{A_xi}, \eqref{B_xi}:
\begin{equation*}
\label{A(0)}
\begin{aligned}
\mathbf{A} u (\mathbf{x}) &= p(\mathbf{x}) u(\mathbf{x}) - \mathbf{B} u(\mathbf{x}),
\\
\mathbf{B}u(\mathbf{x})  &= \intop_{\Omega}\widetilde a(\mathbf{x} -\mathbf{y})\mu(\mathbf{x},\mathbf{y})u(\mathbf{y})\,d\mathbf{y},\ \ u\in L_{2}(\Omega),
\end{aligned}
\end{equation*}
where
$$
p(\mathbf{x}) = \intop_\Omega \widetilde a(\mathbf{x}-\mathbf{y})\mu(\mathbf{x}, \mathbf{y})\,d\mathbf{y},
$$
see \eqref{e1.9}. (Previously, the operator $\mathbf B$ was denoted by $\mathbb{B}(\mathbf{0})$.)
Recall that
\begin{equation}
\label{p_est}
  \mu_- \| \widetilde{a}\|_{L_1(\Omega)} \leqslant p(\mathbf{x}) \leqslant \mu_+ \| \widetilde{a}\|_{L_1(\Omega)},\quad
  \mathbf{x} \in \Omega.
\end{equation}

For each $N \in \mathbb{N}$ we define a bounded selfadjoint operator $\mathbf{A}_N$ in
$L_2(\Omega)$ as follows:
\begin{equation}
\begin{aligned}
\label{AN}
\mathbf{A}_N u (\mathbf{x}) &:= p_N(\mathbf{x}) u(\mathbf{x}) - \mathbf{B}_N u(\mathbf{x}),
\\
\mathbf{B}_Nu(\mathbf{x})  &:= \intop_{\Omega}\widetilde a_N(\mathbf{x} -\mathbf{y})\mu(\mathbf{x},\mathbf{y})u(\mathbf{y})\,d\mathbf{y},\ \ u\in L_{2}(\Omega),
\end{aligned}
\end{equation}
where
\begin{equation}
\label{5.6}
p_N(\mathbf{x}) = \intop_\Omega \widetilde a_N(\mathbf{x} - \mathbf{y})\mu(\mathbf{x},\mathbf{y})\,d\mathbf{y}.
\end{equation}
Obviously,
\begin{equation}
\label{5.6a}
0 < \mu_- \| \widetilde{a}_N\|_{L_1(\Omega)} \leqslant p_N(\mathbf{x}) \leqslant \mu_+ \| \widetilde{a}_N\|_{L_1(\Omega)},
\end{equation}
$$
\| p_N - p\|_{L_\infty} \leqslant \mu_+ \| \widetilde{a}_N - \widetilde{a}\|_{L_1(\Omega)}.
$$
By the Schur lemma,
$$
\| \mathbf{B}_N - \mathbf{B}\|_{L_2(\Omega) \to L_2(\Omega)} \leqslant \mu_+ \| \widetilde{a}_N - \widetilde{a}\|_{L_1(\Omega)}.
$$
Hence,
\begin{equation}
\label{5.10}
\| \mathbf{A}_N - \mathbf{A}\|_{L_2(\Omega) \to L_2(\Omega)} \leqslant 2 \mu_+ \| \widetilde{a}_N - \widetilde{a}\|_{L_1(\Omega)}.
\end{equation}
By \eqref{5.4}, it follows that
$\| \mathbf{A}_N - \mathbf{A}\|_{L_2(\Omega) \to L_2(\Omega)} \to 0$ as $N \to \infty$.

\subsection{Approximation for the operator $\mathbf{A}^{-1}_\perp$}
Recall that the kernel $\operatorname{Ker} \mathbf{A}$ is one-dimensional and consists of constants:
$\mathfrak{N} := \operatorname{Ker} \mathbf{A} = {\mathcal L}\{\mathbf{1}_\Omega\}$.
Since  $\mathbf{A}$ is selfadjoint, we have $\operatorname{Ran} \mathbf{A} = \mathfrak{N}^\perp$.
The operator $\mathbf{A}_\perp: \mathfrak{N}^\perp \to \mathfrak{N}^\perp$,
which is a restriction of $\mathbf{A}$ onto the subspace
$$
\mathfrak{N}^\perp = \left\{ u \in L_2(\Omega):\ \int_\Omega u(\mathbf{x})\,d\mathbf{x} =0 \right\},
$$
is correctly defined.
The operator $\mathbf{A}_\perp$ is invertible and, by \eqref{e1.31},
\begin{equation}
\label{5.11}
\| \mathbf{A}_\perp^{-1}\|_{\mathfrak{N}^\perp  \to \mathfrak{N}^\perp} \leqslant  \left(\mu_- \mathcal{C}_\pi(a) \right)^{-1}.
\end{equation}

\begin{lemma}
\label{lem5.1_1}
For any $N \in \mathbb{N}$ we have  $\operatorname{Ker} \mathbf{A}_N =\mathfrak{N}$.
The operator $\mathbf{A}_{N,\perp}: \mathfrak{N}^\perp \to \mathfrak{N}^\perp$, which is a restriction of  $\mathbf{A}_N$ onto the subspace $\mathfrak{N}^\perp$, is correctly defined.
The operator $\mathbf{A}_{N,\perp}$ is invertible and
\begin{equation}
\label{5.12}
\| \mathbf{A}_{N,\perp}^{-1}\|_{\mathfrak{N}^\perp  \to \mathfrak{N}^\perp} \leqslant
\| \mathbf{A}_{1,\perp}^{-1}\|_{\mathfrak{N}^\perp  \to \mathfrak{N}^\perp} \leqslant \left( \mu_- \widetilde{\mathcal C}_{2\pi}(\widetilde{a}_1) \right)^{-1},
\quad N \in \mathbb{N},
\end{equation}
where the constant $\widetilde{\mathcal C}_{2\pi}(\widetilde{a}_1)$ is defined below according to  \eqref{5.15a}, \eqref{5.16}.
\end{lemma}

\begin{proof}
From \eqref{AN} and \eqref{5.6} it is seen that  $\mathbf{1}_\Omega \in \operatorname{Ker} \mathbf{A}_N$.
Let us show that $\operatorname{Ker} \mathbf{A}_N = \mathfrak{N}$.

It is easy to check the following representation:
\begin{equation}
\label{5.13}
(   \mathbf{A}_{N}u,u)_{L_2(\Omega)} = \frac{1}{2}\intop_{\Omega}d\mathbf{x}\intop_{\Omega}\,d\mathbf{y}\,
\widetilde{a}_N(\mathbf{x}- \mathbf{y})\mu(\mathbf{x}, \mathbf{y}) \bigl| u(\mathbf{x}) - u(\mathbf{y}) \bigr|^{2},\ \ u\in L_{2}(\Omega),
\end{equation}
cf. \eqref{e1.11}. Together with  \eqref{5.4a} this yields
\begin{equation}
\label{5.14}
(   \mathbf{A}_{N}u,u)_{L_2(\Omega)}  \leqslant (   \mathbf{A}_{N+1}u,u)_{L_2(\Omega)},\ \ u\in L_{2}(\Omega),
\quad N \in \mathbb{N}.
\end{equation}
Therefore, for our purpose, it suffices to consider the operator $\mathbf{A}_{1}$ (the case $N=1$).

From \eqref{5.13} with $N=1$ and estimates for the function $\mu$ (see \eqref{e1.2}) it follows that
\begin{equation}
\label{5.15}
\mu_-   \mathbf{A}_{1}^0 \leqslant \mathbf{A}_{1}   \leqslant \mu_+   \mathbf{A}_{1}^0,
\end{equation}
where $\mathbf{A}_{1}^0$ is the operator of the form \eqref{AN} with the coefficients $\widetilde{a}_1$ and $\mu = \mu_0 \equiv 1$.
The operator $\mathbf{A}_{1}^0$ is diagonalized by the discrete Fourier transform $\mathcal F$ (cf. Section \ref{sec1.4}): $\mathbf{A}_{1}^0$ is unitarily equivalent to the operator of multiplication by the symbol
$$
\widehat{A}_1 (2\pi \mathbf{n}) := ({\mathcal F} \widetilde{a}_1) (\mathbf{0}) - ({\mathcal F}\widetilde{a}_1) (\mathbf{n})
= \int_\Omega \widetilde{a}_1(\mathbf{z}) (1 - \cos \langle \mathbf{z}, 2\pi\mathbf{n} \rangle)\,d\mathbf{z},\quad
 \mathbf{n} \in \mathbb{Z}^d,
$$
in the space  $\ell_2 (\mathbb{Z}^d)$. Similarly to the arguments from Section \ref{sec1.4}, consider the value
\begin{equation}
\label{5.15a}
\widehat{A}_1 (\mathbf{y}) = \int_\Omega \widetilde{a}_1(\mathbf{z})
(1 - \cos \langle \mathbf{z}, \mathbf{y} \rangle)\,d\mathbf{z} =
\frac{1}{2} \int_\Omega \widetilde{a}_1(\mathbf{z})  \sin^2 \left(\frac{\langle \mathbf{z}, \mathbf{y}   \rangle}{2} \right)\,d\mathbf{z}.
\end{equation}
By \eqref{5.5}, the following is true: the function $\widehat{A}_1 (\mathbf{y})$ is continuous, converges to  $\| \widetilde{a}_1\|_{L_1(\Omega)} >0$ as  $| \mathbf{y}| \to \infty$, and  $\widehat{A}_1 (\mathbf{y}) >0$ for $\mathbf{y} \ne \mathbf{0}$.
Hence,
\begin{equation}
\label{5.16}
\widetilde{\mathcal C}_{2\pi}(\widetilde{a}_1) := \min_{| \mathbf{y} | \geqslant 2\pi} \widehat{A}_1 (\mathbf{y}) >0.
\end{equation}
Since $\widehat{A}_1 (2\pi \mathbf{n}) \geqslant \widetilde{\mathcal C}_{2\pi}(\widetilde{a}_1)$ for ${\mathbf 0} \ne
\mathbf{n} \in \mathbb{Z}^d$,
$$
( {\mathbf A}_1^0 u,u)_{L_2(\Omega)} \geqslant \widetilde{\mathcal C}_{2\pi}(\widetilde{a}_1) \|u\|^2_{L_2(\Omega)},
\quad u \in \mathfrak{N}^\perp.
$$
Together with  \eqref{5.15} this yields
$$
( {\mathbf A}_1 u,u)_{L_2(\Omega)} \geqslant \mu_- \widetilde{\mathcal C}_{2\pi}(\widetilde{a}_1) \|u\|^2_{L_2(\Omega)},
\quad u \in \mathfrak{N}^\perp.
$$
Combining this with  \eqref{5.14} and recalling that   $\mathbf{1}_\Omega \in \operatorname{Ker} \mathbf{A}_N$,
we obtain    $\operatorname{Ker} \mathbf{A}_N = \mathfrak{N}$ for any $N \in \mathbb{N}$.
 Then the operator
$\mathbf{A}_{N,\perp}: \mathfrak{N}^\perp \to \mathfrak{N}^\perp$,
which is a restriction of  $\mathbf{A}_N$ onto the subspace $\mathfrak{N}^\perp$, is correctly defined and estimate
 \eqref{5.12} holds.
\end{proof}

\begin{lemma}
\label{lem5.1_2}
We have
\begin{equation*}
\label{5.17}
\| \mathbf{A}_{N,\perp}^{-1} -\mathbf{A}_{\perp}^{-1} \|_{\mathfrak{N}^\perp  \to \mathfrak{N}^\perp} \to 0 \quad \text{as}\ N \to \infty.
\end{equation*}
\end{lemma}

\begin{proof}
Write down the resolvent identity
$$
\mathbf{A}_{N,\perp}^{-1} -\mathbf{A}_{\perp}^{-1} =
\mathbf{A}_{N,\perp}^{-1} \left( \mathbf{A} - \mathbf{A}_{N} \right) \mathbf{A}_{\perp}^{-1}.
$$
Together with estimates  \eqref{5.10}--\eqref{5.12} this implies
\begin{equation*}
\label{5.18}
\| \mathbf{A}_{N,\perp}^{-1} -\mathbf{A}_{\perp}^{-1} \|_{\mathfrak{N}^\perp  \to \mathfrak{N}^\perp} \leqslant
 2 \mu_+ \left( \mu_-^2 \mathcal{C}_\pi(a) \widetilde{\mathcal C}_{2\pi}(\widetilde{a}_1) \right)^{-1}
  \| \widetilde{a}_N - \widetilde{a}\|_{L_1(\Omega)}.
\end{equation*}
It remains to take  \eqref{5.4} into account.
\end{proof}

\subsection{Uniform estimate for the norm of the solution $v_N = \mathbf{A}_{N,\perp}^{-1} w$ in $L_\infty$}
\label{Sec5.3}
In the space  $L_\infty(\Omega)$, consider a subspace  of functions with zero mean:
$$
L_\infty^\perp(\Omega) := \left\{ u \in L_\infty(\Omega):\ \int_\Omega u(\mathbf{x})\,d\mathbf{x} =0   \right\}.
$$
Let $w \in L_\infty^\perp(\Omega)$ and $v_N := \mathbf{A}_{N,\perp}^{-1} w$, i.\,e., $v_N \in L_2(\Omega)$
is a solution of the problem
\begin{equation}
\label{5.19}
\intop_{\Omega}  \widetilde{a}_N (\mathbf{x} - \mathbf{y}) \mu(\mathbf{x},\mathbf{y}) (v_N(\mathbf{x}) - v_N(\mathbf{y})) \, d\mathbf{y} =w(\mathbf{x}),
 \quad \mathbf{x} \in \Omega;
\quad \intop_\Omega v_N(\mathbf{x})\,d\mathbf{x} =0.
\end{equation}
Taking \eqref{5.6} and \eqref{5.6a} into account, rewrite problem \eqref{5.19} as
\begin{equation}
\label{5.20}
v_N(\mathbf{x}) = \frac{w(\mathbf{x})}{p_N(\mathbf{x})} -  \frac{1}{p_N(\mathbf{x})} \intop_{\Omega}  \widetilde{a}_N (\mathbf{x} - \mathbf{y}) \mu(\mathbf{x}, \mathbf{y})  v_N(\mathbf{y}) \, d\mathbf{y}, \quad \mathbf{x} \in \Omega;
\quad \intop_\Omega v_N(\mathbf{x})\,d\mathbf{x} =0.
\end{equation}
Combining this with  \eqref{5.3}, \eqref{5.6a}, and Lemma \ref{lem5.1_1}, we conclude that the solution  $v_N(\mathbf{x})$ is bounded:
\begin{multline*}
\|v_N \|_{L_\infty(\Omega)} \leqslant \frac{\|w\|_{L_\infty(\Omega)}}{\mu_- \| \widetilde{a}_N\|_{L_1(\Omega)}} +
\frac{N \mu_+ \|v_N\|_{L_2(\Omega)}}{\mu_- \| \widetilde{a}_N\|_{L_1(\Omega)}}
\\
\leqslant
\frac{\|w\|_{L_\infty(\Omega)}}{\mu_- \| \widetilde{a}_N\|_{L_1(\Omega)}} +
\frac{N \mu_+ \| w\|_{L_2(\Omega)}}{\mu_-^2   \widetilde{\mathcal C}_{2\pi}(\widetilde{a}_1) \| \widetilde{a}_N\|_{L_1(\Omega)}}
\leqslant {\mathfrak C}_N \|w\|_{L_\infty(\Omega)},
\\
{\mathfrak C}_N =
\left(\mu_- \| \widetilde{a}_N\|_{L_1(\Omega)} \right)^{-1} + N \mu_+ \left(\mu_-^2   \widetilde{\mathcal C}_{2\pi}(\widetilde{a}_1) \| \widetilde{a}_N\|_{L_1(\Omega)} \right)^{-1}.
\end{multline*}
This means that the operator $\mathbf{A}_{N,\perp}^{-1}$ maps $L_\infty^\perp(\Omega)$ into itself and
$$
\| \mathbf{A}_{N,\perp}^{-1} \|_{L_\infty^\perp(\Omega) \to L_\infty^\perp(\Omega)} \leqslant {\mathfrak C}_N.
$$

Our nearest goal is to obtain a uniform estimate for the norm of this operator.

\begin{lemma}
\label{lem5.3}
We have
$$
\| \mathbf{A}_{N,\perp}^{-1} \|_{L_\infty^\perp(\Omega) \to L_\infty^\perp(\Omega)} \leqslant {\mathfrak C}, \quad N \in \mathbb{N}.
$$
The constant  ${\mathfrak C} = {\mathfrak C} (\widetilde{a},\mu)$ is defined below in  \eqref{5.24}.
\end{lemma}

\begin{proof}
As above, let $w \in L_\infty^\perp(\Omega)$ and $v_N := \mathbf{A}_{N,\perp}^{-1} w$.
Let us estimate the norm of  $v_N$ in $L_\infty(\Omega)$; it suffices to assume that  $\|w\|_{L_\infty(\Omega)} >0$.
From \eqref{5.20}, taking  \eqref{5.5} and \eqref{5.6a} into account, we deduce that
\begin{equation}
\label{5.21}
| v_N(\mathbf{x})| \leqslant \frac{\|w\|_{L_\infty(\Omega)}}{\mu_- \| \widetilde{a}_1\|_{L_1(\Omega)}} +
\frac{1}{\mu_- \| \widetilde{a}_1\|_{L_1(\Omega)}}
\intop_{\Omega}  \widetilde{a}_N (\mathbf{x} - \mathbf{y}) \mu(\mathbf{x}, \mathbf{y}) | v_N(\mathbf{y})| \, d\mathbf{y}, \quad \mathbf{x} \in \Omega.
\end{equation}
Represent the integral on the right-hand side as a sum of two integrals, over the sets $\Omega_{N,1}$ and $\Omega_{N,2} := \Omega \setminus \Omega_{N,1}$. Here
$$
\Omega_{N,1} := \left\{ \mathbf{x} \in \Omega:\ |v_N(\mathbf{x})| > \| w\|_{L_\infty(\Omega)}^{1/2} \|v_N \|^{1/2}_{L_\infty(\Omega)}
 \right\}.
$$

For $\mathbf{x} \in \Omega_{N,2}$ we have
$|v_N(\mathbf{x})| \leqslant \| w\|_{L_\infty(\Omega)}^{1/2} \|v_N \|^{1/2}_{L_\infty(\Omega)}$, whence
\begin{equation}
\label{5.21_0}
\intop_{\Omega_{N,2}}  \widetilde{a}_N (\mathbf{x} -\mathbf{y}) \mu(\mathbf{x},\mathbf{y})  |v_N(\mathbf{y})| \, d\mathbf{y}
\leqslant \mu_+ \| w\|_{L_\infty(\Omega)}^{1/2} \|v_N \|^{1/2}_{L_\infty(\Omega)} \| \widetilde{a}\|_{L_1(\Omega)}.
\end{equation}
We have used the obvious inequality
$\widetilde{a}_N(\mathbf{z}) \leqslant \widetilde{a}(\mathbf{z})$.

From the definition of  $\Omega_{N,1}$ and Lemma \ref{lem5.1_1} it follows that
\begin{multline}
\label{5.21a}
\begin{aligned}
\operatorname{mes} \Omega_{N,1} &\leqslant \intop_{ \Omega_{N,1}} \frac{|v_N(\mathbf{x} )|}{\| w\|_{L_\infty(\Omega)}^{1/2} \|v_N \|^{1/2}_{L_\infty(\Omega)}}\,d\mathbf{x}
 \leqslant \frac{\|v_N\|_{L_2(\Omega)}}{\| w\|_{L_\infty(\Omega)}^{1/2} \|v_N \|^{1/2}_{L_\infty(\Omega)}}
 \\
& \leqslant  \frac{\|w\|_{L_2(\Omega)}}{\mu_-  \widetilde{\mathcal C}_{2\pi}(\widetilde{a}_1) \| w\|_{L_\infty(\Omega)}^{1/2} \|v_N \|^{1/2}_{L_\infty(\Omega)}} \leqslant  \frac{\|w\|^{1/2}_{L_\infty(\Omega)}}{\mu_-  \widetilde{\mathcal C}_{2\pi}(\widetilde{a}_1) \|v_N \|^{1/2}_{L_\infty(\Omega)}}.
 \end{aligned}
\end{multline}
Since $\widetilde{a}_N(\mathbf{z}) \leqslant \widetilde{a}(\mathbf{z})$, we have
\begin{equation}
\label{5.22}
\begin{aligned}
\intop_{\Omega_{N,1}}  \widetilde{a}_N (\mathbf{x} - \mathbf{y}) \mu(\mathbf{x}, \mathbf{y})  |v_N(\mathbf{y})| \, d\mathbf{y}
&\leqslant \mu_+ \| v_N\|_{L_\infty(\Omega)} \intop_{\Omega_{N,1}}  \widetilde{a}_N (\mathbf{x} - \mathbf{y}) \, d\mathbf{y}
\\
&\leqslant \mu_+ \| v_N\|_{L_\infty(\Omega)} \intop_{\Omega_{N,1}}  \widetilde{a} (\mathbf{x} - \mathbf{y}) \, d\mathbf{y}.
\end{aligned}
\end{equation}
Denote
\begin{equation}
\label{5.22}
F_{\widetilde{a}}(t) := \sup \left\{  \int_{\mathcal O} \ \widetilde{a}(\mathbf{z})\,d \mathbf{z}: \ {\mathcal O} \subset [-2,2]^d,\ \operatorname{mes} {\mathcal O} \leqslant t \right\},\quad t>0.
\end{equation}
Obviously, the function  $F_{\widetilde{a}}(t)$ is monotonously nondecreasing  and $F_{\widetilde{a}}(t) \to 0$ as \hbox{$t\to +0$}.
From \eqref{5.21a}--\eqref{5.22} it follows that
$$
\intop_{\Omega_{N,1}}  \widetilde{a}_N (\mathbf{x} - \mathbf{y}) \mu(\mathbf{x},\mathbf{y})  |v_N( \mathbf{y})| \, d\mathbf{y}
\leqslant \mu_+ \| v_N\|_{L_\infty(\Omega)} F_{\widetilde{a}} \left(  \bigl(\mu_-  \widetilde{\mathcal C}_{2\pi}(\widetilde{a}_1)\bigr)^{-1}
\|w\|^{1/2}_{L_\infty(\Omega)}  \|v_N \|^{-1/2}_{L_\infty(\Omega)} \right).
$$
Together with \eqref{5.21} and \eqref{5.21_0} this yields
\begin{equation}
\label{5.23}
\begin{aligned}
\| v_N \|_{L_\infty(\Omega)} \leqslant & \frac{1}{\mu_- \| \widetilde{a}_1\|_{L_1(\Omega)}}
\left( \|w\|_{L_\infty(\Omega)} +
\mu_+ \| w\|_{L_\infty(\Omega)}^{1/2} \|v_N \|^{1/2}_{L_\infty(\Omega)} \| \widetilde{a}\|_{L_1(\Omega)} \right.
\\
& \left.+
\mu_+ \| v_N\|_{L_\infty(\Omega)} F_{\widetilde{a}} \left(  \bigl(\mu_-  \widetilde{\mathcal C}_{2\pi}(\widetilde{a}_1)\bigr)^{-1}
\|w\|^{1/2}_{L_\infty(\Omega)}  \|v_N \|^{-1/2}_{L_\infty(\Omega)} \right) \right).
\end{aligned}
\end{equation}
Fix a number $t_0 = t_0(\widetilde{a}, \mu) >0$ such that
$$
\frac{\mu_+ F_{\widetilde{a}}(t_0)}{\mu_- \| \widetilde{a}_1\|_{L_1(\Omega)}} \leqslant \frac{1}{2}.
$$
If  $ \bigl(\mu_-  \widetilde{\mathcal C}_{2\pi}( \widetilde{a}_1)\bigr)^{-1}
\|w\|^{1/2}_{L_\infty(\Omega)}  \|v_N \|^{-1/2}_{L_\infty(\Omega)} > t_0$, then  we obviously have
$$
\| v_N \|_{L_\infty(\Omega)} \leqslant \frac{\|w\|_{L_\infty(\Omega)}}{\bigl(t_0  \mu_-  \widetilde{\mathcal C}_{2\pi}(\widetilde{a}_1)\bigr)^2}.
$$
In this case  
 it follows from \eqref{5.23} that
$$
\frac{1}{2}\| v_N \|_{L_\infty(\Omega)} \leqslant  \frac{1}{\mu_- \| \widetilde{a}_1\|_{L_1(\Omega)}}
\left( \|w\|_{L_\infty(\Omega)} +
\mu_+ \| w\|_{L_\infty(\Omega)}^{1/2} \|v_N \|^{1/2}_{L_\infty(\Omega)} \| \widetilde{a}\|_{L_1(\Omega)} \right).
$$
This is a quadratic inequality with respect to  $X := \| v_N \|^{1/2}_{L_\infty(\Omega)}$.
Rewrite it as
$$
X^2 -  \frac{2 \mu_+\| w\|_{L_\infty(\Omega)}^{1/2} }{\mu_-} X \leqslant  \frac{2 \| w\|_{L_\infty(\Omega)}}{\mu_- \| \widetilde{a}_1\|_{L_1(\Omega)}}.
$$
Solving this inequality, we obtain
$$
 \| v_N \|_{L_\infty(\Omega)} \leqslant
 \left( \frac{ \mu_+}{\mu_-}
 + \left( \frac{2 }{\mu_- \| \widetilde{a}_1\|_{L_1(\Omega)}} +
 \frac{\mu_+^2 }{\mu_-^2} \right)^{1/2}\right)^2
\| w\|_{L_\infty(\Omega)}.
$$
As a result, we arrive at the estimate
\begin{equation}
\label{5.24}
\begin{aligned}
\| v_N \|_{L_\infty(\Omega)} &\leqslant {\mathfrak C} \| w\|_{L_\infty(\Omega)},
\\
 {\mathfrak C} = {\mathfrak C} (\widetilde{a},\mu) &:= \max \left\{  \frac{1}{\bigl(t_0  \mu_-  \widetilde{\mathcal C}_{2\pi}(\widetilde{a}_1)\bigr)^2},
 \left( \frac{ \mu_+}{\mu_-}
 + \left( \frac{2 }{\mu_- \| \widetilde{a}_1\|_{L_1(\Omega)}} +
 \frac{\mu_+^2 }{\mu_-^2} \right)^{1/2}\right)^2
 \right\}.
 \end{aligned}
\end{equation}
\end{proof}

\subsection{Boundedness of the solution  $v = \mathbf{A}_{\perp}^{-1} w$}
As above, let  $w \in L_\infty^\perp(\Omega)$ and $v := \mathbf{A}_{\perp}^{-1} w$, i.\,e., $v \in L_2(\Omega)$
is the solution of the problem
\begin{equation}
\label{5.25}
\intop_{\Omega}  \widetilde{a} (\mathbf{x} - \mathbf{y}) \mu(\mathbf{x}, \mathbf{y}) (v(\mathbf{x}) - v(\mathbf{y})) \, d\mathbf{y} =w(\mathbf{x}), \quad  \mathbf{x} \in \Omega;
\quad \intop_\Omega v(\mathbf{x})\,d\mathbf{x} =0.
\end{equation}
As in Section \ref{Sec5.3}, let $v_N := \mathbf{A}_{N,\perp}^{-1} w$ be the solution of problem \eqref{5.19}.
By Lemma \ref{lem5.1_2},
$$
\| v_N - v \|_{L_2(\Omega)} \to 0 \quad \text{as}\ N \to \infty.
$$
Then, by the Riesz theorem, there exists a subsequence $N_k \to \infty$ such that
$$
v_{N_k}(\mathbf{x}) \to v(\mathbf{x}), \quad k \to \infty, \ \text{for almost every}\  \mathbf{x} \in \Omega.
$$
By Lemma \ref{lem5.3},
$$
|v_{N_k}(\mathbf{x})| \leqslant {\mathfrak C} \| w\|_{L_\infty(\Omega)}\ \text{for almost every}\  \mathbf{x} \in \Omega.
$$
Hence,
$$
|v(\mathbf{x})| \leqslant {\mathfrak C} \| w\|_{L_\infty(\Omega)}\ \text{for almost every}\  \mathbf{x} \in \Omega.
$$
We arrive at the following result.

\begin{theorem}
\label{th5.4}
The operator $\mathbf{A}_{\perp}^{-1}$ maps $L_\infty^\perp(\Omega)$ into itself and
$$
\| \mathbf{A}_{\perp}^{-1} \|_{L_\infty^\perp(\Omega) \to L_\infty^\perp(\Omega)} \leqslant {\mathfrak C}.
$$
The constant  ${\mathfrak C} = {\mathfrak C} (\widetilde{a},\mu)$ is defined by  \eqref{5.24}.
\end{theorem}

\begin{corollary}\label{cor5.5}
Suppose that conditions \eqref{e1.1}--\eqref{e1.3} are satisfied and $M_1(a) < \infty$. Then the solutions  $v_j(\mathbf{x})$, $j=1,\dots,d,$ of problems \eqref{e2.22a} are bounded. We have
\begin{equation}
\label{5.26}
\| v_j \|_{L_\infty(\Omega)} \leqslant  \mu_+ M_1(a) {\mathfrak C}(\widetilde{a}, \mu),\quad j=1,\dots,d.
\end{equation}
\end{corollary}

\begin{proof}
Recall that  $v_j = \mathbf{A}_{\perp}^{-1} w_j$, where the function  $w_j(\mathbf{x})$ is defined by  \eqref{e2.21}. It is easily seen that under conditions \eqref{e1.1}, \eqref{e1.2} and $M_1(a) < \infty$
we have  $w_j \in L_\infty^\perp(\Omega)$ and
$$
\| w_j \|_{L_\infty(\Omega)} \leqslant \mu_+ M_1(a).
$$
Combining this with Theorem \ref{th5.4}, we obtain   \eqref{5.26}.
\end{proof}

\begin{remark}
\label{rem5.6}
Under the additional assumption $\widetilde{a} \in L_2(\Omega)$ an estimate for the solution of problem \eqref{5.25} can be derived directly from the equation
$$
v(\mathbf{x}) = \frac{w(\mathbf{x})}{p(\mathbf{x})} -  \frac{1}{p(\mathbf{x})} \intop_{\Omega}  \widetilde{a} (\mathbf{x} - \mathbf{y}) \mu(\mathbf{x}, \mathbf{y})  v(\mathbf{y}) \, d\mathbf{y}
$$
by using  \eqref{p_est}, \eqref{5.11}, and the  Schwarz inequality\emph{:}
\begin{multline*}
\|v\|_{L_\infty(\Omega)} \leqslant \frac{\|w\|_{L_\infty}}{\mu_- \| \widetilde{a}\|_{L_1(\Omega)}}
+ \frac{\mu_+ \| \widetilde{a}\|_{L_2(\Omega)} \|v\|_{L_2(\Omega)}}{\mu_- \| \widetilde{a}\|_{L_1(\Omega)}}
\\
\leqslant \left( \frac{1}{\mu_- \| \widetilde{a}\|_{L_1(\Omega)}} +
\frac{ \mu_+ \| \widetilde{a}\|_{L_2(\Omega)} }{\mu^2_- {\mathcal C}_\pi(a)\| \widetilde{a}\|_{L_1(\Omega)}} \right)  \|w\|_{L_\infty(\Omega)}.
\end{multline*}
Then the solutions $v_j,$ $j=1,\dots,d,$ of problems  \eqref{e2.22a} satisfy the estimates
\begin{equation*}
\| v_j \|_{L_\infty(\Omega)} \leqslant  \mu_+ M_1(a)\left( \frac{1}{\mu_- \| \widetilde{a}\|_{L_1(\Omega)}} +
\frac{ \mu_+ \| \widetilde{a}\|_{L_2(\Omega)} }{\mu^2_- {\mathcal C}_\pi(a)\| \widetilde{a}\|_{L_1(\Omega)}} \right),
\quad j=1,\dots,d.
\end{equation*}
The constant on the right depends only on   $\mu_-,$ $\mu_+,$ $\| \widetilde{a}\|_{L_1(\Omega)},$
$\| \widetilde{a}\|_{L_2(\Omega)},$ ${\mathcal C}_\pi(a),$ $M_1(a)$.
\end{remark}

\section*{Funding}
The research of A.~Piatnitski and E.~Zhizhina was partially supported by the project ``Pure Mathematics in Norway'' and the UiT Aurora project MASCOT.
The research of A.~Piatnitski was supported by the megagrant of Ministry of Science and Higher Education of the Russian Federation, project 075-15-2022-1115.

The research of V.~Sloushch and T.~Suslina was supported by~Russian Science Foundation, project no.~22-11-00092.


\begin{thebibliography}{99}


\bibitem{BaPa} N.~S.~Bakhvalov,  G.~P.~Panasenko, \emph{Homogenisation: Averaging processes in periodic media}. Mathematics and its Applications 36, Springer, Dordrecht, 1989.

\bibitem{BeLP} A.~Bensoussan, J.-L.~Lions, G.~Papanicolaou, \emph{Asymptotic analysis for periodic structures}. Stud. Math. Appl., vol. 5, North-Holland Publishing Co., Amsterdam--New York, 1978.

\bibitem{BSu1}
M.~Sh. Birman, T.~A. Suslina, 
Second order periodic differential operators. Threshold properties and homogenization.
 \emph{Algebra i Analiz} \textbf{15} (2003), no.~5, 1--108 (in Russian); English transl., \emph{St.~Petersburg Math. J.} \textbf{15} (2004), no.~5, 639--714.

\bibitem{BSu3} 
M.~Sh.~Birman, T.~A.~Suslina, Homogenization  with corrector term for periodic  elliptic differential operators. \emph{Algebra i Analiz} \textbf{17} (2005), no.~6, 1--104 (in Russian); English transl., \emph{St.~Petersburg Math. J.} \textbf{17} (2006), no.~6, 897--973.

\bibitem{BSu4} 
M.~Sh.~Birman, T.~A.~Suslina, Homogenization with corrector for periodic  differential operators. Approximation of solutions in the Sobolev class $H^1(\mathbb{R}^d)$. \emph{Algebra i Analiz} \textbf{18} (2006), no. 6, 1--130 (in Russian); English transl.,  St.~Petersburg Math. J. \textbf{18} (2007), no. 6, 857--955. 

\bibitem{BraPia21} A.~Braides, A.~Piatnitski, Homogenization of quadratic convolution energies in periodically perforated domains. \emph{Adv. Calc. Var.} (2019),  https://doi.org/10.1515/acv-2019-0083.


\bibitem{KPMZh} Yu.~Kondratiev, S.~Molchanov, S.~Pirogov, E.~Zhizhina, {On ground state of some non local Schr\"{o}dinger operators}. \emph{Appl. Anal.} \textbf{96} (2017), no.~8, 1390--1400.

\bibitem{PZh} A.~Piatnitski, E.~Zhizhina, Periodic homogenization of nonlocal operators with a convolution-type kernel. \emph{SIAM J. Math. Anal.} \textbf{49} (2017), no.~1, 64--81.

\bibitem{PiaZhi19} A.~Piatnitski, E.~Zhizhina, Homogenization of biased convolution type operators. \emph{Asymptotic Anal.} {\bf 115} (2019), no. 3-4, 241--262.

\bibitem{PSlSuZh} A.~Piatnitski, V.~Sloushch, T.~Suslina, E.~Zhizhina, On operator estimates in homogenization of nonlocal operators of convolution type. \emph{J. Diff. Equ.} {\bf 352} (2023), 153--188.

\bibitem{Sk} M.~M.~Skriganov,
Geometric and arithmetic  methods in spectral theory of multidimensional periodic operators. \emph{Tr. MIAN SSSR}~\textbf{171} (1985), 3--122 (in Russian);
English transl., \emph{Proc. Steklov Inst. Math.} \textbf{171} (1987), 1--121.

\bibitem {Su_UMN2023} T.~A.~Suslina, Operator-theoretic approach to homogenization 
	of the  Schr\"{o}dinger-type equations with periodic coefficients. 
	 \emph{Uspekhi Matem. Nauk} \textbf{78}, no. 6 (2023) (in Russian); English transl., \emph{Russian Math. Surveys} \textbf{78} (2023), no. 6.

\bibitem{Zh} V.~V.~Zhikov, On operator estimates in homogenization theory. \emph{Dokl.~Akad.~Nauk} \textbf{403} (2005), no. 3, 305--308 (in Russian); English transl., \emph{Dokl. Math.} \textbf{72} (2005), no. 1, 534--538.

\bibitem{JKO} V.~V.~Zhikov, S.~M.~Kozlov, O.~A.~Oleinik,
\emph{Homogenization of differential operators}.  Springer-Verlag, Berlin, 1994.


\bibitem{ZhPas1} V.~V.~Zhikov, S.~E.~Pastukhova, On operator estimates for some problems in homogenization theory. \emph{Russ. J. Math. Phys.}~\textbf{12} (2005), no.~4, 515--524.

\bibitem{ZhPas3} 
V.~V.~Zhikov, S.~E.~Pastukhova, Operator estimates in homogenization theory. \emph{Uspekhi Matem. Nauk} \textbf{71} (2016), no. 3, 27--122 (in Russian); English transl., \emph{Russian Math. Surveys} \textbf{71} (2016), no. 3, 417--511. 





\end{thebibliography}
\end{document}